\documentclass[12pt,psamsfonts]{amsart}
\usepackage{amsmath}
\usepackage{amsthm}
\usepackage{amssymb}
\usepackage{amscd}
\usepackage{amsfonts}
\usepackage{amsbsy}
\usepackage{epsfig,afterpage}
\usepackage{psfrag}
\usepackage{graphicx}
\usepackage{indentfirst, latexsym, bm,amssymb}
\usepackage{bbding}

\newtheorem {theorem} {Theorem}
\newtheorem {proposition} [theorem]{Proposition}
\newtheorem {corollary} [theorem]{Corollary}

\newtheorem {example} [theorem]{Example}
\newtheorem {remark} [theorem]{Remark}

\newtheorem {problem} [theorem]{Problem}
\newcommand{\R}{\mathbb{R}}

\newcommand{\X}{{\mathcal X}}

\begin{document}

\title[An inverse approach to the center-foci problem]
{An inverse approach to the center-foci problem}

\date{}
\dedicatory{}

\author[ Rafael Ram\'{\i}rez and Valent\'{\i}n Ram\'{\i}rez]
{ Rafael Ram\'{\i}rez$^1$, Valent\'{\i}n Ram\'{\i}rez$^2$}

\address{$^1$ Departament d'Enginyeria Inform\`{a}tica i Matem\`{a}tiques,
Universitat Rovira i Virgili, Avinguda dels Pa\"{\i}sos Catalans 26,
43007 Tarragona, Catalonia, Spain.}
\email{rafaelorlando.ramirez@urv.cat}

\address{$^2$ Universitat Central de Barcelona, Gran V\'{\i}a de las
Cortes Catalanas, 585 08007 Barcelona, Spain.}
\email{vramirsa8@alumnes.ub.edu}



 \maketitle

\begin{abstract}
The classical Center-Focus Problem posed by H. Poincar\'e in 1880's
is concerned on the characterization of planar polynomial vector
fields $\X=(-y+X(x,y))\dfrac{\partial }{\partial
x}+(x+Y(x,y))\dfrac{\partial }{\partial y},$ with $X(0,0)=Y(0,0)=0,$
such that all their integral trajectories are closed curves whose
interiors contain a fixed point called center or such that all their
integral trajectories are spirals called foci. In this paper we
state and study the inverse problem to the Center-Foci Problem i.e.,
we require to determine the analytic planar vector fields $\X$ in
such a way that for a given Liapunov function
\[V=V(x,y)=\dfrac{\lambda}{2}(x^2+y^2)+\displaystyle\sum_{j=3}^{\infty}
H_j(x,y),\] where $H_j=H_j(x,y)$ are homogenous polynomial of degree
$j,$ the following equation holds
\[\X(V)=\displaystyle\sum_{j=3}^{\infty}V_j(x^2+y^2)^{j+1},
\]
where $V_j$ for $j\in\mathbb{N}$ are the Liapunov constants. In
particular we study the case when the origin is a center and the
vector field is polynomial.

\end{abstract}

\section{Introduction }
The  {\it Center-Focus problem} first was studied by Poincar\'e
\cite{Po} and further developed by Lyapunov \cite{Liapunov},
Bendixson \cite{Bendix} and Frommer \cite{From}.

\smallskip

 Consider the set $\Sigma$ of all planar real polynomial vector fields
$$\X=P\dfrac{\partial}{\partial x}+Q\dfrac{\partial}{\partial y},$$
associated to the differential polynomial systems
\begin{equation}\label{1}
\dot{x}=P(x,y),\quad \dot{y}=Q(x,y).
\end{equation}
 here the dot denotes derivative respect to the time $t,$ where $P=P(x,y)$ and $Q=Q(x,y)$ are
real coprime in $\mathbb{R}[x,y]$ polynomials of degree
$n=\max\left\{\mbox{deg{P}},\,\mbox{deg{Q}}\right\},$ in the
variables $x$ and $y$, or more generally, real analytic functions
defined in an open neighborhood of the origin $\textsc{O}:=(0,0)$.
One assumes that origin is a singular (equilibrium, fixed) point,
i.e. $P(0,0)=Q(0,0)=0.$

\smallskip

 The equilibrium point is called a
center if there exists an open neighborhood $U$ of $\textsc{O}$ that
does not contain another equilibrium points such that any trajectory
of \eqref{1}  that intersects $U \setminus {\textsc{O}}$ in some
point is closed. The problem of the center can be formulated as
follows: Determine within the class of real planar polynomial vector
fields $\X,$ all the systems possessing a center at the origin, i.e.
the singular point surrounded by close phase curves.

\smallskip

 Suppose That $\X$ is an analytic planar vector field associated to
differential system
\begin{equation}\label{2}
\begin{array}{rl}
\dot{x}= & ax + by +\displaystyle\sum_{k=2}^\infty X_k(x,y),\vspace{0.2cm}\\
 \dot{x} =& cx + dy + \displaystyle\sum_{k=2}^\infty
Y_k(x,y),
\end{array}
\end{equation}
where $X_k$, $Y_k$ are real homogeneous polynomials of degree $ k$.
In the aforementioned papers the case of a non-degenerated
equilibrium point was studied, i.e. the matrix

\[A=\left(
    \begin{array}{cc}
      a & b \\
      c & d \\
    \end{array}
  \right)
\] assumed to be invertible. It was proved by Poincar\'e that a
necessary condition for $\textsc{O}$ to be a center is that $A$ has
pure imaginary eigenvalues. In this case, making a  linear change of
variables  and then a linear re-parametrization of trajectories one
reduces \eqref{2} to an equivalent system:
\begin{equation}\label{3}
\dot{x}=-\lambda y+X(x,y),\quad \dot{y}=\lambda x+Y(x,y)
\end{equation}
where $\lambda$ is a constant, $X$ and $Y$ are real analytic
functions in an open neighborhood of  $\textsc{O}$ whose Taylor
expansions at $\textsc{O}$ do not contain constant and linear terms.
For $X,\,Y$ polynomials of a given degree, the classical Poincar\'e
Center-Focus Problem asks about conditions on the coefficients of
$X$ and $Y$ under which all trajectories of \eqref{3}  situated in a
small open neighborhood of the origin are closed. Poincar\'e's
theorem says that a necessary and sufficient condition to have a
center at a singular point at the origin with pure imaginary
eigenvalues is that there exists a local analytic non-constant first
integral in the neighborhood of $\textsc{O}$.  The search for an
analytic first integral led Poincar\'e to an algorithm for computing
the so-called Poincar´e--Lyapunov quantities $V_k$ ({\it Poincar\'e
criterions}) , for $ k\in\mathbb{N}$ of the singularity, which are
polynomials over $\mathbb{Q}$ in the coefficients of the system. A
necessary and sufficient condition to have a center is then the
annihilation of all these quantities. In view of Hilbert's basis
theorem for this to occur it suffices to have for a finite number of
$k$, $k < j$ and $j$ sufficiently large, $V_k = 0.$ Unfortunately,
trying to apply Poincar\'e criterions in the converse direction ,
gives rise, in the general case,to almost insurmountable
difficulties. Therefore, as it was emphasized by Poincar\'e, it is
the matter of great importance to find typical situations for which
all equations of the system determining center are satisfied. As an
example, Poincar\'e described a case related to certain symmetries
for $X$ and $Y$ in \eqref{3}.

 Although we have an algorithm for computing the Poincar\'e-Lyapunov constants for
singularities with pure imaginary eigenvalues, we have no algorithm
to determine how many of them need to be zero to imply that all of
them are zero for cubic or higher degree polynomial differential
systems.  Bautin showed in 1939 that for a quadratic polynomial
system, to annihilate all $V_k$'s it suffices to have $V_k = 0$ for
$i =1,2, 3$.  So the problem of the center is solved for $n = 2$.
This problem was solved for the so called quasihomogenous cubic
differential systems (see for instance
\cite{Bautin,Sibirskii,Schlomiuk, Malkin,Saharnikov}).
\[\begin{array}{rl}
\dot{x}=&-y+a_{20}x^2+a_{21}x^2y+a_{12}xy^2+a_{03}y^3,\vspace{0.2cm}\\
 \dot{y}=&x+b_{20}x^2+b_{21}x^2y+b_{12}xy^2+b_{03}y^3.
\end{array}
\] Another necessary and sufficient condition for $\textsc{O}$ to
be a center for   for analytic vector fields  was obtained by
Lyapunov in \cite{Liapunov}.
\begin{theorem}\label{Liacenter2}
Origin is a center for \eqref{3} if and only if there exist a first
integral $1/2{\left(x^2+y^2\right)}+f(x,y)=C,$ where $f=f(x,y)$ is a
real analytic functions in an open neighborhood of $\textsc{O}$
whose Taylor expansions at $\textsc{O}$ do not contain constant,
linear and quadratic terms.
\end{theorem}
 Unfortunately, trying to apply the Liapunov's Theorem gives
rise, in general, to difficulties comparable with those for
Poincare's criterions.

\smallskip

To show that the singular point is a center we have two basic
mechanisms: we either show that the system is symmetric or using
Poincar\'e--Liapunov's Theorems we show that we have a local
analytic first integral.

\smallskip

Darboux gave his geometric method of integration in his seminar work
\cite{Darboux} of 1878.This method has not played a significant role
in the problem of the center until almost the last 20 years of the
20th century when the situation abruptly changed and during the past
years the impact of the geometric work of Darboux has been growing
steadily in research work on polynomial vector fields.  The
geometric method of Darboux uses algebraic invariant curves of a
system to prove integrability and Darboux's theorem affirms that if
we have a sufficient number of such invariant algebraic curves then
the system had a first integral which is analytic on the complement
of the union of these curves. This connection between the problem of
the center and algebraic geometry has come to the forefront in
recent years. There were numerous publications on the problem of the
center during the last part of the 20th century and the beginning of
the 21st century (see for instance \cite{Zoladek, Llibre1}).

\smallskip

\section{Methods to study the stability of the planar non degenerate  differential
system}

\smallskip

For simplicity we shall assume that all the functions which appear
in this section and in the following ones are of class
$\mathcal{C}^\infty,$ although most of the results remain valid
under weaker hypotheses.

\smallskip

 We shall study the stability of the origin of the planar
differential system of the type \eqref{3}
 For this differential
system it is possible to determine the general structure of the
trajectories in the neighborhood of the origin. We shall study three
methods (see for instance \cite{Malkin1}).

\smallskip

{\it Method I}

\smallskip
Passing to  polar coordinates $x=r\cos
\vartheta,\,\,y=r\sin\vartheta,$  differential system \eqref{3}
becomes
\begin{equation}\label{02}
\begin{array}{rl}
\dot{r}=&\cos\vartheta
X(r\cos\vartheta,r\sin\vartheta)+\sin\vartheta
Y(r\cos\vartheta,r\sin\vartheta)\vspace{0.2cm}\\
:=&rR(r,\vartheta),\vspace{0.2cm}\\
\dot{\vartheta}=&\dfrac{1}{r}\left(\cos \vartheta
Y(r\cos\vartheta,r\sin\vartheta)-\sin\vartheta
X(r\cos\vartheta,r\sin\vartheta)\right)+\lambda\vspace{0.2cm}\\
:=&\lambda+\theta(r,\vartheta).
\end{array}
\end{equation}
where $R=R(r,\vartheta)$ and $\theta=\theta(r,\vartheta)$ are real
analytic functions in an open neighborhood of the origin whose
Taylor expansions at $\textsc{O}$ do not contain constant and linear
terms which converges for $r$ small and vanished for $r=0.$ The
coefficients of the Taylor extension  are trigonometric polynomials,
consequently are  $2\pi$ periodical functions on $\vartheta.$ To
obtain the trajectories from \eqref{02} we exclude the variable $t$.
Thus we obtain
\begin{equation}\label{03}
\dfrac{d
r}{d\vartheta}=\dfrac{rR}{\lambda+\theta}=r^2R_2(\vartheta)+r^3R_3(\vartheta)+\ldots.
\end{equation}
The series of the right hand converges for small $r.$ The
coefficients $R_j=R_j(\vartheta)$ are polynomial on the variables
$\cos\vartheta$ and $\sin\vartheta.$ Equation \eqref{03} has the
trivial solution $r=0.$

\smallskip

By considering that the right hand of \eqref{03} is analytic  then
any solution $r=r(\vartheta,c)$ of this equation with the initial
condition $r(\vartheta,c)|_{\vartheta=0}=c,$ can be expand in Taylor
series as follows
\begin{equation}\label{r04}
r(\vartheta,c)=r_1(\vartheta) c+r_2(\vartheta) c^2+\ldots
\end{equation}
with converges for $c$ small and satisfy the initial conditions
\begin{equation}\label{04}
r_1(0)=1,\quad
r_2(0)=r_3(0)=\ldots=0.\end{equation}
 By inserting \eqref{04} into \eqref{03} we obtain that
\begin{equation}\label{05}
\dfrac{dr_1(\vartheta)}{d\vartheta}=0,\quad\dfrac{dr_j(\vartheta)}{d\vartheta}=F_j(\vartheta),
\end{equation}
for $j=2,3,\ldots,$  where $F_2=r^2_1R_2,
F_3=r^3_1R_3+2r_1r_2R_2,\quad \ldots.$

\smallskip

After integration \eqref{05} and in view of \eqref{04} we obtain
\[
r_1=1,\quad r_j=\displaystyle\int_0^\vartheta
F_j(\vartheta)d\vartheta,
\]
for $j=2,3,\ldots.$

\smallskip

By considering that the function $R_2$ is $2\pi$ periodical
functions then  we have the following expansion in Fourier's series
\[R_2(\vartheta)=g_2+\displaystyle\sum_{n=1}^\infty\left(
A_n\cos\left(\dfrac{2\pi
n\vartheta}{\omega}\right)+B_n\sin\left(\dfrac{2\pi
n\vartheta}{\omega}\right)\right),
\]
Consequently
\[
r_2=\displaystyle\int_0^\vartheta
R_2(\vartheta)d\vartheta=g_2\,\vartheta+\varphi_2(\vartheta),\quad
\varphi_2(\vartheta+2\pi)=\varphi_2(\vartheta).
\]
If $g_2=0$ then $r_2$ is a $2\pi$ periodical function. Hence
$F_3=R_3+2r_2R_2$ is $2\pi$ periodical function. Thus have the
following expansion in Fourier's series
\[F_3(\vartheta)=\tilde{g}+\displaystyle\sum_{n=1}^\infty\left(
\tilde{A}_n\cos\left(\dfrac{2\pi
n\vartheta}{\omega}\right)+\tilde{B}_n\sin\left(\dfrac{2\pi
n\vartheta}{\omega}\right)\right).
\]
Hence
\[
r_3=\displaystyle\int_0^\vartheta
F_3(\vartheta)d\vartheta=\tilde{g}\,\vartheta+\tilde{\varphi}(\vartheta),\quad
\tilde{\varphi}(\vartheta+2\pi)=\tilde{\varphi}(\vartheta).
\]
If $r_m$ is the first non-periodical function then it takes the form
\begin{equation}\label{07}
r_m(\vartheta)=\displaystyle\int_0^\vartheta
F_m(\vartheta)d\vartheta=g\,\vartheta+\varphi(\vartheta),\quad
\varphi(\vartheta+2\pi)=\varphi(\vartheta)
\end{equation}
If all coefficients $r_m$ are periodic functions then \eqref{04} is
a periodic function for all values of $c$ which belong to the
convergence region of  the series \eqref{04}. Consequently all the
integral curves which are in the neighborhood of the origin are
closed curves which contain the point $(0,0)$. Poincar\'e called
this point {\it center}. The origin in this case is stable.

Now we assume that $r_2,r_3,\ldots,r_{m-1}$ are periodic functions
and $r_m$ can be determine by the formula \eqref{07} with $g\ne 0.$

\smallskip

We shall transform equation \eqref{03} by using the change
\[
\begin{array}{rl}
r=&\varrho+\varrho^2r_2(\vartheta)+\ldots+\varrho^{m-1}r_{m-1}(\vartheta)+\varrho^m\varphi(\vartheta)\vspace{0.2cm}\\
=&\varrho+\varrho^2r_2(\vartheta)+\ldots+\varrho^{m-1}r_{m-1}(\vartheta)+\varrho^m\left(r_m(\vartheta)-g\vartheta\right).
\end{array}
\] Hence in view of \eqref{05} we get
\[\begin{array}{rl}
\dfrac{d\varrho}{d\vartheta}=&\dfrac{\varrho^2F_2+\varrho^3F_3+\ldots+\varrho^mF_m+
\varrho^{m+1}F_{m+1}+\ldots}{1+2r_2\varrho+\ldots+m\varrho^{m-1}\varphi(\vartheta)}\vspace{0.2cm}\\
=&\dfrac{g\varrho^m+\varrho^{m+1}F_{m+1}+\ldots}{1+2r_2\varrho+\ldots+m\varrho^{m-1}\varphi(\vartheta)}=
g\varrho^m+\varrho^{m+1}\tilde{F}_{m+1}(\vartheta)+\ldots.
\end{array}\]
From this relation follows that in the small neighborhood of the
origin $\dfrac{d\varrho}{d\vartheta}$ has constant sign which
coincide with the sign of $g.$ If $g<0$ then if
$\vartheta\longrightarrow+\infty$ then
$\varrho(\vartheta)\longrightarrow 0,$ if  $g>0$ then when
$\vartheta\longrightarrow-\infty$ then
$\varrho(\vartheta)\longrightarrow 0$ the same behavior has the
radius $r.$ Consequently all integral curves which are in the
vicinity of the origin are spirals. Poincar\'e called this
trajectory {\it foci.} If $g<0$ is asymptotic stable foci and if
$g>0$ is  unstable foci.

\smallskip

The problems arise when the firs non periodic coefficient has a big
index, in this case it is necessary to do a a lot of tedious
computations. But the most complicated problem appears when all the
coefficients $r_j$ are periodic functions, ie. when the origin is a
center.  In this case we have an infinity number of conditions and
it is practically impossible to show that all these conditions hold.
In the general case this problem is open. Poincar\'e in \cite{Po}
proved the following theorem
\begin{theorem}
Polynomial differential system
\[
\dot{x}=-\,y+\displaystyle\sum_{j=2}^{m}X_j(x,y),\quad \dot{y}=
x+\displaystyle\sum_{j=2}^{m}Y_j(x,y),
\]
has a center at the origin if and only if there exists a local first
integral  of the form
\[
F(x,y)=\dfrac{1}{2}(x^2+y^2)+f(x,y),
\]
 defined in  a neighborhood of the origin, where $f=f(x,y)$ is a real analytic
functions in an open neighborhood of  $\textsc{O}$ whose Taylor
expansions at $\textsc{O}$ do not contain constant, linear and
quadratic terms.
\end{theorem}

\smallskip

If we determine the ''time" dependence of the variables $x$ and $y$
i.e, $x=x(t)$ and $y=y(t)$ then these  functions are periodic
functions with period $T$ which depend analytically on the parameter
$c.$ To obtain this dependence we transform the second of equation
\eqref{02} by using \eqref{r04}. We obtain

\[\dfrac{d\vartheta}{dt}=\lambda+\theta(r,\vartheta)=
\lambda\left(1+c\theta_1(\vartheta)+c^2\theta_2(\vartheta)+\ldots\right),
\]
where $\theta_j=\theta_j(\vartheta)$ are convenient periodic
function. Without loss the generality we assume that
$\vartheta|_{t=0}=0.$ Thus we have that

\[\begin{array}{rl}
t(\vartheta)=&\dfrac{1}{\lambda}\displaystyle\int_0^\vartheta
\dfrac{d\vartheta}{1+c\theta_1(\vartheta)+c^2\theta_2(\vartheta)+\ldots}\vspace{0.2cm}\\
=&\dfrac{1}{\lambda}\displaystyle\int_0^\vartheta\left(1+c\tilde{\theta}_1(\vartheta)+c^2\tilde{\theta}_2(\vartheta)+\ldots\right),
\end{array}
\]
in view of periodic character of function $\theta_j$ we deduce the
relation
\[
t(\vartheta+2\pi)-t(\vartheta)=\dfrac{1}{\lambda}\displaystyle\int_0^{\vartheta+2\pi}
\dfrac{d\vartheta}{1+c\theta_1(\vartheta)+c^2\theta_2(\vartheta)+\ldots}.
\]
Thus we obtain that if the origin is center then periodic solution
has period $T$ such that
\[
\begin{array}{rl}
T:=t(2\pi)-t(0)=&\dfrac{1}{\lambda}\displaystyle\int_0^{2\pi}
\dfrac{d\vartheta}{1+c\theta_1(\vartheta)+c^2\theta_2(\vartheta)+\ldots}\vspace{0.2cm}\\
=&\dfrac{2\pi}{\lambda}\left(1+h_1c+h_2c^2+\ldots\right),
\end{array}
\]
where
$h_j=\dfrac{1}{2\pi}\displaystyle\int_0^{2\pi}\theta_j(\vartheta)d\vartheta.$
This series  converge for small $c.$ The analytic function $T=T(c)$
is called {\it period function}. If  $ T(c)=\dfrac{2\pi}{\lambda}$
then the center is called {\it isochronous center}.

\smallskip

{\it Method II}

\smallskip

To study the case when the origin is a center for \eqref{3} we
transform this equation with the change
\[t=\dfrac{\tau}{\lambda}\left(1+h_1c+h_2c^2+\ldots\right).\]
Under this change \eqref{3} becomes
\begin{equation}\label{09}
\begin{array}{rl}
{x}'=&(-y+\dfrac{X}{\lambda})\left(1+h_1c+h_2c^2+\ldots\right)\vspace{0.2cm}\\
=&(-y+\dfrac{1}{\lambda}\left(X_2+X_3+\ldots\right))\left(1+h_1c+h_2c^2+\ldots\right),\\
{y}'=&\left(x+\dfrac{Y}{\lambda})(1+h_1c+h_2c^2+\ldots)\right)\vspace{0.2cm}\\
&\left(x+\dfrac{1}{\lambda}\left(Y_2+Y_3+\ldots)\right)(1+h_1c+h_2c^2+\ldots\right),
\end{array}
\end{equation}
where $'=\dfrac{d}{d\tau},$ and $X_k=X_k(x,y),\,Y=Y_k(x,y)$ are
homogenous polynomial of degree $k,$ for $k=2,3,\ldots.$  Solution
of these equations with the initial conditions
\[
x|_{\tau=0}=c,\quad y|_{\tau=0}=0,
\]
is $2\pi$ periodic function. By considering that \eqref{09} has an
analytic dependence on the parameter $c.$ In view of the well known
theorem of the dependence of the solutions on the parameter we
deduce the validity of the relations
\begin{equation}\label{111}
\begin{array}{rl}
x=x(\tau)=&cx_1(\tau)+c^2x_2(\tau)+\ldots,\\
y=y(\tau)=&cy_1(\tau)+c^2y_2(\tau)+\ldots,
\end{array}
\end{equation}
 these series converges for small values of $c.$ Clearly that these
solutions are such that
\[x(\tau+2\pi)=x(\tau),\quad y(\tau+2\pi)=y(\tau),
\]
with the initial conditions
\begin{equation}\label{011}
x_1|_{\tau=0}=1,\quad x_2|_{\tau=0}=0,\ldots,\quad
y_1|_{\tau=0}=0,\quad y_2|_{\tau=0}=0,\ldots.
\end{equation}
Inserting \eqref{111} into \eqref{09} and comparing different power
coefficients  of $c$ we obtain the following set of differential
equations with respect to $x_k$ and $y_k.$
\[
\begin{array}{rl}
x'_1=&-y_1,\quad y'_1=x_1,\vspace{0.2cm}\\
x'_2=&-y_2-h_1\,y_1+\dfrac{1}{\lambda}X_2(\cos\tau,\sin\tau),\quad y'_2=-x_2-h_1\,x_1+\dfrac{1}{\lambda}Y_2(\cos\tau,\sin\tau), \vspace{0.2cm}\\
\vdots&\qquad\vdots\qquad \vdots\qquad\qquad\vdots\vspace{0.2cm}\\
x'_k=&-y_k-h_{k-1}\sin\tau+P_k,\quad y'_k=-x_k-h_{k-1}\cos\tau+Q_k,,\vspace{0.2cm}\\
\vdots&\qquad\vdots\qquad \vdots\qquad\qquad\vdots
\end{array}
\]
where $P_k$ and $Q_k$ are polynomials with respect to
$x_2,\,y_2,\ldots,x_{k-1},y_{k-1}$ which coefficients which depends
on the parameters $h_1,h_2,\ldots,h_{k-1}.$

\smallskip

In view of that the first two equations with the initial conditions
$x_1(0)=1$ and $y_1(0)=0$ admits the solutions
\[x_1(\tau)=\cos\tau,\quad y_1(\tau)=\sin\tau,\]
then we obtain that the following two differential equations become
\begin{equation}\label{013}
x'_2=-y_2+P_2(\tau),\quad y'_2=x_2+Q_2(\tau),
\end{equation}
where $P_2(\tau)$ and $Q_2(\tau)$ are $2\pi$ periodic functions such
that
\[
P_2(\tau)=-h_1\sin\tau+\dfrac{1}{\lambda}X_2(\cos\tau,\sin\tau),\quad
Q_2(\tau)=h_1\cos\tau+\dfrac{1}{\lambda}Y_2(\cos\tau,\sin\tau).
\]
Differential system \eqref{013} can be rewritten as follows
\[
z'_2=i z_2+(P_2(\tau)+iQ_2(\tau)),\quad z_2=x_2+i y_2.
\]
The solution of this system by considering that $z_2(0)=0,$ is
\[z_2=\displaystyle\int_0^{\tau}e^{i(\tau-\varsigma)}(P_2(\varsigma)+iQ_2(\varsigma))d\varsigma.\]
Thus this solution is $2\pi$ periodic function if and only if
\begin{equation}\label{015}
\displaystyle\int_0^{2\pi}e^{-i\varsigma}(P_2(\varsigma)+iQ_2(\varsigma))d\varsigma=0.
\end{equation}
 By expanding  in Fourier's series
the functions $P_2$ and $,Q_2$ i.e.
\[
\begin{array}{rl}
P_2=&a_0+\displaystyle\sum_{n=1}^\infty\left( a_n\cos(n\tau)+
b_n\sin(n\tau)\right)=\displaystyle\sum_{n=-\infty}^\infty\alpha_ne^{in\tau},\vspace{0.2cm}\\
Q_2=&b_0+\displaystyle\sum_{n=1}^\infty\left(A_n\cos(n\tau)+B_n\sin(n\tau)\right)=\displaystyle\sum_{n=-\infty}^\infty\beta_ne^{in\tau},
\end{array}
\]
In these notations condition \eqref{015} takes the form
\[
\displaystyle\int_0^{2\pi}e^{i(n-1)\varsigma}(\alpha_n+i\beta_n)d\varsigma=\alpha_1+i\beta_1=a_1+B_1+i(A_1-b_1)=0.
\]
If the all functions $x_2,y_2,\ldots,x_{k-1},y_{k-1}$ are periodic
functions, then the functions $x_k$ and $y_k$ are periodic if and
only if
\[
\displaystyle\int_0^{2\pi}e^{-i\varsigma}(P_k(\varsigma)+iQ_k(\varsigma))d\varsigma=0,\quad
\]
hence
\[\begin{array}{rl}
2\pi\,h_{k-1}+\displaystyle\int_0^{2\pi}(P_k(\varsigma)\cos\varsigma-iQ_k(\varsigma)\sin\varsigma)d\varsigma=&0,
\vspace{0.2cm}\\
\displaystyle\int_0^{2\pi}(P_k(\varsigma)\sin\varsigma-Q_k(\varsigma)\cos\varsigma)d\varsigma=&0.
\end{array}\]
The first relation always holds by choosing properly the constant
$h_{k-1}.$  If second relation do not takes place, ie.
\[
\displaystyle\int_0^{2\pi}(P_k(\varsigma)\sin\varsigma-Q_k(\varsigma)\cos\varsigma)d\varsigma=\dfrac{a_1+B_1}{2}\ne
0,
\]
consequently we have that
\[
\begin{array}{rl}
y|_{\tau=2\pi}=&y|_{\tau=0}=0,\vspace{0.2cm}\\
x|_{\tau=2\pi}=&c+2\pi\dfrac{a_1+B_1}{2}c^m+\ldots,\quad
x|_{\tau=0}=c.
\end{array}
\]
If $a_1+B_1>0\,(a_1+B_1<0)$ then the origin is asymptotically stable
(unstable) foci.

\smallskip

Finally we observe that the first constants $h_k$ different from
zero has always  even index.

\smallskip

{\it Method III}

\smallskip

 Another method to study the origin's stability by
applying the following Liapunov's Theorem.

Consider the differential system

\begin{equation}\label{Liap}
\dot{\textbf{x}} = X(\textbf{x})),\quad \textbf{x}(0) =
\textbf{x}_0,
\end{equation}
where $\textbf{x}= \left(x_1,x_2,\ldots,x_N\right),$ and   $X:
\mathcal{D} \longrightarrow \mathbb{R}^n$ is analytic function,
$\mathcal{D}$ is an open set containing the origin, which is an
equilibrium point, i.e.  $X(0)=0.$

\smallskip

Lyapunov, in his original 1892 work, proposed two methods for
demonstrating stability (see \cite{Malkin1,Liapunov}). The first
method developed the solution in a series which was then proved
convergent within limits. The second method, which is almost
universally used nowadays, makes use of a Lyapunov function
$V(\textbf{x})$ which has an analogy to the potential function of
classical dynamics. It is introduced as follows for a system having
a point of equilibrium at $\textbf{x}=\textbf{0}$. Consider a
function $V(\textbf{x}) : \mathbb{R}^n \rightarrow \mathbb{R}$ such
that
\begin{itemize}
\item[(i)]
$V(x) \ge 0$ with equality if and only if $\textbf{x}=\textbf{0},$
i.e. $V$ is a positive definite function.
\item[(ii)]
 $\dot{V}(\textbf{x}) = \dfrac{d}{dt}V(\textbf{x})\le 0$ with equality not
constrained to only $\textbf{x}=\textbf{0},$ i.e. function
$\dot{V}(\textbf{x})$ is  negative semi-definite.  For asymptotic
stability, $\dot{V}(\textbf{x})$  is required to be negative
definite. Then $V(\textbf{x})$ is called a Lyapunov function
candidate and the system is stable in the sense of Lyapunov.
\end{itemize}
\smallskip

It is easier to visualize this method of analysis by thinking of a
physical system (e.g. vibrating spring and mass) and considering the
energy of such a system. If the system loses energy over time and
the energy is never restored then eventually the system must grind
to a stop and reach some final resting state. This final state is
called the attractor. However, finding a function that gives the
precise energy of a physical system can be difficult, and for
abstract mathematical systems, economic systems or biological
systems, the concept of energy may not be applicable.

Lyapunov's realization was that stability can be proven without
requiring knowledge of the true physical energy, provided a Lyapunov
function can be found to satisfy the above conditions.
\begin{theorem} (Liapunov's Theorem)
 Let $\textbf{x} = \textbf{0}$ be an equilibrium point for the
system \eqref{Liap}. Let $V $ be a positive definite continuously
differentiable function in $\textsc{D}\subseteq\mathbb{R}^N.$
\begin{itemize}
\item[(1)]
 If $\dfrac{dV}{dt}$ is negative semi-definite, then
the origin is stable.
\item[(2)] If $\dfrac{dV}{dt}$ is negative definite , then the origin is
asymptotically stable.
\item[(3)] If $\dfrac{dV}{dt}$ is positive definite , then the origin is
 unstable.
\end{itemize}
\end{theorem}
Now we apply Liapunov's  Theorem to study the stability of the
analytic planar vector field associated to differential system
\eqref{Liap} with $N=2,$ and $x_1=x,\, x_2=y$
\[
\dot{x}=-y+\displaystyle\sum_{j=1}^\infty\,X_j(x,y),\quad
\dot{y}=x+\displaystyle\sum_{j=1}^\infty\,Y_j(x,y),
\]
where $X_j$ and $Y_j$ are homogenous polynomials of degree $j.$

\smallskip
We will need the following results (see for instance \cite{Malkin}).
\begin{theorem}\label{Chetaev1}
The partial first order differential equations
\[
x\dfrac{\partial\,f_n}{\partial\,y}-y\dfrac{\partial\,f_n}{\partial\,x}=g_n,
\]
where $f_n$ and $g_n$ are homogenous polynomial of degree $n,$
admits a unique solution if $n$ is odd and if $n$ is even i.e.
$n=2m$ and $g_n$ is a polynomial of degree $2m$ and such that
\[
g_{2m}=\displaystyle\sum_{k+j=2m}a_{kj}x^j\,y^k+K_{m-1}(x^2+y^2)^m.
\]
\end{theorem}
where $a_{kj}$ are constants and $K_{m-1}$ are arbitrary constants
such that
\[K_{m-1}=-\dfrac{1}{2\pi}\displaystyle\sum_{k+j=2m}a_{kj}\displaystyle\int_0^{2\pi}\cos^kt\sin^jtdt.\]

\smallskip

The aim of the following studies is to construct the Liapunov
function which satisfy Liapunov's Theorem . Thus we shall determine
the function $V:$
\[V=\displaystyle\sum_{n=2}^\infty H_n(x,y):=\dfrac{1}{2}(x^2+y^2)+\displaystyle\sum_{n=3}^\infty H_n(x,y),\]
where $ H_n(x,y)$ are homogenous polynomials of degree $n,$ i.e.
\[x\dfrac{\partial\,H_n}{\partial\,x}+y\dfrac{\partial\,H_n}{\partial\,y}=nH_n.\]

We choose the functions $H_n$ in such a way that $\dfrac{dV}{dt}$ is
a positive (negative) definite function. By considering that
\begin{equation}\label{Polynomial}
\begin{array}{rl}
\dfrac{dV}{dt}=&\left(x+\dfrac{\partial\,H_3}{\partial
x}+\ldots\right)\left(-y+X_2+X_3+\ldots\right)\vspace{0.2cm}\\
&+\left(y+\dfrac{\partial\,H_3}{\partial
y}+\ldots\right)\left(x+Y_2+Y_3+\ldots\right)\vspace{0.2cm}\\
=&xX_2+yY_2+\{H_2,H_3\}\vspace{0.2cm}\\
&+xX_3+yY_3+\dfrac{\partial\,H_3}{\partial
x}X_2+\dfrac{\partial\,H_3}{\partial y}Y_2+\{H_2,H_4\}\vspace{0.2cm}\\
&xX_4+yY_4+\dfrac{\partial\,H_3}{\partial
x}X_3+\dfrac{\partial\,H_3}{\partial
y}Y_3\vspace{0.2cm}\\
&+\dfrac{\partial\,H_4}{\partial
x}X_2+\dfrac{\partial\,H_4}{\partial y}Y_2+\{H_2,H_5\}
+\ldots\vspace{0.2cm}\\
&\vdots\qquad\qquad\vdots\qquad\qquad
\vdots\qquad\qquad\vdots\vspace{0.2cm}\\
&xX_{m}+yY_m+\dfrac{\partial\,H_3}{\partial
x}X_{m-1}+\dfrac{\partial\,H_3}{\partial
y}Y_{m-1}+\ldots\vspace{0.2cm}\\
&+\dfrac{\partial\,H_{m}}{\partial x}X_{2}
+\dfrac{\partial\,H_{m}}{\partial
y}Y_2+\{H_2,H_{m+1}\}\vspace{0.2cm}\\
&\vdots\qquad\qquad\vdots\qquad\qquad \vdots\qquad\qquad\vdots
\vspace{0.2cm}\\
 :=&L_3+L_4+L_5+\ldots+L_{m+1}+\ldots
\end{array}
\end{equation}
where $\{f,g\}:=\dfrac{\partial f}{\partial
x}\dfrac{\partial\,g}{\partial y}- \dfrac{\partial f}{\partial
y}\dfrac{\partial\,g}{\partial x},$ $L_n$ are homogenous polynomial
of degree $n.$ Function $\dfrac{dV}{dt}$ will be positive (negative)
definite it is necessary that it beginning with even term. Thus
\[L_3=xX_2+yY_2+\{H_2,H_3\}=0,\]
By considering that $X_2,\,Y_2, H_2$ and $H_3$ are homogenous
polynomial, then  in view of Theorem \ref{Chetaev1} we obtain that
there exist an unique homogenous polynomial $H_3$  which is a
solution of this equation.

\smallskip

 Another hand, from the relation
\[L_4=\,xX_3+yY_3+\dfrac{\partial\,H_3}{\partial
x}X_2+\dfrac{\partial\,H_3}{\partial y}Y_3+\{H_2,H_4\},\] and
 by considering that
 \[
F_4:=-xX_3-yY_3-\dfrac{\partial\,H_3}{\partial
x}X_2-\dfrac{\partial\,H_3}{\partial y}Y_3+L_4
\]
 is homogenous polynomial of degree four, then taking
$L_4=V_1(x^2+y^2)^2,$  and in view of Theorem \ref{Chetaev1} we
deduce that there exist an unique polynomial $H_4.$

\smallskip

By continue these process we obtain that it is possible to construct
the Liapunov function which satisfy Liapunov's Theorems.

Clearly is the constants $V_j=0$  for $j\in\mathbb{N}$ then there
exist a first integral
\[
\dfrac{\lambda}{2}(x^2+y^2)+\displaystyle\sum_{j=3}^\infty
H_j=C.
\]
Consequently the origin is a center. If there exist a non--zero
Liapunov constant $V_j$ then in view of the relation
\[
\dfrac{d V}{dt}=V_j(x^2+y^2)^{j+1}+\ldots,
\]
the origin is asymptotically stability foci if $V_j<0$ and
unstability foci if $V_j>0.$

\section{Poincar\'e-Liapunov integrability }
We introduce the following definition

\smallskip
 Let $U$ be an open and dense set in
$\R^2$. We say that a non-constant $C^1$ function $H \colon U \to
\R$ is a \emph{first integral} of the polynomial vector field $\X$
on $U$, if $H(x(t),y(t))$ is constant for all values of $t$ for
which the solution $(x(t),y(t))$ of $\X$ is defined on $U$. Clearly
$H$ is a first integral of $\X$ on $U$ if and only if $\X H=0$ on
$U$.

\smallskip

 Differential system \eqref{1} for which the origin is a
singular point, i.e. $P(0,0)=Q(0,0)=0$  we call {\it
Poincar\'e-Liapunov integrable} if admits a  first integral
$F=F(x,y)$ such that the development at the point $(0,0)$ is the
following
\[
F=F(0,0)+\dfrac{1}{2}(ax^2+by^2+cxy)+f(x,y),
\]
where $f$ is a real analytic functions in an open neighborhood of
$\textsc{O}$ whose Taylor expansions at $\textsc{O}$ do not contain
constant and linear terms and $ax^2+by^2+cxy $ is a definite
positive (negative) quadratic form (see \cite{R1}).

\smallskip

Now we shall study the case when the differential system is given by
the formula \eqref{3}. Particular cases of Poincar\'e-Liapunov
integrable system are the following:

\begin{itemize}

\item[(i)] System \eqref{3} is a Hamiltonian system, i.e.
\[
X(x,y)=-\dfrac{\partial F(x,y)}{\partial y},\quad
Y(x,y)=\dfrac{\partial F(x,y)}{\partial x}.
\]
Hence $F$ is a first integral.

\smallskip
\item[(ii)]

Beside Hamiltonian  systems there is another class of systems
\eqref{3} for which the origin is a center, namely that of
reversible systems satisfying the following definition.

\smallskip

We say that system \eqref{3} is {\it reversible with respect to the
straight line $l$} through the origin if it is invariant with
respect to reversion about $l$ and a reversion of time $t,$ (see for
instance \cite{Conti}).

\smallskip

The following criterion going back to Poincar\'e, namely (see for
instance \cite{Stepanov}, p.122)
\begin{theorem}
The origin is a center of \eqref{3} if it is reversible.
\end{theorem}

\smallskip

We apply this theorem for the case when \eqref{3}  is invariant
under the transformations $(x,y,t)\longrightarrow (-x,y,-t)$ and
$(x,y,t)\longrightarrow (x,-y,-t).$

\smallskip

 By introducing the complex coordinates
$z=x+i\,y,\,\bar{z}=x-i y,$ we can rewrite the analytic planar
differential systems as follows
\begin{equation}\label{tt12}
\dot{z}=\displaystyle\sum_{n,k=0}^M a_{nk}z^n\bar{z}^k,
\end{equation}
where $a_{nk}\in\mathbb{C},$ and $M$ can be equal to infinity.
\begin{proposition}\label{rev1}
Differential systems \eqref{tt12} is invariant under the change
$(x,y,t)\longrightarrow (x,-y,-t),$ if and only if $a_{nk}=ib_{nk}$
where $b_{nk}\in\mathbb{R}.$
\end{proposition}
\begin{proof}
Under the change $(t,x,y)\longrightarrow (-t,x,-y)$ differential
system \eqref{tt12} takes the form
\[
\dot{\bar{z}}=-\displaystyle\sum_{n,k=0}^\infty
a_{nk}\bar{z}^n{z}^k.
\]
On the other hand from \eqref{tt12} follows that
\[
\dot{\bar{z}}=\displaystyle\sum_{n,k=1}^\infty\bar{a}_{nk}\bar{z}^n{z}^k
\] hence $\bar{a}_{nk}=-a_{nk}.$ From this relation
the proof follows.
\end{proof}
We shall study the case when
\begin{equation}\label{eq1}
\dot{z}=i\left(z+\displaystyle\sum_{n,k=2}^\infty
b_{nk}z^n\bar{z}^k\right),\quad b_{nk}\in\mathbb{R} .
\end{equation}
Consequently analytic differential system \eqref{3}  is reversible
if and only
\[X(x,y)=yf(x,y^2),\quad
Y(x,y)= g(x,y^2),
\]
where $f$ and $g$ are analytic  functions such that  $g(0,0)=0.$

\smallskip

From this result follows that differential system
\begin{equation}\label{req1}
\dot{x}=-\,y+yf(x,y^2)),\quad \dot{y}= x+g(x,y^2),
\end{equation}
admits a center at the origin. In particular we deduce the following
system (Lienard's type equation)
\[
\dot{x}=-\,y,\quad \dot{y}= x+g(y^2).
\]

\smallskip

By introducing the variable $Y=y^2$ from \eqref{req1} we obtain the
differential equation
\[\dfrac{dY}{dx}=\dfrac{x+g(x,Y)}{-1+f(x,Y)}.\]
\begin{proposition}\label{rev1}
Differential systems \eqref{tt12} is invariant under the change
$(x,y,t)\longrightarrow (-x,y,-t),$ if and only if
\[
a_{nk}=(-1)^{n+k}\bar{a}_{nk}.
\]
\end{proposition}
\begin{proof}
Under the change $(x,y,t)\longrightarrow (-x,y,t)$ we obtain that
\[
z\longrightarrow -\bar{z},\quad \bar{z}\longrightarrow-z,\quad
t\longrightarrow -t,
\]
 thus differential system \eqref{tt12} takes the form
\[
\dot{\bar{z}}=\displaystyle\sum_{n,k=0}^\infty
a_{nk}(-1)^{n+k}\bar{z}^k{z}^n.
\]
On the other hand from \eqref{tt12} follows that
\[
\dot{\bar{z}}=\displaystyle\sum_{n,k=1}^\infty\bar{a}_{nk}\bar{z}^n{z}^k,
\]
 consequently $\bar{a}_{nk}=(-1)^{n+k}a_{nk}.$ From this relation
the proof follows.

\smallskip

Consequently analytic differential system \eqref{3}  is invariant
under the change $(x,y,t)\longrightarrow (-x,y,-t),$ if and only
\[X(x,y)=f(x^2,y),\quad
Y(x,y)= xg(x^2,y),
\]
where $f$ and $g$ are analytic  functions such that  $f(0,0)=0.$

\smallskip

From this result follows that differential system
\begin{equation}\label{Req1}
\dot{x}=-\,y+f(x^2,y)),\quad \dot{y}= x+xg(x^2,y),
\end{equation}
admits a center at the origin.

\smallskip

Particular case of this system is (Lienard's type equation)
\[
\dot{x}=-\,y+f(x^2),\quad \dot{y}= x,
\]
 Under the change $X=x^2$ we deduce from \eqref{Req1}
the differential equation
\[\dfrac{dX}{dy}=\dfrac{-y+f(X,y)}{1+g(X,y)}.\]
\end{proof}

Differential system
\begin{equation}\label{Rqh}
\dot{x}=-\,y+X_m(x,y),\quad \dot{y}= x+Y_m(x,y),
\end{equation}
where $X_m=X_m(x,y)$ and $Y=Y_m(x,y)$ are homogenous polynomial of
degree $m,$ is called {\it quasi--homogenous differential system.}

\begin{corollary}\label{RVF}
The quasihomogenous differential system \eqref{Rqh} is invariant
under the change $(x,y,t)\longrightarrow (-x,y,-t),$ if and only if
\[
{a}_{nk}=\left\{
               \begin{array}{ll}
                 \bar{a_{nk}}, & \hbox{if $m=2l$,} \\
                 -\bar{a_{nk}}, & \hbox{if $m=2l+1$.}
               \end{array}
             \right.
\]
\end{corollary}
\begin{proof}
Indeed from the relations
$\bar{a}_{nk}=(-1)^{n+k}a_{nk}=(-1)^ma_{nk}$ the proof follows.
\end{proof}

Below for simplicity shall say that the planar differential system
is reversible if it is invariant under the change
$(x,y,t)\longrightarrow (x,-y,-t),$ or $(x,y,t)\longrightarrow
(-x,y,-t).$

From Proposition \ref{rev1} and  Corollary \ref{RVF} we get  the
following result.
\begin{proposition}\label{RVN}
Quasihomogenous differential system \eqref{Rqh} for $m=2l+1$ is
reversible if and only if $\Re(a_{jk})=0$ for $j+k=2l+1,$ i.e.
\begin{equation}\label{qh}
\begin{array}{rl}
\dot{x}=&y\left(a_{00}+a_{2l\,0}x^{2l}+a_{2l-2,\,2}x^{2l}y^2+\ldots
+a_{0,\,2l}y^{2l}\right),\vspace{0.2cm}\\
\dot{y}=&x\left(b_{00}+b_{2l,\,0}x^{2l}+b_{2l-2,\,2}x^{2l}y^2
+\ldots+b_{0,\,2l}y^{2l}\right),
\end{array}
\end{equation}

\end{proposition}

\item[(iii)] The following condition is well known as {\it weak condition
of the center} (see for instance\cite{Lloyd}).

\smallskip
\begin{proposition}(Weak condition of the center)\label{LLloyd}
The origin is a center of \eqref{Rqh} if there  exist
$\mu\in\mathbb{R}$ such that
\begin{equation}\label{LLR}
(x^2+y^2)\left( \dfrac{\partial X_m}{\partial x}+\dfrac{\partial
Y_m}{\partial y} \right)=\mu\left(xX_m+yY_m\right),
\end{equation}
and either $m=2k$ is even, or $m=2k-1$ is odd and $\mu\ne 2k,$ or
$m=2k-1$ and
\begin{equation}\label{9}
\displaystyle\int_0^{2\pi}\left.\left(xX_m+yY_m\right)\right.|_{x=\cos{t},\,y=\sin{t}}dt=0
\end{equation}
\end{proposition}
In \cite{Devlin} the author proved that if $\mu=2m$ then there exist
the rational first integral
\[
\dfrac{x^2+y^2-2\left(xY_m-yX_m\right)}{(x^2+y^2)^{m}}=Const.
\]
\item[(iv)] Another particular case of differential system with a center is the system
which satisfy the Cauchy--Riemann conditions  (see for instance
\cite{Conti}).

\begin{proposition}
 System \eqref{3} has a center at the
origin if  $X$ and $Y$ satisfy the Cauchy-Riemann equation
\begin{equation}\label{CRR}
\dfrac{\partial X}{\partial x}=\dfrac{\partial Y}{\partial y},\quad
\dfrac{\partial X}{\partial y}=-\dfrac{\partial Y}{\partial x}
\end{equation}
\end{proposition}
Center for which \eqref{CRR} holds is called {\it{holmorphic
center}} which is a particular case of the isochronous center.
\end{itemize}
The most general analytic planar differential system with
holomorphic center is
\[\dot{z}=\displaystyle\sum_{j=1}^\infty c_jz^j,\quad c_j\in\mathbb{C}.\]
\section{Inverse problem of the center problem for the polynomial planar vector fields.}

\smallskip

In this section we state and solve the following inverse problem for
of the center for the polynomial planar vector fields.

\smallskip

\begin{problem}
 Determine the polynomial planar vector fields
\begin{equation}\label{RRR}
 \X=(-y+\displaystyle\sum_{j=2}^mX_j(x,y))\dfrac{\partial
}{\partial x}+(x+\displaystyle\sum_{j=2}^mY_j(x,y))\dfrac{\partial
}{\partial y},
\end{equation}
 where $X_j$ and $Y_j$ for
$j=2,3,\ldots,m$ are unknown homogenous polynomial of degree $j,$ in
such a way that the
\begin{equation}\label{a1}
V=\displaystyle\sum_{j=2}^\infty
H_j(x,y)=\dfrac{\lambda}{2}(x^2+y^2)+\displaystyle\sum_{j=3}^\infty
H_j(x,y)=C,
\end{equation}
  is it a first integral , where $H_j$ is a  homogenous polynomial of degree
$j,$ for $j=2,3,\ldots.$
\end{problem}
Analogously problem can be stated for the case when the vector field
$\X$ is analytic.

\smallskip

 The solution of this inverse problem we obtain
from the following theorem.

\smallskip

\begin{theorem}\label{Lia1}
The most general polynomial planar vector field of degree $n$
 for which \eqref{a1} is a first integral is
\begin{equation}\label{inverse1}
\dot{x}=\displaystyle\sum_{j=2}^{m+1}g_{m+1-j}\{\Psi_j,x\},\quad
\dot{y}=\displaystyle\sum_{j=2}^{m+1}g_{m+1-j}\{\Psi_j,y\},
\end{equation}
with the complementary conditions
\begin{equation}\label{inverse2}
\displaystyle\sum_{j=1}^{m}g_{j-1}\{H_{2},H_{n+2-j}\}+\{H_{m+1},H_{n+2-n}\}=0,\quad
n=m,m+1,\ldots
\end{equation}
 where $g_{n+1-j}$ is an arbitrary homogenous
polynomial of degree $n+1-j\geq 0,$ and
\[\Psi_j=\displaystyle\sum_{k=2}^{j}H_k,\quad\mbox{for}\quad j=2,\ldots,m+1,\]
where $H_k$ are homogenous polynomial of degree $k$.
\end{theorem}

\begin{proof}
From the Liapunov Theorem \ref{Liacenter2} we obtain that the origin
is center for the  vector field $\X$ if and only if \eqref{a1} is it
a first integral.  Hence \eqref{Polynomial} holds. This equation is
satisfied if and only if the following relations take place
\[\begin{array}{rl}
xX_2+yY_2+\{H_2,H_3\}=&0,\vspace{0.2cm}\\
xX_3+yY_3+\dfrac{\partial\,H_3}{\partial
x}X_2+\dfrac{\partial\,H_3}{\partial y}Y_2+\{H_2,H_4\}=&0,\vspace{0.2cm}\\
xX_4+yY_4+\dfrac{\partial\,H_3}{\partial
x}X_3+\dfrac{\partial\,H_3}{\partial
y}Y_3+\dfrac{\partial\,H_4}{\partial x}X_2+
\dfrac{\partial\,H_4}{\partial y}Y_2+\{H_2,H_5\}=&0,\vspace{0.2cm}\\
\vdots\qquad\qquad\vdots\qquad\qquad\vdots\qquad\qquad\vdots\qquad\qquad
\vdots \qquad\qquad&\vdots\vspace{0.2cm}\\
xX_{n}+yY_n+\dfrac{\partial\,H_3}{\partial
y}Y_{n-1}+\ldots+\dfrac{\partial\,H_{n}}{\partial x}X_{2}
+\dfrac{\partial\,H_{n}}{\partial y}Y_2+\{H_2,H_{n+1}\}=&0,\vspace{0.2cm}\\
xX_{n+1}+yY_{n+1}+\dfrac{\partial\,H_3}{\partial
y}Y_{n}+\ldots+\dfrac{\partial\,H_{n+1}}{\partial x}X_{2}
+\dfrac{\partial\,H_{n+1}}{\partial y}Y_2+\{H_2,H_{n+2}\}=&0,\vspace{0.2cm}\\
\vdots\qquad\qquad\vdots\qquad\qquad\vdots\qquad\qquad\vdots\qquad\qquad
\vdots \qquad\qquad&\vdots.
\end{array}
\]
  The first equation can be rewritten
as follows
\[x\left(X_2+\dfrac{\partial\,H_3}{\partial y}\right)+y\left(Y_2-\dfrac{\partial\,H_3}{\partial
x}\right)=0,\] by solving with respect to $X_2$ and $Y_2$ we obtain
that
\[\begin{array}{rl}
X_2=&-\dfrac{\partial\,H_3}{\partial
y}-yg_1=\{H_3,x\}+g_1\{H_2,x\},\vspace{0.2cm}\\
Y_2=&\dfrac{\partial\,H_3}{\partial x}+xg_1=\{H_3,y\}+g_1\{H_2,y\},
\end{array}
\] where $g_1=g_1(x,y)$ is an
  arbitrary homogenous polynomial of degree one. Inserting
these polynomial into the second equation we obtain
\[x\left(X_3-\dfrac{\partial\,H_4}{\partial
y}+g_1\dfrac{\partial\,H_3}{\partial y}\right)
+y\left(Y_3-\dfrac{\partial\,H_4}{\partial\,x}-g_1\dfrac{\partial\,H_3}{\partial
x}\right)=0,
\]
solving this equation with respect to $X_3$ and $Y_3$ we have
\[
\begin{array}{rl}
X_3=&-\dfrac{\partial\,H_4}{\partial
y}-g_1\dfrac{\partial\,H_3}{\partial y}-yg_2=\{H_4,x\}+g_1\{H_3,x\}+g_2\{H_2,x\},\vspace{0.2cm}\\
Y_3=&\dfrac{\partial\,H_4}{\partial
x}+g_1\dfrac{\partial\,H_3}{\partial
x}+xg_2=\{H_4,y\}+g_1\{H_3,y\}+g_2\{H_2,y\},
\end{array}
\]
where $g_2=g_2(x,y)$ is  an  arbitrary homogenous polynomial of
degree two.  By continue this process we obtain
$X_4,Y_4,\ldots,X_n,Y_n.$ Inserting the obtained polynomials in the
remaining partial differential equations we deduce \eqref{inverse2},
and introducing the respectively notations we get
\[\begin{array}{rl}
X_2+X_3+\ldots+X_m=&\displaystyle\sum_{j=2}^{m+1}g_{m+1-j}\{\Psi_j,x\},\vspace{0.2cm}\\
Y_2+Y_3+\ldots+Y_m=&\displaystyle\sum_{j=2}^{m+1}g_{m+1-j}\{\Psi_j,y\}.
\end{array}
\]
By inserting $X_j$ and $Y_j$ for $j=1,\ldots,m$ into the remain
equations we deduce the partial differential equations
\eqref{inverse2}. Thus the proof of the theorem follows.
\end{proof}
From Theorem \ref{Lia1} follows that to solve the center problem for
the polynomial planar vector field it is necessary to solve the
infinity number of first order partial differential equations
\eqref{inverse2}. This problem can be solve in some particular case
of the  Poincar\'e-- Liapunov integrable differential system.
\begin{proposition}\label{a2}
Differential system \eqref{inverse1} is Hamiltonian if and only if
\begin{equation}\label{a3}
\displaystyle\sum_{j=2}^{n+1}\{\Psi_j, g_{n+1-j}\}=0.
\end{equation}
\end{proposition}
 Condition \eqref{a3} is called {\it divergence condition.}
\begin{proof}
System \eqref{inverse1} is Hamiltonian if
\[
\displaystyle\sum_{j=2}^{n+1}g_{n+1-j}\{\Psi_j,x\}=-\dfrac{\partial\,H}{\partial
y},\quad
\displaystyle\sum_{j=2}^{n+1}g_{n+1-j}\{\Psi_j,y\}=\dfrac{\partial\,H}{\partial
x}
\]
 $H=H(x,y)$ is
$C^r$ function with $r\geq 1$. From the compatibility conditions
follows that
\[
\displaystyle\sum_{j=2}^{n+1}\left(\dfrac{\partial\,g_{n+1-j}}{\partial
x}\{\Psi_j,x\}+\dfrac{\partial\,g_{n+1-j}}{\partial
y}\{\Psi_j,y\}\right)=\displaystyle\sum_{j=2}^{n+1}\{\Psi_j,
g_{n+1-j}\}=0.
\]

\eqref{a3}. In this case the first integral is
\[\begin{array}{rl}
H(x,y)=&\displaystyle\sum_{j=2}^{m+1}\left(\displaystyle\int_{x_0}^{x}
g_{j+1-j}\{\Psi_j,y\}dx-\displaystyle\int_{y_0}^y
\left.g_{j+1-j}\{\Psi_j,x\}\right|_{x=x_0}dy\right)\vspace{0.2cm}\\
:=&H_2+H_3+\ldots+H_{m+1}+\Omega (x,y)=C,
\end{array}
\] where $H_j=0$ for $j>n.$
\end{proof}
\smallskip

 Now we shall study the particular case of \eqref{inverse1} when
$H_3=H_4=\ldots=H_{m}=0$ and $g_1=g_2=g_{n-2}=0,$ i.e. the
differential system
\begin{equation}\label{inverse21}
\dot{x}=-\dfrac{\partial\,H_{m+1}}{\partial y}-y\,g_{m-1},\quad
\dot{y}=\dfrac{\partial\,H_{m+1}}{\partial x}+x\,g_{m-1},
\end{equation}
with the complementary conditions
\begin{equation}\label{inverse20}
g_{m-1}\{H_{2},H_{n+2-m}\}
+\{H_{2},H_{n+1}\}+\{H_{m+1},H_{n+2-m}\}=0,
\end{equation}
for $ n=m,m+1,\ldots,$  where  $g_{n-1}=g_{m-1}(x,y)$ is an
arbitrary homogenous polynomial of degree $m-1.$
\begin{corollary}\label{cor1}
Differential system \eqref{inverse21}  is Hamiltonian if and only if
\begin{equation}\label{fil}
g_{m-1}=\left\{
    \begin{array}{ll}
      \nu\,H^k_2, & \hbox{if\,$m=2k+1,$} \\
      0, & \hbox{if\,$m=2k,$.}
    \end{array}
  \right.,
\end{equation}
where $\nu$ is an arbitrary constant.
\end{corollary}
\begin{proof}
Clearly that \eqref{inverse21} is Hamiltonian if
$\{H_2,g_{m-1}\}=0,$ consequently \eqref{fil} holds. The
 Hamiltonian is  $H=H_2+H_3+\ldots+H_{2k}=C$
if $m=2k$  and $H=H_2+H_3+\ldots+H_{2k+1}+\nu H^{k+1}_2=C$ if
$m=2k+1$ . Conditions \eqref{inverse20} in this case hold
identically in view of that $H_j=0$ for $j>2k,\,m=2k$ and
$H_{2k+2}=\nu H^{k+1}_2,\,H_j=0$ for $j>2k+2,\,m=2k+1.$
\end{proof}
From conditions \eqref{inverse2} and by considering that
\[\displaystyle\int_0^{2\pi}\left.\{H_2,G\}\right|_{x=\cos{t},\,y=\sin{t}}dt=
\displaystyle\int_0^{2\pi}\left.\dfrac{dG}{dt}\right|_{x=\cos{t},\,y=\sin{t}}dt=0,\]
for arbitrary $C^1$ function $G=G(x,y),$ we deduce the relations
\[
\displaystyle\int_0^{2\pi}\left.\left(\displaystyle\sum_{j=2}^{m}g_{j-1}\{H_{2},H_{n+2-j}\}
+\{H_{m+1},H_{n+2-m}\}\right)\right|_{x=\cos{t},\,y=\sin{t}}dt=0,
\]
for $ n=m,m+1,\ldots$.

\smallskip

From the previous proposition we have the following results.

\smallskip
\begin{corollary}
 If the polynomial planar vector field
$\textsc{Y}=P(x,y)\dfrac{\partial}{\partial
x}+Q(x,y)\dfrac{\partial}{\partial y}$ of degree $n$ has a
non--degenerate center at the origin then can be written as
\[
\displaystyle\sum_{j=2}^{m+1}g_{m+1-j}\{\Psi_j,x\}=P(x,y),\quad
\displaystyle\sum_{j=2}^{m+1}g_{m+1-j}\{\Psi_j,y\}=Q(x,y).
\]
\end{corollary}
 Hence the arbitrary functions $g_1,g_2,\ldots,g_{m-1}$ must be
satisfies the partial differential equation
\begin{equation}\label{inverse22}
\displaystyle\sum_{j=2}^{m}\{\Psi_j, g_{m+1-j}\}=\dfrac{\partial
P(x,y)}{\partial x}+\dfrac{\partial Q(x,y)}{\partial y}.
\end{equation}

\section{Weak conditions of the center. Generalization}

Below we need the following concept and result.

\smallskip

 Let $\R[x,y]$ be the ring of all real polynomials in the
variables $x$ and $y$, and let
\[\X=P\dfrac{\partial }{\partial x}+Q\dfrac{\partial}{\partial\,y}
\]
 be a polynomial vector field of
degree $m$, and let $g=g(x,y)\in \R[x,y]$. Then $g=0$ is an {\it
invariant algebraic curve} of $\X$ if
\[
\X g=P\dfrac{\partial g}{\partial x}+Q\dfrac{\partial g}{\partial
y}=K g,
\]
where $K=K(x,y)$ is a polynomial of degree at most $m-1$, which is
called the {\it cofactor} of $g=0$. If the polynomial $g$ is
irreducible in $\R[x,y]$, then we say that the invariant algebraic
curve $g=0$ is {\it irreducible} and that its {\it degree} is the
degree of the polynomial $g$.  For more details on the so--called
Darboux theory of integrability see for instance the Chapter 8 of
\cite{DLA}.

For the polynomial system \eqref{Rqh}
 the divergence condition can be weakened as we give in Proposition
\ref{LLloyd}. By applying this result and in view of Theorem
\ref{Lia1}
 we deduce the proof of the next result.
\begin{proposition}\label{Lia11}
If \eqref{LLR} holds then the system \eqref{Rqh} is
Poincar\'{e}--Liapunov integrable.
\end{proposition}
\begin{proof}
By considering that \eqref{LLR} holds then the system \eqref{Rqh}
has a center, consequently  can be rewritten as follows
\[
\begin{array}{rl}
\dot{x}=&-y+X_m=-\dfrac{\partial\,H_2+H_{m+1}}{\partial
y}-y\,g_{m-1},\vspace{0.2cm}\\
 \dot{y}=&x+Y_m=\dfrac{\partial\,H_2+H_{m+1}}{\partial
x}+x\,g_{m-1}.
\end{array}
\]
 Hence
\[\dfrac{\partial X_m}{\partial x}+\dfrac{\partial
Y_m}{\partial y}=\{H_2,g_{m-1}\},\quad xX_m+yY_m=\{H_{m+1},H_2\},
\]
consequently the condition \eqref{LLR} becomes
\[\lambda H_2\{H_2,g_{m-1}\}=\{H_{m+1},H_2\}\Longrightarrow \{H_2,H_{m+1}+\lambda H_2g_{m-1}\}=0\] Thus
\[H_{m+1}=\left\{
        \begin{array}{ll}
        - \lambda H_2g_{m-1}, & \hbox{if\,$m=2k$,} \\
          - \lambda H_2g_{m-1}+\nu H^{k+1}_2, & \hbox{if\,$m=2k+1$.}
        \end{array}
      \right.
\]
Below we study only the case when $\nu=0.$

\smallskip

 From the equations
\begin{equation}\label{jj1}
\begin{array}{rl}
\dot{x}=&-y-\dfrac{\partial\,H_{m+1}}{\partial
y}-yg_{m-1}\vspace{0.2cm}\\
=&\lambda\,H_2\,\dfrac{\partial\,g_{m-1}}{\partial
y}-y\left(1+(1-\lambda)\,g_{m-1})\right)\vspace{0.2cm}\\
\dot{y}=& x+\dfrac{\partial\,H_{m+1}}{\partial
x}+xg_{m-1}\vspace{0.2cm}\\
=&-\lambda\,H_2\,\dfrac{\partial\,g_{m-1}}{\partial
x}+x\left(1+(1-\lambda)\,g_{m-1})\right)
\end{array}
\end{equation}
follows that
\[
\begin{array}{rl}
 \dot{H}_2=&\lambda
H_2\{H_2,g_{m-1}\},\vspace{0.2cm}\\
\dot{g}_{m-1}=&\{H_2,g_{m-1}\}\left(1+(1-\lambda\,)g_{m-1}\right),
\end{array}
\]
thus that the curve $x^2+y^2=0$ and $1+(1-\lambda)g_{m-1}=0$ are
invariant curves of the polynomial vector field if
$\lambda\in\mathbb{R}\setminus\{0,1\}$. The cofactors are $\lambda
\{H_2,g_{m-1}\}$ and $ \{H_2,g_{n-1}\}$ respectively.

\smallskip

The first integral
\begin{equation}\label{ak2}
\begin{array}{rl}
F=&\left(1+(1-\lambda)g_{m-1}\right)^{\lambda/(\lambda-1)}H_2=H_2
-\lambda\,H_2g_{m-1}+\ldots\vspace{0.2cm}\\
=&H_2+f(x,y)=Const.,
\end{array}
\end{equation}
is the Darboux first integral (see \cite{RS1}),
 where
  \[ \lambda=2/\mu\in{\mathbb{R}\setminus{\{0,1\}}}
  \]
    and $f=f(x,y)$ is a real analytic
functions in an open neighborhood of  $\textsc{O}$ whose Taylor
expansions at $\textsc{O}$ do not contain constant, linear and
quadratic terms.

\smallskip

Clearly that in this  case we have that
\[H_3=H_4=\ldots=H_{m}=0,\quad H_{m+1}=-\lambda\,H_2g_{m-1},\ldots.\]
After some computations follows that these functions satisfy the
complementary conditions \eqref{inverse20}.

\smallskip

If $\lambda=1/m$ then this first integral is the rational function
\begin{equation}\label{first}
\tilde{F}=\dfrac{H^{m-1}_2}{\left(1+(1-1/m)g_{m-1}\right)}.
\end{equation}

\smallskip

If $\lambda=1$ then system \eqref{jj1} takes the form

\[
\dot{x}=-y+\,H_2\,\dfrac{\partial\,g_{m-1}}{\partial y},\quad
\dot{y}=x-H_2\,\dfrac{\partial\,g_{m-1}}{\partial x}
\]
consequently
 $$
\dot{H}_2= H_2\{H_2,g_{m-1}\},\quad \dot{g}_{m-1}=\{H_2,g_{m-1}\}.
$$

Thus $F=H_2\,e^{-g_{m-1}}$ is the first integral, which admits the
following Taylor expansion in the neighborhood of the origin
\[
F=H_2\,e^{-g_{m-1}}=H_2-H_2g_{m-1}+H_2\dfrac{g^2_{m-1}}{2!}+\ldots.
\]
Hence
\[H_3=H_4=\ldots=H_{m}=0,\quad H_{m+1}=-H_2g_{m-1},\ldots.\]
These functions are solutions of the system \eqref{inverse20}.

\smallskip

 The second condition of the center in this case
takes the form
\[
\begin{array}{rl}
\displaystyle\int_0^{2\pi}\left.\left(xX_m+yY_m\right)\right.|_{x=\cos{t},\,y=\sin{t}}dt=&
\displaystyle\int_0^{2\pi}\left.\{H_2,H_{m+1}\}\right.|_{x=\cos{t},\,y=\sin{t}}dt\vspace{0.2cm}\\
=& \displaystyle\int_0^{2\pi}\left.\dfrac{d
H_{m+1}}{dt}\right.|_{x=\cos{t},\,y=\sin{t}}dt=0.
\end{array}
\]
In short the proposition is proved\end{proof}.
\smallskip

\begin{corollary}\label{rrr}
For the system \eqref{jj1} admits an  isochronous center if
$\lambda=1/m$ and uniformly  isochronous center at the origin if
$\lambda=2/(m+1).$
\end{corollary}
\begin{proof}
From \eqref{jj1} follows that
\[
x\dot{y}-y\dot{x}=2H_2\left(1+\dfrac{g_{m-1}}{2}(2-(m+1)\lambda)\right),\]
 which in
polar coordinates $x=r\cos \theta,\,y=r\sin \theta$ becomes
\[\dot{\theta}=1+\dfrac{\left(2-(m+1)\lambda\right)}{2}r^{m-1}\Psi(\vartheta),\]
where $\Psi(\vartheta)$ is a $2\pi$ periodic function. Hence if
$\lambda=2/(m+1)$ then $ \dot{\theta}=1. $ The firs integral in this
case is a rational function
\[ F=\dfrac{H^{m-1}_2}{\left(1+\dfrac{m-1}{m+1}g_{m-1}\right)^2},\]
Consequently the origin is an uniformly isochronous center.

\smallskip

We observe that if $\lambda=2/(m+1)$ then differential system
\eqref{jj1} becomes
\[\begin{array}{rl}
\dot{x}=&-y+\dfrac{1}{m+1}\left((x^2+y^2)\dfrac{\partial
g_{m-1}}{\partial
y}-y(m-1)g\right)\vspace{0.2cm}\\
=&-y+\dfrac{1}{m+1}\left((x^2+y^2)\dfrac{\partial g_{m-1}}{\partial
y}-y\left(x\dfrac{\partial g_{m-1}}{\partial x}+y\dfrac{\partial
g_{m-1}}{\partial y}\right)\right)\vspace{0.2cm}\\
=&-y+\dfrac{x}{m+1}\left(x\dfrac{\partial g_{m-1}}{\partial
y}-y\dfrac{\partial g_{m-1}}{\partial
x}\right):=-y+\dfrac{x}{m+1}\{H_2,g_{m-1}\},\vspace{0.2cm}\\
\dot{y}=&x+\dfrac{y}{m+1}\{H_2,g_{m-1}\}.
\end{array}
\]
Here we consider that the function $g_{m-1}$ is homogenous function
of degree $m-1.$

\smallskip

Now we study the case when $\lambda=1/m.$ From \eqref{first} and
from the equation
 $F=C,$ where $C$  is an arbitrary constant , follows that
\[r^{m-1}=\dfrac{C(m-1)}{2m}\Psi(\vartheta)\pm\sqrt{(C+\left(\dfrac{C(m-1)}{2m}\Psi(\vartheta)\right)^2}:=\Phi(\vartheta),\]
thus in view of the periodicity of $\Psi(\vartheta)$ we get
\[\displaystyle\int_{0}^{2\pi}\dfrac{d\vartheta}{1+\dfrac{(m-1)}{2m}\Phi(\vartheta)\Psi(\vartheta)}=2\pi.\]
Hence the origin is an isochronous center.

\smallskip
 After some computations we can show that differential equations
 \eqref{jj1} for $\lambda=1/m$ becomes
 \begin{equation}\label{center0}
\dot{x}=-\nu\,y+\sigma\,x,\quad \dot{y}=\nu\,x+\sigma\,y,
\end{equation}
 where
\[\nu=1+\dfrac{m-1}{2m}g_{m-1},\quad
\sigma=\dfrac{1}{2m}\{H_2,g_{m-1}\}.
\]
Equation \eqref{center0} in polar coordinates becomes
\[\dot{r}=\dfrac{r}{2m}\dfrac{\partial g_{m-1}}{\partial
\vartheta},\quad \dot{\vartheta}=1+\dfrac{m-1}{2m}g_{m-1}.
\]
\end{proof}

 Below we shall need the
following result.

\begin{proposition}\label{important}
Let $\X=(P,Q)$ be the polynomial vector field  where  $P=P(x,y)$,
and $Q=Q(x,y)$ be the polynomials such that
$\max(\mbox{deg}P,\mbox{deg}Q)=m$ and such that
\begin{equation}\label{xx1}
\displaystyle\int_0^{2\pi}\left.\left(\dfrac{\partial P}{\partial
x}+\dfrac{\partial Q}{\partial
y}\right)\right.|_{x=\cos{t},\,y=\sin{t}}dt=0.
\end{equation}
Then there exist an unique polynomials $\tilde{H}=\tilde{H}(x,y)$
and $\tilde{g}=\tilde{g}(x,y)$ of degree $m+1$ and $m-1$
respectively for which the following relations hold
\begin{equation}\label{xx2}
-\dfrac{\partial \tilde{H}}{\partial y}-y\tilde{g}=P,\quad
\dfrac{\partial \tilde{H}}{\partial x}+x\tilde{g}=Q.
\end{equation}
\end{proposition}
\begin{proof}
The polynomials $P,\,Q,\,\tilde{H}$ and $\tilde{g}$ can be
represented as summa of the homogenous polynomials, ie.
\[\begin{array}{rl}
P=&\displaystyle\sum_{j=1}^mP_j,\quad
Q=\displaystyle\sum_{j=1}^mQ_j,\vspace{0.2cm}\\
\tilde{H}=&\displaystyle\sum_{j=1}^{m+1}\tilde{H}_j,\quad
\tilde{g}=\displaystyle\sum_{j=1}^{m-1}\tilde{g}_j,
\end{array}
\]

{F}rom \eqref{xx2} follows that
\begin{equation}\label{xxx2}
\{H_2,\tilde{g}\}=\dfrac{\partial P}{\partial x}+\dfrac{\partial
Q}{\partial y},
\end{equation}
hence
\[\{H_2,\tilde{g}_j\}=\dfrac{\partial P_j}{\partial
x}+\dfrac{\partial Q_j}{\partial y}.
\]
Thus in view of Theorem \ref{Chetaev1}  and condition \eqref{xx1} we
deduce that there exist an unique homogenous  polynomial
$\tilde{g}=\displaystyle\sum_{j=1}^{m-1}\tilde{g}_j$ such that
\eqref{xx2} holds.

\smallskip

The function $\tilde{H}$ we determine from the equations
\[\dfrac{\partial \tilde{H}}{\partial y}=-y\tilde{g}-P,\quad
\dfrac{\partial \tilde{H}}{\partial x}=-x\tilde{g}+Q,
\]
where $\tilde{g}$ is a solution of the equation \eqref{xxx2}.
\end{proof}
 Proposition \ref{LLloyd} can be generalized as follows.
\begin{proposition}(Generalized weak condition of the center)
Let $X$ and $Y$ be the polynomials of degree at most $m$ i.e.
\[X=\displaystyle\sum_{j=1}^m
X_j,\quad Y=\displaystyle\sum_{j=1}^m Y_j,\]

then the origin is a center of vector field \eqref{RRR} if there
exist $\mu\in\mathbb{R}$ such that
\begin{equation}\label{R188}
(x^2+y^2)\left( \dfrac{\partial X}{\partial x}+\dfrac{\partial
Y}{\partial y} \right)=\mu\left(xX+yY\right),
\end{equation}
and
\[
\displaystyle\int_0^{2\pi}\left.\left(\dfrac{\partial X}{\partial
x}+\dfrac{\partial Y}{\partial y}
\right)\right.|_{x=\cos{t},\,y=\sin{t}}dt=0
\]
\end{proposition}
\begin{proof}
Under the given conditions and in view of Proposition
\ref{important} there exist the polynomial functions $H=H(x,y)$ and
$g=g(x,y)$ of degree $m+1$ and $m-1$ respectively such that the
following relations hold

\begin{equation}\label{ak3}
\begin{array}{rl}
\dot{x}=-y+X=&\{H,x\}+g\{H_2,x\},\vspace{0.2cm}\\
\dot{y}=x+Y=&\{H,y\}+g\{H_2,y\},
\end{array}
\end{equation}

 Hence
\[\dfrac{\partial X}{\partial x}+\dfrac{\partial
Y}{\partial y}=\{H_2,g\},\quad xX+yY=\{H,H_2\},
\]
consequently the condition \eqref{R188} becomes
\[\lambda H_2\{H_2,g\}=\{H,H_2\}\Longrightarrow \{H_2,H+\lambda H_2g\}=0\] Thus
\[H=\left\{
        \begin{array}{ll}
        - \lambda H_2g, & \hbox{if\,$m=2k$,} \\
          - \lambda H_2g+\nu H^{k+1}_2, & \hbox{if\,$m=2k+1$.}
        \end{array}
      \right.
\]
We shall study  the case when $\nu=0.$  Analogously to Proposition
\ref{Lia11} we prove that the vector field $\X$ is
Liapunov-Poincar\'e integrable with the first integral (Darboux
first integral)
\[F=(1+(1-\lambda)g)H^{(\lambda-1)/\lambda}\quad
\mbox{if}\quad\lambda\in\mathbb{R}\setminus\{0,1\},\] and
\[F=H_2e^{-\tilde{g}}\quad \mbox{if}\quad \lambda=1,\]
where $ \tilde{g}=g|_{\lambda=1}.$

\smallskip

 Clearly that if the function $F$
is first integral then the function $F^{\lambda/(\lambda-1)}$ is a
first integral. By considering that the Taylor expansion of this
function at the origin $F^{\lambda/(\lambda-1)}$ and $H_2e^{-g}$ we
obtain
\[H_2+h.o.t.,\]
thus the origin is a center.
\end{proof}
Clearly that differential system \eqref{ak3} with $H=-\lambda\,H_2g$
becomes
\begin{equation}\label{Jjj1}
\begin{array}{rl}
\dot{x}=&-y-\dfrac{\partial\,H}{\partial\,y}-yg=-y+\lambda\,H_2\,\dfrac{\partial\,g}{\partial
y}-y\left((1-\lambda)\,g)\right)\vspace{0.2cm}\\
=&-y+\lambda\,H_2\,\dfrac{\partial\,\Psi}{\partial
y}-y\left((1-\lambda)\,\Psi\right)\vspace{0.2cm}\\
&\lambda\,H_2\,\dfrac{\partial\,g_{m-1}}{\partial
y}-y\left((1-\lambda)\,g_{m-1}\right),\vspace{0.2cm}\\
 \dot{y}=& x+\dfrac{\partial\,H}{\partial
x}+xg=x-\lambda\,H_2\,\dfrac{\partial\,g}{\partial
x}+x\left(1-\lambda)\,g\right)\vspace{0.2cm}\\
=&x+\lambda\,H_2\,\dfrac{\partial\,\Psi}{\partial
x}+x\left((1-\lambda)\,\Psi\right)\vspace{0.2cm}\\
&+\lambda\,H_2\,\dfrac{\partial\,g_{m-1}}{\partial
x}+x\left((1-\lambda)\,g_{m-1})\right),
\end{array}
\end{equation}
where $\Psi=\displaystyle\sum_{j=1}^{m-2}g_j,$ with $g_j$ is
homogenous function of degree $j$ for $j=1,\ldots,m-2.$

We consider the system
\begin{equation}\label{comp1}
\begin{array}{rl}
\dot{x}=-y+\displaystyle\sum_{j=1}^{m-1}X_j+xR_{m-1},\vspace{0.2cm}\\
\dot{y}=x+\displaystyle\sum_{j=1}^{m-1}Y_j+yR_{m-1},
\end{array}
\end{equation}
where $R_{m-1}=R_{m-1}(x,y)$ is a convenient nonzero homogenous
polynomial of degree $m-1.$ Such system are polynomial system with
degenerate infinity. This name is due to the fact that in
Poincar\'{e} compactification of \eqref{comp1} the line at infinity
is filled with critical points (see for instance \cite{soto})
\begin{proposition}\label{deg}
Differential system \eqref{Jjj1} is a polynomial differential system
with degenerate infinity if $\lambda=2/(m+1).$
\end{proposition}
\begin{proof}
From the equation
\[\begin{array}{rl}
&\left.\lambda\,H_2\,\dfrac{\partial\,g_{m-1}}{\partial
y}-y\left((1-\lambda)\,g_{m-1})\right)\right|_{\lambda=2/(m+1)},\vspace{0.2cm}\\
=&\dfrac{1}{m+1}\left((x^2+y^2)\dfrac{\partial g_{m-1}}{\partial
y}-y(m-1)g_{m-1}\right)\vspace{0.2cm}\\
=&\dfrac{1}{m+1}\left((x^2+y^2)\dfrac{\partial g_{m-1}}{\partial
y}-y\left(x\dfrac{\partial g_{m-1}}{\partial x}+y\dfrac{\partial
g_{m-1}}{\partial y}\right)\right)\vspace{0.2cm}\\
=&\dfrac{x}{m+1}\left(x\dfrac{\partial g_{m-1}}{\partial
y}-y\dfrac{\partial g_{m-1}}{\partial
x}\right):=\dfrac{x}{m+1}\{H_2,g_{m-1}\},
\end{array}
\]
Here we consider that the function $g_{m-1}$ is homogenous function
of degree $m-1.$ Hence \eqref{Jjj1} becomes
\begin{equation}\label{Jjj2}
\begin{array}{rl}
\dot{x}=&-y+\dfrac{2}{m+1}\,H_2\,\dfrac{\partial\,\Psi}{\partial
y}-y\left(\dfrac{m-1}{m+1}\,\Psi\right)\vspace{0.2cm}\\
&\dfrac{x}{m+1}\{H_2,g_{m-1}\}:=P_{m-1}+\dfrac{x}{m+1}\{H_2,g_{m-1}\},\vspace{0.2cm}\\
 \dot{y}=&x-\dfrac{2}{m+1}\,H_2\,\dfrac{\partial\,\Psi}{\partial
x}+x\left(\dfrac{m-1}{m+1}\,\Psi\right)\vspace{0.2cm}\\
&+\dfrac{y}{m+1}\{H_2,g_{m-1}\}:=Q_{m-1}+\dfrac{y}{m+1}\{H_2,g_{m-1}\},
\end{array}
\end{equation}
where $P_{m-1}=P_{m-1}(x,y)$ and $Q_{m-1}=Q_{m-1}(x,y)$ are
polynomials of degree $m-1.$ Clearly that $\{H_2,g_{m-1}\}$ is a
polynomial of degree $m-1.$

\smallskip

The first integral (Darboux first integral) in this case is
\[F=\dfrac{H_2^{m-1}}{\left(1+\dfrac{m-1}{m+1}g\right)^2}\]
\end{proof}

 Now we shall study the case when the center is holomophic center.
\begin{proposition}
The origin is a holomorphic center of the polynomial differential
system \eqref{ak3}  if and only if
\begin{equation}\label{r18}
\begin{array}{rl}
\Delta H+2g=&2\Phi,\vspace{0.2cm}\\
(x^2+y^2)g+2mH=&-x\dfrac{\partial H}{\partial x}-y\dfrac{\partial
H}{\partial\,y}+(m+1)H\vspace{0.2cm}\\
&+\displaystyle\int\Phi (x^2+y^2)d(x^2+y^2).
\end{array}
\end{equation}
where $\Delta=\dfrac{\partial^2 }{\partial x\partial
x}+\dfrac{\partial^2 }{\partial y\partial y}$ and $\Phi$ is a
 function such that
\[
\Phi(x^2+y^2)=\left\{
                \begin{array}{ll}
                 0, & \hbox{if $m=2k$,} \\
                  \lambda\,(x^2+y^2)^k, & \hbox{if $m=2k+1$.}
                \end{array}
              \right.
\]

\end{proposition}
\begin{proof}
{F}rom \eqref{CRR} and \eqref{ak3}  we obtain
\[
\begin{array}{rl}
y\dfrac{\partial g}{\partial x}+x\dfrac{\partial g}{\partial
y}=-2\dfrac{\partial^2 H}{\partial y\partial x},\vspace{0.2cm}\\
y\dfrac{\partial g}{\partial y}-x\dfrac{\partial g}{\partial
x}=\dfrac{\partial^2 H}{\partial x\partial x}-\dfrac{\partial^2
H}{\partial y\partial y}.
  \end{array}
\]
Hence
\[\begin{array}{rl}
(x^2+y^2)\dfrac{\partial g}{\partial x}=-2y\dfrac{\partial^2H
}{\partial x\partial
y}-x\dfrac{\partial^2H }{\partial x\partial x}+x\dfrac{\partial^2H }{\partial y\partial y},\vspace{0.2cm}\\
(x^2+y^2)\dfrac{\partial g}{\partial y}=-2x\dfrac{\partial^2H}
{\partial x\partial y}+y\dfrac{\partial^2H}{\partial x\partial
x}-y\dfrac{\partial^2H }{\partial y\partial y}.
\end{array}
\]
Consequently
\[\begin{array}{rl}
(x^2+y^2)\dfrac{\partial g}{\partial x}=&x\Delta H-2\dfrac{\partial
\Psi}{\partial x},\vspace{0.2cm}\\
(x^2+y^2)\dfrac{\partial g}{\partial y}=&y\Delta H-2\dfrac{\partial
\Psi}{\partial y},
\end{array}
\]
where $\Psi=x\dfrac{\partial H}{\partial x}+y\dfrac{\partial
H}{\partial y}-H, $ or equivalently
\[\begin{array}{rl}
\dfrac{\partial}{\partial x} \left((x^2+y^2) g+2\Psi\right)=&x\left(\Delta H+2g\right),\vspace{0.2cm}\\
\dfrac{\partial}{\partial y} \left((x^2+y^2)
g+2\Psi\right)=&y\left(\Delta H+2g\right).
\end{array}
\]
Thus
\[\Delta H+g=2\Phi (x^2+y^2),\quad (x^2+y^2) g+2\Psi=\displaystyle\int\Phi
(x^2+y^2)d(x^2+y^2).\]
 \end{proof}

\begin{corollary}
Is the center is holomorphic center with $H$ homogenous polynomial
of degree $m+1$, then the first integral is
\begin{equation}\label{nn}
F=\left(1+\dfrac{1-m}{m}\,\Delta H\right)H_2^{1-m}.
\end{equation}
\end{corollary}
\begin{proof}
If $H=H_{m+1}$ is a homogenous polynomial of degree $m+1$, then  in
view of the relation
\[\dfrac{\partial H_{m+1}}{\partial\,x}+y\dfrac{\partial
H_{m+1}}{\partial\,y}=(m+1)H_{m+1},\] then by choosing $\Phi=0$ then
from \eqref{r18} follows
\[\Delta H+2g=0,\quad H_2g+mH=0,\]
thus from Proposition \ref{Lia11} we obtain the existence of the
first integral \eqref{nn}. \end{proof}
\section{Isochronous center for polynomial vector fields with degenerate infinity}
In this section we study the existence of isochronous center for
\eqref{comp1} (see Proposition \ref{deg}.).

\smallskip

\begin{proposition}\label{ququ}
Quadratic system with degenerate infinity
\begin{equation}\label{chava56}
\dot{x}=-y+x(ax+by),\quad \dot{y}=x+y(ax+by),
\end{equation}
has at he origin an uniformly isochronous center.
\end{proposition}
\begin{proof}
Indeed, after some computation it is easy to show that the given
system admits the invariant algebraic curves
\[g_1=1+ay-bx=0,\quad g_2=H_2=0,\quad ,\]
with cofactor $K_1=ax+by$ and $K_2=2(ax+by)$  respectively.
Consequently the Darboux first integral is
\[F=\dfrac{\dfrac{x^2+y^2}{2}}{(1+ay-bx)^2},\]
by considering that the following expansion hold
$F=\dfrac{x^2+y^2}{2}+h.o.t.$ we obtain that the origin is a center
and by considering that $x\dot{y}-y\dot{x}=x^2+y^2$ we easily obtain
the proof.

We observe that in view of the relation $\{g_1,g_2\}=ax+by,$ system
\eqref{chava56} can be rewritten as follows
\[\dot{x}=-y+x\{H_2,g_1\},\quad \dot{y}=x+y\{H_2,g_1\}.\]

\end{proof}
\begin{proposition}
cubic system with degenerate infinity
\begin{equation}\label{chava23}
\begin{array}{rl}
\dot{x}=&-y+a(x^2-y^2)+2b\,xy+x(Lx^2+Mxy-Ly^2),\vspace{0.2cm}\\
 \dot{y}=&x-b(x^2-y^2)+2axy+y(Lx^2+Mxy-Ly^2),
\end{array}
 \end{equation}
  has at  the
origin an isochronous center at the origin.
\end{proposition}
\begin{proof}
Indeed, after some computations  we obtain that
\[F=\dfrac{\dfrac{x^2+y^2}{2}}{(1+2ay-2bx+My^2+2Lxy+\kappa
(x^2+y^2))},\] where $\kappa$ is an arbitrary constant, is a firs
integral. In the neighborhood of the origin this function has the
following expansion
\[F=\dfrac{x^2+y^2}{2}+bx^3-ay^3+\dfrac{b}{3}xy^2-\dfrac{a}{3}x^2y+\ldots,\]
thus in view of Poincar\'{e} Theorem we get that the origin is a
center.

\smallskip

Now we prove that this center is isochronous center. In polar
coordinates $x=r\cos\vartheta,\,y=r\sin\vartheta$ the first integral
becomes
\[F=\dfrac{\dfrac{r^2}{2}}{(1+r(2a\sin\vartheta-2b\cos\vartheta)+r^2\left(M\sin^2\vartheta+2L\sin\vartheta
\cos\vartheta+\kappa\right) }=F(r,\vartheta).\] By solving the
equation $F(r,\vartheta)=C,$ where $C$ is an arbitrary constant,
with respect to $r$ we get $r=\Phi(\vartheta).$ On the other hand
from the equation
\[x\dot{y}-y\dot{x}=(1+ay-bx)(x^2+y^2)\Longleftrightarrow
\dot{\vartheta}=1+r(a\sin\vartheta-b\cos\vartheta)=1+\Phi(\vartheta)(a\sin\vartheta-b\cos\vartheta),\]
after the integration and in view of periodicity of
$\Phi(\vartheta)$ we obtain
\[\displaystyle\int_0^{2\pi}\dfrac{d\vartheta}{1+\Phi(\vartheta)(a\sin\vartheta-b\cos\vartheta}=2\pi.\]
Hence the origin is an isochronous center.
\end{proof}
\begin{corollary}
System \eqref{chava23} with $a=b=0$ has an uniformly isochronous
center at the origin.
\end{corollary}
\begin{proof}
Indeed, if $a=b=0$ then the system \eqref{chava23} becomes
\[
\begin{array}{rl}
\dot{x}=&-y+x(Lx^2+Mxy-Ly^2),\vspace{0.2cm}\\
 \dot{y}=&x+y(Lx^2+Mxy-Ly^2),
\end{array}
 \]
Thus $x\dot{y}-y\dot{x}=x^2+y^2\Longleftrightarrow
\dot{\vartheta}=1.$ Hence, by considering that this system admits
the first integral
\[F=\dfrac{\dfrac{x^2+y^2}{2}}{(1+2Lxy+\kappa
(x^2+y^2))},\] which in the neighborhood of the origin has the
following Taylor expansion
\[F=\dfrac{x^2+y^2}{2}+h.o.t.,\]
we deduced that the origin is an uniformly isochronous center.
\end{proof}
\begin{remark}
In the paper \cite{chava} the following results is given.
\begin{theorem}\label{Ch}
The cubic polynomial planar vector field with degenerate infinity
admits an isochronous center at the origin if and only if this
system can be brought to one of the following systems
\begin{itemize}
\item[(a)]
\[\begin{array}{rl}
\dot{x}=&-y+\dfrac{4}{3}x^2-2k_1xy-\dfrac{x}{3}\left(4k_1x^2+3k^2_1xy\right),\vspace{0.2cm}\\
\dot{y}=&x+k_1x^2-\dfrac{16}{3}xy-k_1x^2-\dfrac{y}{3}\left(4k_1x^2+3k^2_1xy\right)
\end{array}
\]
\item[(b)]

\[\begin{array}{rl}
\dot{x}=&-y+\dfrac{16}{3}x^2-2k_1xy-\dfrac{4}{3}y^2 +\dfrac{k_1}{3}x\left(16x^2-3k_1xy-4y^2\right),\vspace{0.2cm}\\
\dot{y}=&x+k_1x^2+\dfrac{8}{3}xy-k_1x^2+\dfrac{k_1}{3}y\left(16x^2-3k_1xy-4y^2\right)
\end{array}
\]
\item[(c)]
\[\begin{array}{rl}
\dot{x}=&-y+x^2-y^2+x\left(c_1x^2+2c_2xy-c_2y^2\right),\vspace{0.2cm}\\
\dot{y}=&x+2xy+y\left(c_1x^2+2c_2xy-c_2y^2\right)
\end{array}
\]
\item[(d)]
\[\begin{array}{rl}
\dot{x}=&-y+a_1x^2-a_1y^2+a_2xy+x\left(2c_2xy+a_1a_2y^2\right),\vspace{0.2cm}\\
\dot{y}=&x+2a_1xy+a_2y^2+y\left(2c_2xy+a_1a_2y^2\right)
\end{array}
\]
\end{itemize}
\end{theorem}
The following question arise. It is possible under linear change of
coordinates $x,y$ and a scaling of time $t$ to transformed
\eqref{chava23} into the one of the four differential system
a),b),c) and d)? If the answer is negative then Theorem \ref{Ch}
give only sufficient conditions.
\end{remark}

\section{Quadratic vector field with non--degenerate center}
{F}rom Poincar\'e- Liapunov's work it is known that such system with
a center are characterized by a finite number of algebraic
independent conditions $D_j=0,$ which are polynomials on the
coefficients of the system. The importance of this result is more
theoretical than practical. In \cite{Stepanov} the following problem
was stated: '' In order  to make an effective use of these
conclusions we must answer to the question: Given that right hand
members of our equations are polynomial of degree $m$, to determine
$\textsc{N}(m)$ such that all the equations $D_j=0$ for
$j>\textsc{N}(m)$ are consequences of such equalities for
$j\leq\textsc{N}(m).$ The problem of characterization of
$\textsc{N}(m)$ is still unsolved."

\smallskip

 For nondegenerate quadratic (see  for instance \cite{Bautin,{Schlomiuk}})
and cubic system  (see  for instance \cite{Malkin,Sibirskii}) the
center-focus problem has
 been solved in terms of algebraic equalities satisfied by the  coefficients
\begin{proposition}
{F}or the quadratic system
\begin{equation}\label{R66}
\begin{array}{rl}
\dot{x}=&-y-\lambda_3x^2+(2\lambda_2+\lambda_5)xy+\lambda_6y^2,\vspace{0.2cm}\\
\dot{y}=& x+\lambda_2x^2+(2\lambda_3+\lambda_4)xy-\lambda_2y^2,
\end{array}
\end{equation}
 the origin is a center  if and only if one of the following
four conditions holds
\[
\begin{array}{rl}
i)\quad \lambda_4=\lambda_5=&0,\vspace{0.2cm}\\
ii)\quad\lambda_2=\lambda_5=&0,\vspace{0.2cm}\\
iii)\quad \lambda_3-\lambda_6=&0,\vspace{0.2cm}\\
iv)\quad \lambda_5=&0,\quad \lambda_4+5(\lambda_3-\lambda_6)=0,\quad \lambda_3\lambda_6-2\lambda^2_6-\lambda^2_2=0\vspace{0.2cm}\\
\end{array}
\]
\end{proposition}
In \cite{Schlomiuk} the following results is proved.
\begin{proposition}
The origin is a center of
\begin{equation}\label{Schl}
\dot{x}=y+ax^2+bxy+cy^2,\quad \dot{x}=-x+kx^2+lxy+my^2,
\end{equation}
if and only if one of the following conditions satisfied:
\begin{itemize}
\item[(i)]
\[
\begin{array}{rl}
(a+c)(b+2m)-(2a+l)(k+m)=&0,\vspace{0.2cm}\\
k(a+c)^3+(l-a)(a+c)^2(k+m)+(m-b)(a+c)(k+m)^2-c(k+m)^3=&0,\vspace{0.2cm}\\
\end{array}
\]
\item[(ii)]
\[2a+l=0,\quad b+2m=0.\]

\item[(iii)]
\[
\begin{array}{rl}
5(a+c)-(2a+l)=0&,\vspace{0.2cm}\\
5(k+m)-(b+2m)=0&,\vspace{0.2cm}\\
c^2+c(a+c)+k^2+k(k+m)=0.
\end{array}
\]
\end{itemize}
\end{proposition}
In the following quadratic differential system the previous two
results is comparing
\begin{proposition}
The origin is a center of
\begin{equation}\label{Schl1}
\begin{array}{rl}
\dot{x}=&-y+ax^2+\dfrac{a(3\lambda-2)}{\lambda}y^2+\dfrac{2\beta(1-\lambda)}{\lambda}xy:=-y*X,\vspace{0.2cm}\\
\dot{y}=&x+\dfrac{\beta(3\lambda-2)}{\lambda}x^2+\dfrac{2a(1-\lambda)}{\lambda}xy+\beta
y^2:=x*Y,
\end{array}
\end{equation}
for $\lambda\in\mathbb{R}\setminus\{0,1\}$ and if $\lambda=1$ then
the origin is a center for
\begin{equation}\label{Schl11}
\begin{array}{rl}
\dot{x}=&-y+a(x^2+y^2),\vspace{0.2cm}\\
\dot{y}=&x++\beta(x^2+y^2),
\end{array}
\end{equation}
where $\kappa$ is an arbitrary constant.
\end{proposition}
\begin{proof}
After some computations it is possible to show that the function
\[F=(1+\dfrac{2(1-\lambda)}{\lambda}(ay-\beta
x))^{\lambda/(\lambda-1)}H_2,\] is a first integral of \eqref{Schl1}
for arbitrary  $\lambda\in\mathbb{R}\setminus\{0,1\}.$  This
function have the following  Taylor expansion in the neighborhood of
the origin
\[F=\dfrac{x^2+y^2}{2}+\beta
x^3-ay^3-\dfrac{a}{3}(x^2y+xy^2)+h.o.t.,\] consequently in view of
Poincar\'{e} Theorem we obtain that the origin is a center.

\smallskip

For $\lambda=1$ the first integral is
$F_1=H_2e^{\beta\,x-ay}=H_2+h.o.t.$ Thus the origin is a center in
this case.
\end{proof}
\begin{corollary}
 Quadratic system \eqref{Schl1}  satisfies  conditions
(i) for arbitrary $\lambda\in\mathbb{R}\setminus\{0\}$ and satisfies
the Bautin conditions only if $\lambda=\dfrac{1}{2}.$
\end{corollary}

\begin{proof}

Indeed, by compare \eqref{Schl1} with \eqref{Schl} we obtain that
\[\begin{array}{rl}
a=a,\quad b=\dfrac{2\beta(1-\lambda)}{\lambda},\quad
c=\dfrac{a(3\lambda-2)}{\lambda},\vspace{0.2cm}\\
l=\dfrac{2a(1-\lambda)}{\lambda},\quad
k=\dfrac{\beta(3\lambda-2)}{\lambda},\quad m=\beta.
\end{array}
\]
Consequently conditions  (i) satisfies identically.

\smallskip

On the other hand if by compare \eqref{Schl1} with \eqref{R66} we
obtain that the unique solution is
\[
\lambda_6-\lambda_3=0,\quad\lambda_2=-\beta,\quad \lambda_3=-a,\quad
\lambda_4=4a,\quad \lambda_5=4\beta,
\] and $\lambda=1/2.$
Differential system \eqref{R66} in this case becomes
\begin{equation}\label{Rgg}
\begin{array}{rl}
\dot{x}=&-y+\dfrac{\lambda_4}{4}\left(x^2-y^2\right)+\dfrac{\lambda_5}{2}xy,\vspace{0.2cm}\\
\dot{y}=&x-\dfrac{\lambda_5}{4}\left(x^2-y^2\right)+\dfrac{\lambda_4}{2}xy,
\end{array}
\end{equation}
which is equivalent
\[\dot{z}=iz+(\lambda_4-i\lambda_5)\dfrac{z^2}{4},\quad z=x+i\,y\]
Consequently the singular points $(0,0)$ and
$\left(\dfrac{4\lambda_5}{\lambda^2_4+\lambda^2_5},\dfrac{-4\lambda_4}{\lambda^2_4+\lambda^2_5}\right)$
are holomorphic center.
\end{proof}
\begin{corollary}
Quadratic system \eqref{Schl1} admits an uniformly isochronous
center if $\lambda=2/3.$
\end{corollary}
\begin{proof}
Follows from Corollary \ref{rrr} and Proposition \ref{ququ}.
\end{proof}
\begin{proposition}
Quadratic vector field for which the function \eqref{a1} is the
first integral
  can be rewritten as follows
\begin{equation}\label{ak}
\dot{x}=\{H_3+H_2,x\}+g_1\{H_2,x\},\quad
\dot{y}=\{H_3+H_2,y\}+g_1\{H_2,y\},
\end{equation}
 with the conditions
\begin{equation}\label{in100}
g_1\{H_2,H_m\}+\{H_2,H_{m+1}\}+\{H_3,H_{m}\}=0\quad \mbox{for}\quad
m=2,3,4,\ldots.
\end{equation}
\end{proposition}
\begin{proof}
Follows direct from Theorem \eqref{Lia1}, with $n=2.$
\end{proof}
We determine the representation \eqref{ak} for differential system
\eqref{R66}.
\begin{proposition}
System \eqref{R66} can be rewritten as follows
\[
\begin{array}{rl}
\dot{x}=&\{H_3+H_2,x\}-(\lambda_5x-\lambda_4y)y\vspace{0.2cm}\\
=&\{H_3+H_2+\dfrac{\lambda_4}{3}y^3+\dfrac{\lambda_5}{3}x^3,x\}-\lambda_5xy,\vspace{0.2cm}\\
\dot{y}=&\{H_3+H_2,y\}+(\lambda_5x-\lambda_4y)x\vspace{0.2cm}\\
=&\{H_3+H_2+\dfrac{\lambda_4}{3}y^3+\dfrac{\lambda_5}{3}x^3,y\}+\lambda_4yx,
\end{array}
\]
where $H_3$ is such that
\[ H_3=\dfrac{1}{3}\left(\lambda_2+\lambda_5\right)x^3+\lambda_3x^2y-\dfrac{1}{3}(\lambda_4+\lambda_6)y^3-\lambda_2xy^2.
\]
\end{proposition}
\begin{proof}
In view of the relation
\[
\displaystyle\int_0^{2\pi}\left.\left(\dfrac{\partial X}{\partial
x}+\dfrac{\partial Y}{\partial y}
\right)\right.|_{x=\cos{t},\,y=\sin{t}}dt\equiv 0
\]
where $X$ and $Y$ are quadratic polynomials given by the formula
\eqref{Schl1}, we deduce that for  Bautin  quadratic system
 we have that
\begin{equation}\label{RR11}
\begin{array}{rl}
-y-\dfrac{\partial\,H_3}{\partial\,y}-yg_1=&-y-\lambda_3x^2+(2\lambda_2+\lambda_5)xy+\lambda_6
y^2,\vspace{0.2cm}\\
x+\dfrac{\partial\,H_3}{\partial\,x}+xg_1=&x+\lambda_2x^2+(2\lambda_3+\lambda_4)xy-\lambda_2
y^2,,\vspace{0.2cm}\\
g_1=&A\,x+B\,y.
\end{array}
 \end{equation}
Hence we obtain
\[\{H_2,g_1\}=\lambda_5 y+\lambda_4 x,\]
thus
\[
g_1=\lambda_4 y-\lambda_5 x.
\]
{F}rom \eqref{RR11} we have
\[
\begin{array}{rl}
\dfrac{\partial H_3}{\partial
y}=\lambda_3x^2-2\lambda_2xy-(\lambda_6+\lambda_4)y^2,\vspace{0.2cm}\\
\dfrac{\partial H_3}{\partial
x}=(\lambda_2-\lambda_5)x^2+2\lambda_3xy-\lambda_2y^2,
\end{array}
 \]
 After integration we deduce the expression for $H_3.$
Thus after some computation we deduce the proof of the proposition.
\end{proof}
By solving the equation \eqref{in100}  for $m=3$ i.e.
\[(\lambda_4y-\lambda_5x)\{H_2,H_3\}+\{H_2,H_4\}=0,\]
we deduce that
\[\begin{array}{rl}
H_4=&\left(\dfrac{\lambda_2\lambda_5}{2}-\dfrac{\lambda^2_2}{4}-
\dfrac{\lambda_4\lambda_6}{2}+\dfrac{\lambda^2_5}{4}\right)x^4+\lambda_4\lambda_2xy^3\vspace{0.2cm}\\
&+\left(-\dfrac{\lambda_2\lambda_5}{2}-\dfrac{\lambda^2_2}{4}-
\dfrac{3\lambda_4\lambda_6}{2}+\dfrac{\lambda^2_5}{4}\right)
+\lambda_5\lambda_6yx^3+a_2(x^2+y^2)^2,
\end{array}
\]
if $\lambda_3-\lambda_6=0,$ and $\lambda_5\ne 0.$ If $\lambda_5=0$
then
\[\begin{array}{rl}
H_4=&\left(-\dfrac{\lambda^2_4}{4}-\dfrac{\lambda_4\lambda_6}{4}-\dfrac{\lambda_4\lambda_3}{4}\right)x^4+\lambda_4\lambda_2xy^3+\vspace{0.2cm}\\
&+(-\dfrac{\lambda^2_4}{2}-\dfrac{\lambda_4\lambda_6}{2}-\lambda_4\lambda_3)x^2y^2
+a_2(x^2+y^2)^2,
\end{array}
\]
Now we deduce  Poincar\'e-- Liapunov  integrable  quadratic systems.
\begin{corollary}
Quadratic system \eqref{R66} is Hamiltonian if and only if
\[\lambda_4=\lambda_5=0,\quad H_j=0\quad\mbox{for}\quad j>3.\]
\end{corollary}
\begin{proof}
Indeed, from \eqref{a3} for $n=2$ follows that $\{H_2,g_1\}=0,$
consequently $g_1=g_1(r).$ By considering that $g_1$ is a homogenous
polynomial of degree one we deduce that $g_1=0.$ Consequently the
first integral is $V(x,y)=H_2+H_3$ consequently $H_j=0$ for $j>3.$
In short the corollary is proved.
\end{proof}

\begin{corollary} (Weak condition of the center for quadratic
system)

\smallskip

System \eqref{R66} is Poincar\'{e}--Liapunov  integrable if one of
the following two condition holds.
\begin{itemize}
\item[{i)}]
\[\lambda=\dfrac{1}{2},\quad\lambda_3-\lambda_6=0, \quad \lambda_3=-\dfrac{\lambda_4}{4}\quad\lambda_2=-\dfrac{\lambda_5}{4},\] and
\item[{ii)}]
\[\lambda_2=\lambda_5=0,\quad \lambda_3=-\dfrac{\lambda\lambda_4}{2},\quad
\lambda_6=\dfrac{3\lambda\lambda_4}{2}-\lambda_4,\] for
$\lambda\in\mathbb{R}\setminus\{0,1/2\}.$
\end{itemize}
\end{corollary}
\begin{proof}
Indeed, from the equation
\[\begin{array}{rl}
0=&H_3+\lambda(\lambda_5x-\lambda_4y)
H_2=\dfrac{1}{3}\left(-\lambda_2+\lambda_5\right)x^3
+\lambda_3x^2y+\dfrac{1}{3}(\lambda_4+\lambda_6)y^3-\lambda_2xy^2\vspace{0.2cm}\\
&+\dfrac{\lambda}{2}(\lambda_5x-\lambda_4y)(x^2+y^2)\vspace{0.2cm}\\
=&\left(\dfrac{2\lambda_2}{3}+\dfrac{\lambda_5}{3}-\dfrac{\lambda\lambda_5}{2}\right)x^3+(\lambda_3+\dfrac{\lambda\lambda_4}{2})yx^2
-(\dfrac{\lambda\lambda_5}{2}+\lambda_2)xy^2\vspace{0.2cm}\\
&+(\dfrac{\lambda\lambda_4}{2}-\dfrac{\lambda_4}{3}-\dfrac{\lambda_6}{3})y^3=\displaystyle\sum_{j+k=3}b_{jk}x^jy^k,
\end{array}
\]
and by requiring that
$b_{jk}=0$ for $j,k=1,2,3$ we obtain the proof.

\smallskip

If $\lambda\in\mathbb{R}\setminus\{0,1/2,1\}$ then the first
integral is
\[F=(1+(\lambda-1)\lambda_4y)H_2^{(\lambda-1)/\lambda}.\]
If $\lambda=1$ then the first integral is
\[F=H_2e^{-\lambda_4y}.\]

If $\lambda=1/2$ the first integral is
\[F=\dfrac{1+1/2(\lambda_5x-\lambda_4y)}{H_2}.\]
After some computations we can show that the quadratic system for
the case when $\lambda=1/2$ can be written as \eqref{Rgg}.
Consequently the origin is isochronous center.
\end{proof}
We observe that if we compute the condition of the existence the
holomorphic center for $m=2$ i.e. $\Delta H_3+2g_1=0$ and
$(x^2+y^2)g_1+4H_3=0$ we obtain
\[\begin{array}{rl}
\Delta H_3+2g_1=&2(\lambda_3-\lambda_6)y=0,\vspace{0.2cm}\\
(x^2+y^2)g_1+4H_3=&\dfrac{1}{\lambda_3}(\lambda_5+4\lambda_2)x^3+(4\lambda_3+\lambda_4)yx^2-(\lambda_5+4\lambda_2)xy^2\vspace{0.2cm}\\
&-\dfrac{1}{3}(\lambda_4+4\lambda_6)y^3=0.
\end{array}
\]
Clearly these conditions hold if and only if
\[\lambda_3-\lambda_6=0,\quad \lambda_2=-\dfrac{\lambda_5}{4},\quad
\lambda_3=-\dfrac{\lambda_4}{4},
\]
under these conditions system \eqref{R66} coincide with system
\ref{Rgg}.

\smallskip

Now we study the representation \eqref{ak} for the quadratic system
\eqref{Schl}. After some computations we can prove that the function
$H_3$ and $g_1$ are
\[
\begin{array}{rl}
H_3=&-\dfrac{2m+b+k}{3}x^3-mxy^2+ax^2y-\dfrac{2a+l+c}{3}y^3,\\
q_1=&(2m+b)x-(2a+c)y.
\end{array}
\]  The weak conditions of the center for
quadratic system \eqref{Schl} produce the quadratic system
\eqref{Schl1}. if $\lambda\in\mathbb{R}\setminus\{0,1\}$ and
\eqref{Schl11} if $\lambda=1.$ Clearly, if $g_1\equiv 0,$ i.e.
$2m+b=0$ and $c+2a=0$ then the system is Hamiltonian.

\smallskip

 Now we study the reversible quadratic system.
\begin{proposition}
The most general reversible quadratic system  with non--degenerated
center, invariant under the transformation $(x,-y-t)\longrightarrow
(x,y,t)$ is
\[\dot{z}=i\left(z+b_{20}z^2+b_{02}\bar{z}^2+b_{11}z\bar{z}\right),\]
or equivalently
\begin{equation}\label{Req}
\dot{x}=-y-2\alpha xy,\quad \dot{y}=x+ry^2+sx^2,
\end{equation}
where
$\alpha=b_{02}-b_{20},\,s=b_{11}+b_{02}+b_{20},\,r=b_{11}-b_{02}-b_{20}.$
\end{proposition}
\begin{proof}
Follows from \eqref{eq1} with $m=2.$  The following relations holds
\[b_{11}=(s+r)/2,\quad
b_{02}=\dfrac{\alpha}{2}+\dfrac{s}{4}-\dfrac{r}{4},\quad
b_{20}=\dfrac{s}{4}-\dfrac{\alpha}{2}-\dfrac{r}{4}.
\]
\end{proof}
\begin{proposition}
Quadratic differential system \eqref{Req} is Poincar\'e--Liapunov
integrable, with the first integral
\[\begin{array}{rl}
F=&\left(y^2+\mu_1\left(4rs\alpha^3(1+2\alpha
s)x^2+4\alpha^2s(2\alpha^2s+2\alpha-r)x-2\alpha^2s-2\alpha+r\right)\right)\cdot\vspace{0.2cm}\\
&(1+2\alpha x)^{2\alpha s}=Const.,
\end{array}
\]
if $r\alpha(1+\alpha s)(1+2\alpha s)\ne 0,$ where
$\mu_1=\dfrac{1}{16r\alpha^4(1+\alpha s)(1+2\alpha s)},$
 and
\[\begin{array}{rl}
F=&\left(y^2+\mu_2\left(4r\alpha^3(1+2\alpha r)x^2-4r\alpha^2(1+2\alpha^2
+2r\alpha^3)x+1+2r\alpha^3+2\alpha^2\right)\right)\cdot\vspace{0.2cm}\\
&(1+2\alpha x)^{2\alpha r}=Const.,
\end{array}
\]
if $1+\alpha s=0$ and $r\alpha(1+\alpha r)(1+2\alpha r)\ne 0,$ where
$\mu_2=\dfrac{1}{16r\alpha^4(1+\alpha r)(1+2\alpha r)}.$ Analogously
we can study the case when $1+2\alpha s=0.$
\end{proposition}

\begin{proof}
After the integration equation \eqref{Req} we deduce the existence
of the given first integral. By considering that the Taylor
development in the neighborhood of the origin is
\[F=\dfrac{1}{2}(x^2+y^2)+h.o.t.,\]
thus we obtain that the given quadratic system is
Poincar\'e--Liapunov integrable.
\end{proof}

\begin{proposition}
The most general  quadratic system with non--degenerated center,
invariant under the transformation $(-x,y,-t)\longrightarrow
(x,y,t)$ is
\[
\dot{z}=iz +\beta_{20}z^2+\beta_{11}z\bar{z}+\beta_{02}\bar{z}^2,
\]
where  $\beta_{jk}$ are real constants, or equivalently
\begin{equation}\label{reqq}
\dot{x}=-y+bx^2+cy^2,\quad \dot{y}=x+\beta\,xy,
\end{equation}
where
\[
 \beta=2(\beta_{20}-\beta_{02}),\,\,\,b=\beta_{20}+\beta_{02}+\beta_{11},\,\,\,
c=-\beta_{20}-\beta_{02}+\beta_{11},
\]
\begin{proof}
Follows from Proposition \ref{rev1} with m=2 and $a_{00}=a_{01}=0.$
\end{proof}
\end{proposition}
\begin{proposition}
Differential system \eqref{reqq} is Poincar\'e--Liapunov integrable.
\end{proposition}
\begin{proof}
Indeed, if $b(b-2\beta)(b-\beta)\ne 0$ and $\beta>0$  then system
\eqref{reqq} admits the first integral
\[F=\dfrac{x^2+\lambda{(bc(b-\beta)y^2+b(2\beta+2c-b)y+2\beta+2c-b)}}{(1+\beta
y)^{b/\beta+1}},\] where $\lambda=\dfrac{1}{b(b-2\beta)(b-\beta)}.$
By developing in Taylor series we obtain the following development
\[F=2x^2+y^2-2(b+c+\beta)y^3+3(b+\beta)(b+2\beta+2c)y^5+\ldots,\]
thus the given quadratic is Poincar\'e--Liapunov integrable.

\smallskip

If $b=2\beta$ then \eqref{reqq} admits the first integral
\[
F=\dfrac{x^2}{1+\beta\,y}+\dfrac{c+\beta}{2\beta^3(1+\beta
y)^2}+\dfrac{2c+\beta}{\beta^3 (1+\beta
y)}+\dfrac{c}{\beta^2}\log(1+\beta y).
\]

The Taylor expansion in the neighborhood of the origin is
\[F=2x^2+y^2-(5\beta+c)y^3+18\beta(2\beta+c)y^4-48\beta^2(5\beta+3c)y^5+\ldots,\]
thus the given quadratic system is Poincar\'e--Liapunov integrable
for $b=2\beta.$

\smallskip

The case when $ b(b-\beta)=0,$ can be studied analogously. Thus the
proposition is proved.
\end{proof}
It is well--known the following results (see for instance
\cite{Loud, Mardesic}).
\begin{proposition} \label{st}
The quadratic vector field
\[\dot{x}=y+X_2(x,y),\quad \dot{y}=-x+Y_2(x,y),\]
has an isochronous center at the origin if and only if
\begin{equation}\label{RLoud}
\begin{array}{rl}
i)\quad\dot{x}=&y(1+x),\quad \dot{y}=-x+y^2,\vspace{0.2cm}\\
ii)\quad\dot{x}=&y(1+x),\quad
\dot{y}=-x+\dfrac{y^2}{2}-\dfrac{x^2}{2}\Longleftrightarrow\vspace{0.2cm}\\
\dot{z}=&i(z-\dfrac{z^2}{2}),\vspace{0.2cm}\\
iii)\quad\dot{x}=&y(1+x),\quad \dot{y}=-x+\dfrac{y^2}{4},\vspace{0.2cm}\\
iv)\quad\dot{x}=&y(1+x),\quad \dot{y}=-x+{2y^2}-\dfrac{x^2}{2},
\end{array}
\end{equation}
\end{proposition}
In \cite{Plesh1} the author  obtained a different identification of
the quadratic system  with an isochronous center at the origin,
based directly on the coefficients, namely (see for instance
\cite{Sibirskii})
\begin{proposition}\label{ST}
The quadratic differential system
\begin{equation}\label{RrrLoud}
\dot{x}=-y+ax^2+bxy+cy^2,\quad \dot{y}=x+kx^2+lxy+my^2,
\end{equation}
has an isochronous center at the origin if and only if one of the
following conditions  is satisfied

i)\quad  $a-l=0,\quad c=k=0,\quad b=m.$

\smallskip

ii)\quad $a-c-l=0,\quad a+c=0,\quad b+k-m=0,\quad k+m=0.$

\smallskip

iii)
\[\begin{array}{rl}
4a+5c-l=0,\quad 3b-6k-4m=&0,\vspace{0.2cm}\\
\alpha (\alpha^2+\gamma^2)+\beta (\beta^2-3\delta^2)=&0,\vspace{0.2cm}\\
\gamma (\alpha^2+\gamma^2)-27(3\beta^2-\delta^2)\delta=&0,\vspace{0.2cm}\\
\end{array}
\]
where $\alpha=b+k-m,\, \beta=-b+3k+m,\, \gamma=-a+c+l,\,
\delta=-a-3c+l.$
\end{proposition}
By solving relations i)  and ii) we obtain the following quadratic
vector field
\[\dot{x}=-y+x(lx+my),\quad \dot{y}=x+y(lx+my),\]
and
\[\dot{x}=-y+c(x^2-y^2)+bxy,\quad
\dot{y}=x+\dfrac{b}{2}(y^2-x^2)-2cxy,\] which is equivalent
\[\dot{z}=iz-(c+i\dfrac{b}{2})z^2,\quad z=x+iy.\]
For the conditions iii) after some computations we obtain the
following quadratic system
\[\begin{array}{rl}
\dot{x}=&-y+\dfrac{3\delta-\gamma}{6}x^2+\dfrac{3\alpha-5\beta}{2}xy+\dfrac{\gamma-\delta}{4}y^2,\vspace{0.2cm}\\
\dot{y}=&x+\dfrac{\alpha+\beta}{4}x^2+\dfrac{7\gamma+9\delta}{12}xy+\dfrac{\alpha-3\beta}{4}y^2,
\end{array}
\]
where $\alpha,\beta,\gamma$ and $\delta$  are constants such that
\[
\alpha(\alpha^2+\gamma^2)+\beta(\beta^2-3\delta^2)=0,\quad
\gamma(\alpha^2+\gamma^2)-27\delta(3\beta^2-\delta^2)=0.
\]

\begin{remark}
From the previous results we deduce the following remarks.
\begin{itemize}
\item[(a)]
 The representation \eqref{ak} for the reversible quadratic system we
obtain with the functions
\[g_1=-2(s+\alpha)x,\quad
H_3=\dfrac{1}{2}(x^2+y^2)+sxy^2+\dfrac{r+2s+2\alpha}{3}x^3.
\]
The solutions of the equation $H_3+\lambda g_1 H_2=0$ in this case
are $s=\dfrac{1}{3}$ and $\alpha=\dfrac{\lambda-1}{3\lambda}.$
\item[(b)]
From Proposition \ref{st} follows that any quadratic differential
system with isochronous center at the origin under linear change of
coordinates $x,y$ and a scaling of time $t$  can be transformed any
reversible quadratic system of the type \eqref{RLoud} which is a
particular case of the system \eqref{Req}.

\item[(c)]
The following question arise, It is possible to show that under
linear change of coordinates $x,y$ and a scaling of time $t$ to
transform quadratic non--reversible differential system \eqref{Rgg}
with two isochronous center and with $\lambda_4\ne 0$ in to the one
of the reversible quadratic system \eqref{RLoud} ?. If the answer is
negative then Proposition \ref{st} and \ref{ST} are not equivalent.
\item[(iii)] Differential system \eqref{RrrLoud} with the conditions
(ii) in complex coordinates becomes
\[
\dot{z}=iz+\dfrac{l-ib}{2}z^2,
\]
which coincide with differential system \eqref{Rgg} if we choose
$\lambda_4=2l$ and $\lambda_5=2b.$
\end{itemize}
\end{remark}

\section{Cubic vector field with non--degenerate center}

Now we apply Theorem \ref{Lia1} to study the  cubic planar vector
field with first integral \eqref{a1}.
\begin{proposition}
Cubic polynomial planar differential system
\[\dot{x}=-y+X_2+X_3=P,\quad \dot{y}=x+Y_2+Y_3,\]
where $X_j=X_j(x,y)$ and $Y_j=Y_j(x,y)$ are homogenous polynomial of
degree $j$ for $j=2,3,$ has a center in the origin if and only if
\begin{equation}\label{inverse0101}
\begin{array}{rl}
\dot{x}=&\{H_2+H_3+H_4,x\}+g_1\{H_2+H_3,x\}+g_2\{H_2,x\},\vspace{0.2cm}\\
\dot{y}=&\{H_2+H_3+H_4,y\}+g_1\{H_2+H_3,y\}+g_2\{H_2,y\},
\end{array}
\end{equation}
where $H_j$ and $g_k$  are homogenous polynomial of degree $j$ and
$k$ respectively, where $j=2,3,4$ and $k=1,2,$ which satisfies the
partial differential equations of first degree
\begin{equation}\label{inverse011}
g_2\{H_2,H_{n-1}\}+g_1\{H_2,H_{n}\}+\{H_2,H_{n+1}\}+\{H_{4},H_{n-1}\}=0,
\end{equation}
{for} $ n=3,4,\ldots.$
\end{proposition}
\begin{proof}
System \eqref{inverse1} and  \eqref{inverse2} for $n=3$  coincide
with \eqref{inverse0101} and \eqref{inverse011}.\end{proof}

\smallskip

We shall study the case when the cubic differential system is the
following
\begin{equation}\label{r66}
\begin{array}{rl}
\dot{x}=&-y+ax^2+by^2+cxy+Ax^3+Bx^2y+Cxy^2+Dy^3:=-y+X,\vspace{0.2cm}\\
\dot{y}=& x+\alpha x^2+\beta y^2+\gamma
xy+Kx^3+Lx^2y+M\,x\,y^2+Ny^3:=x+Y.
\end{array}
\end{equation}
\begin{corollary}
Differential system \eqref{r66} can be rewritten in the form
\eqref{inverse0101} if and only if
\begin{equation}\label{raa11}
3(N+A)+L+C=0
\end{equation}
\end{corollary}
\begin{proof}
If \eqref{inverse0101} and \eqref{raa11} hold then the functions
$g_1,\,g_2$ $H_3$ and $H_4$ are
\begin{equation}\label{raa1}
\begin{array}{rl}
H_3=&\dfrac{\alpha+2\beta+c}{3}x^3-\dfrac{2a+\gamma+b}{3}y^3-ax^2y+\beta\,xy^2,\vspace{0.2cm}\\
H_4=&\dfrac{B+M+K}{4}x^4-\dfrac{D}{4}y^4-\Lambda(x^2+y^2)^2,\vspace{0.2cm}\\
&+\dfrac{M}{2}x^2y^2-Ax^3y-\dfrac{(3A+L+C)}{3}y^3x,\vspace{0.2cm}\\
g_1=&-(2\beta+c)x+(2a+\gamma)y,\vspace{0.2cm}\\
g_2=&-(B+M)x^2+(3A+L)xy+\Lambda\,(x^2+y^2),
\end{array}
\end{equation}
where $\Lambda$ is an arbitrary constant.
\end{proof}

\smallskip

 From \eqref{011}  we deduce the equation
$g_1\{H_2,H_3\}=0.$ By inserting $g_1$ and $H_3$ from \eqref{raa1}
into this equation we obtain
\[
a=0,\quad\mbox{and} \quad \gamma\alpha=0,\quad \gamma (\gamma+b)=0.
\]
Consequently condition \eqref{raa11}, functions $H_3,H_4$ and
$g_1,g_2$ takes the form if $a=0$ and $\gamma=0$
\[
\begin{array}{rl}
0=&N+A+\dfrac{L+C}{3},\quad \beta+b=0,\vspace{0.2cm}\\
H_3=&\dfrac{\alpha+c}{3}x^3-\dfrac{b}{3}y^3,\vspace{0.2cm}\\
H_4=&\dfrac{B+D+M+K+c(\alpha+c)}{4}x^4+\dfrac{D+M}{2}x^2y^2+l_{04}(x^2+y^2)^2\vspace{0.2cm}\\
&-Ax^3y-(A+\dfrac{L+C}{3})xy^3,\vspace{0.2cm}\\
g_1=&-cx,\vspace{0.2cm}\\
g_2=&-(B+M)x^2+(3A+L)xy\vspace{0.2cm}\\
&+\nu_0\,(x^2+y^2),
\end{array}
\]
where $\nu_0=\nu|_{a=0,\,\gamma=0}.$ If $a=0,\,\alpha=0$ and
$\gamma+b=0$ then
\[\begin{array}{rl}
0=&N+A+\dfrac{L+C}{3},\quad \beta+b=0,\vspace{0.2cm}\\
H_3=&0,\vspace{0.2cm}\\
H_4=&\dfrac{B+D+M+K+c^2}{4}x^4+\dfrac{D+M}{2}x^2y^2-Ax^3y\vspace{0.2cm}\\
&-(A+\dfrac{L+C}{3})xy^3+\Lambda(x^2+y^2)^2\vspace{0.2cm}\\
=&\dfrac{B+D+M+K}{4}x^4+\dfrac{D+M}{2}x^2y^2-Ax^3y+Nxy^3+\Lambda(x^2+y^2)^2,\vspace{0.2cm}\\
g_1=&\gamma y,\vspace{0.2cm}\\
g_2=&-(B+M)x^2+(3A+L)xy\vspace{0.2cm}\\
&+\nu_1\,(x^2+y^2),
\end{array}
\]
where $\Lambda$ is an arbitrary constant,
$\nu_0=\nu|_{a=0,\,\alpha=0,\gamma+b=0}.$

\smallskip

\begin{corollary}
Cubic system \eqref{r66} under the condition \eqref{raa11} is
Hamiltonian if and only if
\[c=2a,\quad \gamma=-2a,\quad B+M=L+3A=0,\quad 3N+C=0.\]
\end{corollary}
\begin{proof}
From Proposition \ref{a2} for $n=3$ we obtain
\[\{H_2+H_3,g_1\}+\{H_2,g_2\}=0,\]
holds if
\[c=2a,\quad \gamma=-2a,\quad B+M=L+3A=0,\]
thus in view of  the conditions $0=N+A+\dfrac{L+C}{3},\, \beta+b=0,$
we deduce the proof of the corollary. The Hamiltonian is
\[\begin{array}{rl}
H=&\dfrac{D+K}{4}x^4+\dfrac{D+M}{2}x^2y^2-Ax^3y-\dfrac{C}{3}xy^3\vspace{0.2cm}\\
&+\Lambda(x^2+y^2)^2+\dfrac{\alpha}{3}x^3-\dfrac{b}{3}y^3.
\end{array}
\]
\end{proof}
In order to illustrate the previous results we study the Kukles'
system
\[
\begin{array}{rl}
\dot{x}=&-y,\vspace{0.2cm}\\
\dot{y}=& x+\alpha x^2+\beta y^2+\gamma
xy+Kx^3+Lx^2y+M\,x\,y^2+Ny^3,
\end{array}
\]
\begin{example}
The Kukles' system can be written in the form \eqref{inverse0101} if
$3N+L=\gamma\alpha=0,\,\beta=0$  with
\begin{equation}\label{rhh}
\begin{array}{rl}
H_4=&\dfrac{\gamma^2-M}{2}y^4+Nxy^2+(l_{04}+\dfrac{M-\gamma^2}{4})(x^2+y^2)^2,\vspace{0.2cm}\\
H_3=&\dfrac{\alpha}{3}x^3-\dfrac{\gamma}{3}y^3,\vspace{0.2cm}\\
g_1=&\gamma x,\quad
g_2=My^2+3Nxy+(l_{04}+\dfrac{M-\gamma^2}{4})(x^2+y^2)
\end{array}
\end{equation}
\end{example}
Under the given conditions formula \eqref{inverse0101}  holds with
the functions $H_4,H_3,g_1$ and $g_2$ given in \eqref{rhh}.

\smallskip

\begin{proposition}
The most general  cubic differential system invariant under the
change $(x,-y,-t)\longrightarrow (x,y,t)$ is
\[\dot{z}=i\left(b_{10}z+b_{01}\bar{z}+b_{20}z^2+b_{02}\bar{z}^2+b_{11}z\bar{z}+ b_{30}z^3+b_{03}\bar{z}^3+b_{21}z^2\bar{z}+b_{12}z\bar{z}^2
\right),\] where $b_{jk}\in\mathbb{R},$ or equivalently
\[
\begin{array}{rl}
\dot{x}=&-(b_{10}-b_{01})y  -2(b_{20}-b_{02})yx+(b_{21}+b_{03}-b_{30}+b_{12})y^3\vspace{0.2cm}\\
&+\left(3(b_{30}+b_{03})+b_{21}-b_{12}\right)yx^2,\vspace{0.2cm}\\
\dot{y}=&(b_{10}+b_{01})x++(b_{30}+b_{03}+b_{12}+b_{21})x^3\vspace{0.2cm}\\
&+(b_{20}+b_{02}+b_{11})x^2+(b_{11}-b_{20}-b_{02})y^2\vspace{0.2cm}\\
&+\left(b_{21}+b_{12}-3(b_{30}+b_{03})+b_{21}-b_{12}\right)y^2x
\end{array}
\]
\end{proposition}
\begin{proof}
Follows direct from Proposition \ref{rev1} with $m=3.$
\end{proof}

\begin{proposition}
The most general reversible cubic differential system  invariant
under the change $(x,-y,-t)\longrightarrow (x,y,t)$ is
\[\dot{z}=i\left(\alpha_{10}z+\alpha_{01}\bar{z}+ \alpha_{30}z^3+\alpha_{03}\bar{z}^3+\alpha_{21}z^2\bar{z}+\alpha_{12}z\bar{z}^2
\right)++\beta_{20}z^2+\beta_{02}\bar{z}^2+\beta_{11}z\bar{z},\]
where $\alpha_{jk},\,\beta_{jk}\in\mathbb{R},$ or equivalently
\begin{equation}\label{RRryu}
\begin{array}{rl}
\dot{x}=&-(\alpha_{10}-\alpha_{01})y  +(\beta_{20}+\beta_{02}+\beta_{11})x^2+(\beta_{11}+\beta_{02}-\beta_{20})y^2\vspace{0.2cm}\\
&+(\alpha_{21}+\alpha_{03}-\alpha_{30}+\alpha_{12})y^3+\left(3(\alpha_{30}+\alpha_{03})+\alpha_{21}-\alpha_{12}\right)yx^2,\vspace{0.2cm}\\
\dot{y}=&(\alpha_{10}+\alpha_{01})x+(\alpha_{30}+\alpha_{03}+\alpha_{12}+\alpha_{21})x^3+2(\beta_{20}-\beta_{02})xy\vspace{0.2cm}\\
&+\left(\alpha_{21}+\alpha_{12}-3(\alpha_{30}+\alpha_{03})+\alpha_{21}-\alpha_{12}\right)y^2x\vspace{0.2cm}\\
\end{array}
\end{equation}
\end{proposition}
\begin{proof}
Follows direct from Proposition \ref{rev1} with $m=3.$
\end{proof}
\begin{corollary}
Assume that in \eqref{RRryu} $a_{10}=1$ and $a_{01}=0$ then the
origin is a center.
\end{corollary}

\section{Quasi-homogenous cubic vector field with non--degenerate center}
For non-degenerate quasi-homogenous cubic system  (see  for instance
\cite{Malkin,Sibirskii}) the center-focus problem has
 been solved in terms of algebraic equalities satisfied by the
coefficients.
\begin{proposition}
{F}or the cubic system
\begin{equation}\label{r666}
\begin{array}{rl}
\dot{x}=&y+Ax^3+Bx^2y+Cxy^2+Dy^3:=-y+X,\vspace{0.2cm}\\
\dot{y}=&- x+Kx^3+Lx^2y+M\,x\,y^2+Ny^3:=x+Y,
\end{array}
\end{equation}
 the origin is a center  if and only if one of the following
sets of conditions hold

\[
\begin{array}{rl}
i)\qquad
3A+L+C+3N=&0,\vspace{0.2cm}\\
(3A+L)(B+D+K+M)-2(A-N)(B+M)=&0,\vspace{0.2cm}\\
2(A+N)\left((3A+L)^2-(B+M)^2\right)\vspace{0.2cm}\\
+(3A+L)(B+M)(B+K-D-M)=&0,\vspace{0.2cm}\\
ii)\qquad
3A+L+C+3N=&0,\vspace{0.2cm}\\
2A+C-L-2N=&0,\vspace{0.2cm}\\
B+3D-3K-M=&0,\vspace{0.2cm}\\
B+5D+5K+M=&0,\vspace{0.2cm}\\
(A+3N)(3A+N)-16DK=&0
\end{array}
\]
\end{proposition}

 Now we shall study the quasi-homogenous cubic vector field
\eqref{r666} by applying the results obtained from the solution of
the inverse problem of the center problem.
\begin{proposition}
Differential system \eqref{r666} is Poincar\'e--Liapunov integrable
if
\begin{equation}\label{rll}
\begin{array}{rl} L+C=& A+N=0,\quad C=-\dfrac{N(-2+3\lambda)}{\lambda}, \vspace{0.2cm}\\
D=&\dfrac{(2\lambda-1)}{\lambda}B+\dfrac{(2\lambda-1)r}{\lambda},\vspace{0.2cm}\\
K=&\dfrac{(2\lambda-1)r}{\lambda},\quad M=-\dfrac{(\lambda-1)}{\lambda}B+2\dfrac{(1-2\lambda)r}{\lambda^2},\vspace{0.2cm}\\
\end{array}
\end{equation}
where $\lambda\in\mathbb{R}\setminus\{0\}$ and $r$ is an arbitrary
constant.
\end{proposition}
\begin{proof}
In view of the relation
\[
\displaystyle\int_0^{2\pi}\left.\left(\dfrac{\partial X}{\partial
x}+\dfrac{\partial Y}{\partial y}
\right)\right.|_{x=\cos{t},\,y=\sin{t}}dt=L+C+3(A+N)=0,
\]
where $X$ and $Y$ are polynomials given by the formula \eqref{r666},
and after the integration the equations
\[\begin{array}{rl}
-y-\dfrac{\partial H_4}{\partial y}-yg_2=&-y+Ax^3+Bx^2y+Cxy^2+Dy^3,\vspace{0.2cm}\\
x+\dfrac{\partial H_4}{\partial x}+xg_2=&
x+Kx^3+Lx^2y+M\,x\,y^2+Ny^3
\end{array}
\]
we obtain
\[\begin{array}{rl}
H_4=&-\dfrac{K+B+D+M}{4}y^4-\dfrac{K+B}{2}x^2y^2-Ax^3y+Nxy^3+r(x^2+y^2)^2,\vspace{0.2cm}\\
g_2=&(M+B+K)y^2-(3N+C)xy-(4r-K)(x^2+y^2)
\end{array}
\]
The relation $H_4=-\lambda g_2H_2$ holds if \eqref{rll} and
$3(A+N)+L+C=0$ take place.

\smallskip

Differential system \eqref{r666} under conditions \eqref{rll} takes
the form
\[
\begin{array}{rl}
\dot{x}=&-y-Nx\left(x^2+\dfrac{(3\lambda-2)}{\lambda}y^2\right)+By\left(x^2
+\dfrac{2\lambda-1}{\lambda}y^2\right)\vspace{0.2cm}\\
&+2r\dfrac{(2\lambda-1)(\lambda-1)}{\lambda^2}y^3,\vspace{0.2cm}\\
\dot{y}=&x+Ny\left(y^2+\dfrac{3\lambda-2}{\lambda}x^2\right)
-\dfrac{\lambda-1}{\lambda}Bxy^2\vspace{0.2cm}\\
&+2r\dfrac{2\lambda-1}{\lambda^2}(\lambda\,x^2+y^2),
\end{array}
\]
The function $g_2$ in this case takes the form
\[{g}_2=-\dfrac{1}{\lambda}\left(  2r(y^2)x^2-(B-\dfrac{2}{\lambda}r)y^2-\dfrac{2N}{\lambda}xy
\right).\]
 From Corollary \ref{Lia11} we obtain that this system is Poincar\'e--Liapunov integrable with  the
first integral
\[
F=\left(1+(1-\lambda)\left(\dfrac{B}{\lambda}y^2-\dfrac{2N}{\lambda}xy-
\dfrac{2r}{\lambda}(x^2+\dfrac{\lambda-1}{\lambda}
)\right)\right)(x^2+y^2)^{(\lambda-1)/\lambda},
\] if
$\lambda\in\mathbb{R}\setminus\{0,1\}$ and
\[F=H_2e^{1/2(By^2-2Nxy-2rx^2)},\]
if $\lambda=1.$
\end{proof}
\begin{corollary}
 The quasi-homogenous cubic planar vector field with isochronous  center is
\[
\begin{array}{rl}
\dot{x}=&-y+Nx(3y^2-x^2)+By(x^2-y^2)+\dfrac{4r}{3}y^3,\vspace{0.2cm}\\
\dot{y}=&x+Ny(y^2-3x^2)-2Bxy^2-\dfrac{2r}{3}x(x^2+3y^3).
\end{array}
\]
\end{corollary}
\begin{proof}
Indeed, from Corollary \ref{rrr} with $m=3$ and $\lambda=1/3$ we
obtain the proof.
\end{proof}
\begin{corollary}
 The quasi-homogenous cubic planar vector field with uniformly isochronous  center is
\[
\begin{array}{rl}
\dot{x}=&-y+Nx(x^2-y^2)+Byx^2,\vspace{0.2cm}\\
\dot{y}=&x-Ny(x^2-y^2)+Bxy^2,
\end{array}
\]
\end{corollary}
\begin{proof}
Indeed, from Corollary \ref{rrr} with $m=3$ $\lambda=1/2$ we obtain
the proof.
\end{proof}
\begin{corollary}
 The quasi-homogenous cubic planar vector field with holomorphic center
 is
\begin{equation}\label{rST}
\begin{array}{rl}
\dot{x}=&-y-Nx\left(x^2-3y^2\right)+2ry(3x^2-y^2),\vspace{0.2cm}\\
\dot{y}=&x+Ny\left(y^2-3x^2\right)+2rx(3y^2-x^2),
\end{array}
\end{equation}
or equivalently
\[\dot{z}=iz-(N+2ir)z^3.\]
and uniformly isochronous center is
\begin{equation}\label{rST1}
\begin{array}{rl}
\dot{x}=&-y+x(N(y^2-x^2)+Bxy),\vspace{0.2cm}\\
\dot{y}=&x+y(N(y^2-x^2)+Bxy),,
\end{array}
\end{equation}
respectively
\end{corollary}
\begin{proof}
By solving the equations $3H_4+g_2H_2=0$ and $\Delta H_4+2g_2=0$ we
after some computations we obtain the proof of the corollary. We
observe that the first integral in this case is
\[F=\dfrac{1-4(r(x^2-y^2)+Nxy)}{H^2_2}=Const..\]
The proof of \eqref{rST1} follows from Corollary \ref{rrr} with
$m=3.$
\end{proof}

\smallskip

Now we shall study the reversible cubic system.

Now we shall study the quasi-homogenous reversible  cubic
differential system
\begin{equation}\label{ryu}
\dot{z}=i\left(b_{10}z+b_{01}\bar{z}+ b_{30}z^3+b_{03}\bar{z}^3+b_{21}z^2\bar{z}+b_{12}z\bar{z}^2
\right),
\end{equation}
 where $b_{jk}\in\mathbb{R},$ or equivalently
\begin{equation}\label{Ryu}
\begin{array}{rl}
\dot{x}=&-(b_{10}-b_{01})y+(b_{21}+b_{03}-b_{30}+b_{12})y^3\vspace{0.2cm}\\
&+\left(3(b_{30}-b_{03})+b_{21}-b_{12}\right)yx^2\vspace{0.2cm}\\
\dot{y}=&(b_{10}+b_{01})x+(b_{30}+b_{03}+b_{12}+b_{21})x^3\vspace{0.2cm}\\
&+\left(b_{21}+b_{12}-3(b_{30}+b_{03})+b_{21}-b_{12}\right)y^2x
\end{array}
\end{equation}
By introducing the following notations
\[\begin{array}{rl}
b_{01}=&\dfrac{1}{2}\left(\alpha+a\right),\quad
b_{10}=\dfrac{1}{2}\left(\alpha-a\right),\vspace{0.2cm}\\
 b_{03}=&\dfrac{1}{8}\left(\beta+b-\gamma-c\right),\quad
b_{30}=\dfrac{1}{8}\left(\beta+c-\gamma-b\right),\vspace{0.2cm}\\
b_{12}=&\dfrac{1}{8}\left(\gamma+b+3(\beta+c)\right),\quad
b_{21}=\dfrac{1}{8}\left(\gamma-b+3(\beta-c)\right),
\end{array}
\]we obtain from \eqref{Ryu}  the cubic planar vector field
\begin{equation}\label{R01200}
\dot{x}=y(a+b\,x^2+c\,y^2),\quad\dot{y}=x(\alpha+\beta\,x^2+\gamma\,y^2).
\end{equation}
We assume that
\begin{equation}\label{R101}
c(b\gamma-c\beta)((b-\gamma)^2+4c\beta)\ne 0.
\end{equation}
It is easy to prove that system \eqref{R01200} is invariant under
the change $(x,-y,-t)\longrightarrow (x,y,t)$ and under the change
$(-x,y,-t)\longrightarrow (x,y,t).$

\smallskip

 The following corollary was proved in \cite{RS22,LR0}.
\begin{corollary}
System \eqref{R01200} under the condition \eqref{R101} has the
following properties:
\begin{itemize}
\item[(i)] admits two invariant conics  (eventually imaginary)
\[g_j=\nu_j(x^2-\lambda_1)-(y^2-\lambda_2)=0,\quad\mbox{for}\quad
j=1,2,\] where $\nu_1$ and $\nu_2$ are the roots of the polynomial
$P(\nu)=c\nu^2+(b-\gamma)\nu-\beta,$ and
\[\lambda_1=\dfrac{\gamma\,a-\alpha\,c}{{b\gamma-c\beta}},\quad
\lambda_2=\dfrac{\alpha\,b-\beta\,a}{{b\gamma-c\beta}}.\] In view of
\eqref{R101} we obtain that $\nu_1-\nu_2\ne 0,$
\item[(ii)] admits the first integral
$$F(x,y)=\dfrac{\left(\nu_1(x^2-\lambda_1)-y^2+\lambda_2\right)^{\gamma-
\nu_2c}}{\left(\nu_2(x^2-\lambda_1)-y^2+\lambda_2\right)^{\gamma-\nu_1c}};$$
\item[(iii)] the solutions can be represented
as follows
\[\begin{array}{rl}
x^2=&\lambda_1+X(\tau ,x_0,y_0)\\
y^2=&\lambda_2+Y(\tau ,x_0,y_0)\vspace{0.2cm}\\
t=&t_0+\displaystyle\int_{0}^\tau \frac{d\tau}
{\sqrt{(\lambda_1+X(\tau ,x_0,y_0))(\lambda_2+Y(\tau ,x_0,y_0)}}
\end{array}
\]
where $X,\,Y$ are solutions of the linear differential equation of
the second order with constants coefficients
$$
T^{''}-(\gamma+b)T^{'}+ (b\gamma-c\beta)T=0,\quad \mbox{where}\quad
^{'}\equiv{\frac{d}{d\tau}}
$$
\end{itemize}
\end{corollary}
\begin{proposition}
 Differential system \eqref{R01200} is Poincar\'e--Liapunov integrable if and only $a\alpha<0.$
\end{proposition}
\begin{proof}
We shall study the following three cases.

\smallskip

First we assume that \eqref{R101} holds and
\[\Lambda=2(\nu_1-\nu_2)\left(\lambda_2-\nu_1\lambda_1\right)^{\gamma-\nu_2c-1}\left(\lambda_2-\nu_2\lambda_1\right)^{\gamma-\nu_2c-1}\ne 0.\]
 By developing the function
\[
G(x,y)=\dfrac{(\nu_1(x^2-\lambda_1)-y^2+\lambda_2)^{\gamma-
\nu_2c}}{(\nu_2(x^2-\lambda_1)-y^2+\lambda_2)^{\gamma-\nu_1c}}
\]
in Taylor series we obtain
\[\dfrac{G}{\Lambda}=\left(\alpha x^2-ay^2+\kappa_1x^4+\kappa_2y^4+\ldots\right)\]
and $\kappa_1,\,\kappa_2$ are convenient constants. Hence, by
supposing that $a>0$ we get
\[\alpha x^2-ay^2+\kappa_1x^4+\kappa_2y^4+\ldots=a\Lambda\left(-q
x^2+y^2+\kappa_1x^4+\kappa_2y^4+\ldots\right)>0,\] in the
neighborhood of the origin where $q=\dfrac{\alpha}{a}<0.$

\smallskip

 Second  We assume that $c=0$ and $\gamma(\gamma-b)\ne 0.$  For simplicity we shall assume
that $a=1,\,\alpha=-1$ and $b>0.$  After some computations we can
prove that the function
\[
F=\dfrac{y^2+\dfrac{-(\gamma-b)+\beta(a+\gamma x^2)}{\gamma^2-\gamma
b}}{^(1+bx^2)^{\gamma/b}},
\]
 is a first integral of the system
\[
\dot{x}=y(1+b\,x^2),\quad\dot{y}=x(-1+\beta\,x^2+\gamma\,y^2).
\]
By developing the function $F$ in Taylor series in the neighborhood
of the origin we obtain that
\[F=2\left(x^2+y^2+\kappa_1x^4+\kappa_2y^4+\ldots\right),
\]
 where $\kappa_1$ and $\kappa_2$ are conveniens constants.

\smallskip

Now  we study the case when $c=0$ and  $b=\gamma\ne 0.$  The
function
\[
F=\dfrac{y^2-\dfrac{1}{b}+\dfrac{\beta}{b^2}}{1+b
x^2}-\dfrac{\beta}{b^2}\log (1+bx^2),
\]
 is a first integral of the system
\begin{equation}\label{gasull}
\dot{x}=y(1+b\,x^2),\quad\dot{y}=x(-1+\beta\,x^2+b\,y^2).
\end{equation}
By developing the function $F$ in Taylor series in the neighborhood
of the origin we obtain that
\[
F=\dfrac{2}{1}\left( x^2+y^2+\kappa_1x^4+\kappa_2y^4+\ldots\right),
\]
where $\kappa_1$
and $\kappa_2$ are conveniens constants.

\smallskip

 Finally, For the case when $c=0$ and $\gamma=0$ we have that the function
\[
F=y^2-\dfrac{\beta}{b}x^2+\dfrac{(\beta+ b)}{b^2}\log(1+bx^2),
\]
is a first integral of the system
\[
\dot{x}=y(1+b\,x^2),\quad\dot{y}=x(-1+\beta\,x^2).
\]
The Taylor expansion of $F$ in the neighborhood of the origin is
\[
F=\dfrac{2}{a}\left( x^2+y^2+\kappa_1x^4+\kappa_2y^4+\ldots\right),
\]
where $\kappa_1$ and $\kappa_2$ are conveniens constants. Thus the
proposition is proved.
\end{proof}

 Now we shall study the
particular case when \eqref{R01200} is such that
\begin{equation}\label{HH}
\begin{array}{rl}
\dot{x}=&y(\lambda\,b+p+x^2(\lambda+b-2a)+y^2(\lambda-b)),\\\dot{y}=&x(-\lambda
a-p-x^2(\lambda-a)-y^2(\lambda+a-2b)),
\end{array}
\end{equation}
 where $\lambda,b,a,p$ are real parameters and  $b-a\ne 0.$
 Then it is easy to obtain that \[
\nu_1=-\dfrac{\lambda-a}{\lambda-b},\quad \nu_2=-1.\] Thus the
invariant curves of the differential system  are
$$g_1=-\left(y^2+x^2+\lambda\right)=0,\quad g_2=-\left({(\lambda-a)x^2+(\lambda-
b)y^2+\frac{1}{2}(\lambda^2+p)}\right)=0.$$ The first integral $F$
takes the form
$$F(x,y)=\frac{(y^2+x^2+\lambda)^2}{(\lambda-a)x^2+(\lambda-
b)y^2+\frac{1}{2}(\lambda^2+p)}.$$ Consequently all trajectories
\index{trajectories} of \eqref{HH}  are algebraic curves
\begin{equation}\label{RH1}(x^2+y^2)^2+A(K)\,x^2+B(K)\,y^2+
P(K)=0,\end{equation}
 where $F(x,y)=K$ are the
level curves , and
$$
A((K)=2(\lambda-\frac{K}{2}(\lambda-a)),\quad
B(K)=2(\lambda-\frac{K}{2}(\lambda-b)),\quad\,P(K)=\lambda^2-\frac{1}{2}K(\lambda^2+p).$$

\smallskip

It is interesting to observe that a particular case of the family of
planar curves , which is the locus of point $(x,y)$ the product of
whose distance from the fixed points $(0,-c)$ and $(0,c)$ has the
constant value $\kappa^2-c^2$ (for more details see \cite{Basset}).
The quartic equation of this curve is

\begin{equation}\label{LM}
(x^2+y^2)^2+2c^2(x^2-y^2)=\kappa^2\left(\kappa^2-2c^2\right),
\end{equation}
which is equivalent to
\begin{equation}\label{LM1}
(z^2+c^2)\overline{(z^2+c^2)}=(\kappa^2-c^2)^2.
\end{equation}
Thus first, if $A(K)=-B(K)=2c^2,$ and
$P(K)=-\kappa^2\left(\kappa^2-2c^2\right),$ then if $K\ne 2,$ we
obtain that
\[
K=\dfrac{4c^2}{a-b},\quad
p=\dfrac{(a-b)(\kappa^2-c^2)^2}{2c^2}+\dfrac{(c^2-2ab)c^2}{a-b-2c^2},\quad
\lambda=\dfrac{(a+b)c^2}{2c^2-a-b},
\]
 and second,  if $K=2$ then
\[
A(2)=2a,\quad B(2)=2b,\quad P(2)=-p,\quad a=-b=c^2,\quad p=\kappa^2
\left(\kappa^2-2c^2\right),
\]
for arbitrary $\lambda.$

\smallskip

For the first case system \eqref{HH}takes the form
\[
\begin{array}{rl}
x'=&-y((a-b)\left((2a-b-3c^2)x^2+(c^2+b)y^2\right)\vspace{0.2cm}\\
&+p\Big(2c^2+b-a)+c^2b(a+b)\Big),\vspace{0.2cm}\\
y'=&x((a-b)\left((a-c^2)x^2+(3c^2+2b-a)y^2\right)\vspace{0.2cm}\\
&+p\Big(2c^2+b-a)+c^2a(a+b)\Big),
\end{array}
\]
where$'=\dfrac{d}{d\tau},$ with $t=(a-b-2c^2)\tau.$ This
differential system admits as trajectories the family of lemniscate
\eqref{LM}.

\smallskip

For the second case we obtain that the differential system
\eqref{HH} in complex coordinates,  takes the form
\[
\dot{z}=i\left(\kappa^2(2c^2-\kappa^2)\,z+c^2\,z^3-\lambda\,\bar{z}(c^2+z^2)\right).
\]
and admits the family of lemniscate \eqref{LM1}. In particular if
$c=1$ then we obtain the system
\[
\dot{z}=i\left((1-C^2)\,z+\,z^3-\lambda\,\bar{z}(1+z^2)\right).
\]
The bifurcation diagrams  of this differential system in the plane
$(C={|\kappa^2-1|},\lambda)$ are given in \cite{RS22}. Now we study
the case when $\lambda=0,$ i.e. we shall study the differential
equation for which the origin is holomorphic center
\[
\dot{z}=i\left((1-C^2)\,z+\,z^3\right).
\]
The trajectories of this equation are given by the formula
\eqref{RH1} and are the lemniscates given by
\[
(z^2+1)(\bar{z}^2+1)=\left(\kappa^2-1\right)^2=C^2.
\]

\smallskip

Now we study the following particular case of system \eqref{R01200}
\[\dot{x}=y\left(a+(r-q)x^2+y^2\right),\quad
\dot{y}=x\left(\alpha-(p^2+q^2)x^2+(r+q)y^2\right).\] It is easy to
show that the roots $\nu_1$ and $\nu_2$ are $\nu_1=q+ip$ and
$\nu_2=q-ip.$ Thus the invariant curves are complex
\[g_1=(q+ip)(x^2-\lambda_1)-(y^2-\lambda_2)=0,\quad
g_2=(q-ip)(x^2-\lambda_1)-(y^2-\lambda_2)=0.\] Hence the first
integral $F$ is
\[\begin{array}{rl}
F(x,y)=&\left((q(x^2-\lambda_1)-(y^2-\lambda_2))^2+p^2(x^2-
\lambda_1)^2\right)\cdot\vspace{0.2cm}\\
&\exp2\arctan\left(\dfrac{p(x^2-\lambda_1)}{q(x^2-\lambda_1)-(y^2-\lambda_2)}\right).
\end{array}
\]
Particular case of \eqref{ryu} is
\[\dot{z}=i\left(z+b_{20}z^2+b_{30}z^3\right),\]
or equivalently
\[\begin{array}{rl}
\dot{x}=&-y-2b_{20}xy+3b_{30}yx^2-3b_{30}y^3,\vspace{0.2cm}\\
\dot{y}=&x+b_{20}x^2-b_{20}y^2++b_{30}x^3-3b_{30}y^2x,
\end{array}
\]
clearly the origin is holomorphic center for this cubic system
system,

\smallskip

It is well known the following results (see for \cite{Loud}).
\begin{theorem}\label{rst}
The quasi-homogenous cubic differential system \eqref{r666} has an
isochronous center at the origin if and only if a linear change of
coordinates $x,y$ and a scaling of time $t$ transform cubic
differential system to one of the systems
\begin{equation}\label{rLoud1}
\begin{array}{rl}
i)\quad\dot{x}=&y(1+x^2),\quad \dot{y}=-x(1-y^2),\vspace{0.2cm}\\
ii)\quad\dot{x}=&y(1-3x^2+y^2),\quad \dot{y}=-x(1
+{3y^2}-{x^2}) \vspace{0.2cm}\\
&\Longleftrightarrow\dot{z}=i\left( z+z^3 \right),\vspace{0.2cm}\\
iii)\quad\dot{x}=&y(1+9x^2-2y^2),\quad \dot{y}=-x(1-{3y^2}),\vspace{0.2cm}\\
iv)\quad\dot{x}=&y(1-9x^2+2y^2),\quad \dot{y}=-x(1+{3y^2}).
\end{array}
\end{equation}
\end{theorem}
The following result was proved in \cite{Plesh}
\begin{theorem}\label{rst1}
The origin is an isochronous center of \eqref{r666} if and only if
one of the following sets of conditions is satisfied:
(i)
\[
 A+C=0,\quad A-L=0,\quad A+N=0,B-M=0,\quad D=K=0,
\]
(ii)
\[\begin{array}{rl}
 3A+C=&0,\quad 3A-L=0,\quad A+N=0, B+3D=0,\\
B+3K=&0,\quad B-M=0,
\end{array}
\]
(iii)
\[\begin{array}{rl}
 3A+L+C+3N=0,\quad 9A-5L+2C-9N=&0,\\
B+3D-3K-M=0,\quad B+6D+6K+M=&0,\\
(3A+7N)(7A+3N)-100DK=&0,\\
(A+N)\left((3A+L)^2-(B+M)^2\right)-2(3A+L)(B+M)(D-K)=&0
\end{array}
\]
\end{theorem}
\begin{remark}
From the previous results we have the following remarks.

\begin{itemize}
\item[(i)]
 From Proposition \ref{rst} follows that
any cubic quasi-homogenous differential system with isochronous
center at the origin under linear change of coordinates $x,y$ and a
scaling of time $t$ can be transformed any reversible cubic system
of the type \eqref{rLoud1} which is a particular case of the system
\eqref{ryu}.
\item[(ii)]
The following question arise, It is possible to show that under
linear change of coordinates $x,y$ and a scaling of time $t$ to
transform cubic non--reversible differential system \eqref{rST} with
two isochronous center and with $N\ne 0$ in to the one of the
reversible cubic quasi-homogenous differential system \eqref{rLoud1}
? If the answer is negative then Theorem \ref{rst} and \ref{rst1}
are not equivalent.
\item[(iii)] System \eqref{r666} under the conditions (ii) of
Theorem \ref{rst1} in complex coordinates becomes
\[\dot{z}=iz+(A-i\dfrac{B}{3})z^3,\] which coincide with the system \eqref{rST}
if we take $N=-A$ and $r=\dfrac{B}{6}.$
\item[(iv)] In \cite{Gasull} the authors proposed the following
results.
\begin{theorem} The origin is an isochronous center of the cubic
system
\[\dot{z}=iz+Dz^3+Ez^2\bar{z}+Fz\bar{z}^2+G\bar{z}^3,\]
if and only if one of the following three relations is satisfied
\[\begin{array}{rl}
b1:&\qquad E=F=G=0,\vspace{0.2cm}\\
   b2:&\qquad  D=\bar{F},\quad E=G=0,\vspace{0.2cm}\\
 b3:&\qquad D=\dfrac{7}{3}\bar{F},\quad E=0,\qquad
 |G|^2=\dfrac{16}{9}|F|^2,\qquad \Re(\bar{F}^2G)=0
\end{array}
\]
\end{theorem}
The solutions of conditions b3 produces the following cubic
systems
\begin{itemize}
\item[(1)]
The cubic system which satisfy condition b3) are the following.
\[\dot{x}=-y(1+\dfrac{3K}{2}x^2),\quad
\dot{y}=x(1+K(x^3-\dfrac{9}{2}y^2),\] which is a particular case of
\eqref{gasull}.
\item[(2)]
and the cubic system
\[
\begin{array}{rl}
\dot{x}=&-y+\mu\left(10A\,x^3-6(N+4A)xy^2\right)\vspace{0.2cm}\\
&+(N+9A)(7N+3A)y^3+\left(9N^2-198NA-111A^2\right)xy^2,\vspace{0.2cm}\\
\dot{y}=&x+\mu\left(
6(4N+A)x^2y-\dfrac{90NA}{N+9A}y^3\right)\vspace{0.2cm}\\
&+(3N+7A)(A+9N)x^3-\left(111N^2+9A^2-198AN\right)xy^2
\end{array}
\]
where $\mu=\sqrt{-\dfrac{A+9N}{N+9A}}\in\mathbb{R},$ satisfy the
conditions b3).
\end{itemize}

\end{itemize}
\end{remark}

\section{Quartic  differential system with non--degenerate center}
In this section we solve the inverse problem of the center for
quartic planar differential equations.
 \begin{proposition}
Quartic differential system
\begin{equation}\label{ttt1}
\begin{array}{rl}
\dot{x}=&-y+ax^2+by^2+cxy+Ax^3+Bx^2y+Cxy^2+Dy^3\vspace{0.2cm}\\
&+L_1x^4+L_2x^3y+L_3x^2y^2+L_4xy^3+L_5y^4\vspace{0.2cm}\\
:=&-y+X,\vspace{0.2cm}\\
\dot{x}=&x+\alpha x^2+\beta y^2+\gamma
xy+Kx^3+Lx^2y+Mxy^2+Ny^3\vspace{0.2cm}\\
&+K_1x^4+K_2yx^3+K_3x^2y^2+K_4x^3y+K_5y^4\vspace{0.2cm}\\
:=&x+Y,
 \end{array}
\end{equation}
can be rewritten as follows
\begin{equation}\label{rrrr1}
\begin{array}{rl}
\dot{x}=&\{H_2+H_3+H_4+H_5,x\}+g\{H_2,x\},\vspace{0.2cm}\\
\dot{y}=&\{H_2+H_3+H_4+H_5,y\}+g\{H_2,y \},
\end{array}
\end{equation}
if and only if
\[
\displaystyle\int_0^{2\pi}\left.\left(\dfrac{\partial
X}{\partial x}+\dfrac{\partial Y}{\partial
y}\right)\right.|_{x=\cos{t},\,y=\sin{t}}dt=3(N+A)+L+C=0
\]
 where
$g=g(x,y)$ is a polynomial of degree $3$ and $H_j=H_j(x,y)$ is a
homogenous polynomial of degree $j$ for $j=2,3,4,5.$
\end{proposition}

\begin{proof} Indeed,under the given conditions on the parameters, the functions
$H_3,H_4,H_5$ and $g$ are determined as follows
\[
\begin{array}{rl}
H_3=&\dfrac{c+\alpha+2k_{21}}{3}\,x^3-\dfrac{\gamma+b+2a}{3}\,y^3-ax^3y+k_{21}xy^2,\vspace{0.2cm}\\
H_4=&\dfrac{M+B+K+D}{4}\,y^4-\dfrac{B+K}{2}x^2y^2-Ax^3y\vspace{0.2cm}\\
&-\dfrac{L+3A+C}{3}\,xy^3-\Lambda(x^2+y^2)^2,\vspace{0.2cm}\\
H_5=&\dfrac{2K_3+3(K_1+L_2)+2L_4+8L_1+2K_2}{15}\,x^5-\dfrac{3(L_5+K_4)+8L_1+2(K_2+L_3)}{15}\,y^5\vspace{0.2cm}\\
&+K_5xy^4-L_1x^4y-\dfrac{4L_1+K_2+L_3}{3}\,x^2y^3+\dfrac{4K_5+K_3+L_4}{3}x^3y^2,\vspace{0.2cm}\\
g=&-(2k_{21}+c)x+(\gamma+2a)y+By^2+(L+3A)xy+(K-4\Lambda)(x^2+y^2)\vspace{0.2cm}\\
&(4L_1+K_2)yx^2-(4K_5+L_4)xy^2-\dfrac{( 2K_3+2L_4+3L_2+8K_5)}{3}x^3\vspace{0.2cm}\\
&+\dfrac{3K_4+2L_3+8L_1+2K_2}{3}\,y^3,
\end{array}
\]
where $\Lambda$ is an arbitrary constant.
\end{proof}
\begin{corollary}
Differential system \eqref{ttt1} is Hamiltonian if and only if
$g=(K-4\Lambda_2)(x^2+y^2)$ (see Corollary \ref{cor1}) i.e.
\begin{equation}\label{rquar}
\begin{array}{rl}
2k_{21}+c=&0,\quad \gamma+2a=0,\quad 3K_4+2L_3+8L_1+2K_2=0, \vspace{0.2cm}\\
 4L_1+K_2=&0,\quad 3(N+A)+L+C=0,\quad  4K_5+L_4=0,\vspace{0.2cm}\\
L+3A=&0,\quad 2K_3+2L_4+3L_2+8K_5=0,
\end{array}
\end{equation}
\end{corollary}
\begin{proof}
From the divergence condition we get that $\{H_2,g\}=0$ follows
$g=g(H_2)$ In view of \eqref{rquar} we deduce that
$g=(K-\Lambda_2)(x^2+y^2).$ Thus differential system \eqref{rrrr1}
becomes

\[
\begin{array}{rl}
\dot{x}=&\{H_2+\tilde{H}_3+\tilde{H}_4+\tilde{H}_5,x\}+2(K-\Lambda_2)H_2\{H_2,x\}\vspace{0.2cm}\\
=&\{H_2+\tilde{H}_3+\tilde{H}_4+\tilde{H}_5+{(K-\Lambda_2)}H_2^2,x\}:=\{H,x\}\vspace{0.2cm}\\
 \dot{y}=&\{H_2+\tilde{H}_3+\tilde{H}_4+\tilde{H}_5,y\}+2(K-\Lambda_2)H_2\{H_2,y\}\vspace{0.2cm}\\
 =&\{H_2+\tilde{H}_3+\tilde{H}_4+\tilde{H}_5+{(K-\Lambda_2)}H_2^2,y\}:=\{H,y\},
\end{array}
\]
where $\tilde{H}_j$ correspond to the value of $H_j$ under the
conditions \eqref{rquar}, for $j=3,4,5.$  Hence the system is
Hamiltonian with Hamiltonian
$H=H_2+\tilde{H}_3+\tilde{H}_4+\tilde{H}_5+{(K-\Lambda_2)}H_2^2.$

\end{proof}
\begin{proposition}
The  quartic differential system
\[
\begin{array}{rl}
\dot{x}=&-y+ax^2+\dfrac{a(3\lambda-2)}{\lambda}y^2-\dfrac{2\kappa\,(\lambda-1)}{\lambda}xy\vspace{0.2cm}\\
& Ax^3+Bx^2y+\dfrac{A(3\lambda-2)}{\lambda}xy^2+L_1x^4+L_3x^2y^2\vspace{0.2cm}\\
&+\dfrac{((4\lambda^2-6\lambda+2)\Lambda+\lambda(2\lambda-1)B}{\lambda^2}y^3-\dfrac{2K_5(2\lambda-1)}{\lambda}xy^3\vspace{0.2cm}\\
&-\dfrac{(15\lambda^2-16\lambda+4)L_1-\lambda(5\lambda-2)L_3}{3\lambda^2}y^4\vspace{0.2cm}\\
&-\dfrac{2((5\lambda-2)K_5+\lambda(\lambda-1)K_3)}{3\lambda^2}x^3y,\vspace{0.2cm}\\
\dot{y}=&x+\dfrac{\kappa(3\lambda-2)}{\lambda}x^2+\kappa
y^2-\dfrac{2a(\lambda-1)}{\lambda}xy\vspace{0.2cm}\\
&+\dfrac{2\Lambda(2\lambda-1)}{\lambda}x^3-\dfrac{A(3\lambda-2)}{\lambda}x^2y\vspace{0.2cm}\\
&-\dfrac{(\lambda(\lambda-1)B-2(2\lambda-1)\Lambda)}{\lambda^2}xy^2-Ay^3+K_5y^4\vspace{0.2cm}\\
&+K_3x^2y^2-\dfrac{\lambda(\lambda-1)L_3(5\lambda-2)L_1}{3\lambda^2}xy^3-\dfrac{2(2\lambda-1)L_1}{\lambda}x^3y\vspace{0.2cm}\\
&\dfrac{(16\lambda-4-15\lambda^2)K_5+\lambda(5\lambda-2)K_3}{3\lambda^2}x^4
\end{array}
\]
where $\kappa$ and $\Lambda$ are arbitrary constants,
 is Poincar\'e--Liapunov integrable
 if $\lambda\in\mathbb{R}\setminus\{0\}.$
\end{proposition}
\begin{proof}
Indeed, after some computations we can prove that the function
$F=(1+(1-\lambda)g)H_2^{(\lambda-1)/\lambda},$ is a first integral
if $\lambda\in\mathbb{R}\setminus\{0,1\},$ where
\[
\begin{array}{rl}
g=&-\dfrac{2((3\lambda-2)L_1-\lambda\,L_3)}{3\lambda^2}y^3-\dfrac{2((2-3\lambda)K_5+\lambda\,K_3)}{3\lambda^2}x^3\vspace{0.2cm}\\
&\dfrac{2L_1}{\lambda}x^2y-\dfrac{2K_5}{\lambda}xy^2-\dfrac{2\Lambda}{\lambda}x^2+\dfrac{(2\lambda-2)\Lambda+\lambda
B}{\lambda^2}y^2\vspace{0.2cm}\\
&\dfrac{2A}{\lambda}xy-\dfrac{2\kappa}{\lambda}x+\dfrac{2a}{\lambda}y,
\end{array}
\]
and $F=H_2e^{-\tilde{g}}$ if $\lambda=1,$ where
$\tilde{g}=g|_{\lambda=1}.$
\end{proof}

\begin{remark}
After tedious computations it is possible to show that any
differential system \eqref{ttt1} can be rewritten as follows
\[
\begin{array}{rl}
\dot{x}=&\{H_2+H_3+H_4+H_5,x\}+g_1\{H_2+H_3+H_4,x\}\vspace{0.2cm}\\
&+g_2\{H_2+H_3,x\}+g_3\{H_2,x\},\vspace{0.2cm}\\
\dot{y}=&
\{H_2+H_3+H_4+H_5,y\}+g_1\{H_2+H_3+H_4,y\}\vspace{0.2cm}\\
&+g_2\{H_2+H_3,y\}+g_3\{H_2,y\},
\end{array}
\]

\end{remark}

\section{Quartic quasi-homogenous vector field with non--degenerate center}
We shall study the polynomial planar differential system of degree
four of the type
\begin{equation}\label{ff}
\begin{array}{rl}
\dot{x}=&-y+L_{40}x^4+L_{04}y^4+L_{22}x^2y^2+L_{13}xy^3+L_{31}x^3y\vspace{0.25cm}\\
:=&-y+X,\vspace{0.25cm}\\
\dot{y}=&x+K_{40}x^4+K_{04}y^4+K_{22}x^2y^2+K_{13}xy^3+K_{31}x^3y\vspace{0.25cm}\\
:=&x+Y,
\end{array}
\end{equation}

\begin{proposition}
System \eqref{ff} can be rewritten as
\begin{equation}\label{ff1}
\begin{array}{rl}
\dot{x}=&-y-\dfrac{\partial  H}{\partial y}-yg,\vspace{0.25cm}\\
\dot{y}=&x+\dfrac{\partial H}{\partial x}+xg,
\end{array}
\end{equation}
where $H=H_3+H_4+H_5,$ and $g=g(x,y)$ is a convenient polynomial of
degree three.
\end{proposition}
\begin{proof}
Indeed, in this case

\[
\displaystyle\int_0^{2\pi}\left.\left(\dfrac{\partial X}{\partial
x}+\dfrac{\partial Y}{\partial y}
\right)\right.|_{x=\cos{t},\,y=\sin{t}}dt\equiv 0,
\]
where $X$ and $Y$ are polynomials given by the formula \eqref{ff}.
In this case  the function $H$ and $g$ for which \eqref{ff1} holds
are
\[
\begin{array}{rl}
H=&\Lambda
(x^2+y^2)^2+\dfrac{1}{15}\left(8K_{04}+2K_{22}+2L_{13}+3K_{40}+3L_{31}\right)x^5\vspace{0.2cm}\\
&+K_{04}xy^4-L_{40}yx^4-\dfrac{1}{15}\left(8L_{40}+2L_{22}+2K_{31}+3L_{04}+3K_{13}\right)y^5\vspace{0.2cm}\\
&-\dfrac{1}{3}\left(4L_{40}+L_{22}+K_{31}\right)x^2y^3+\dfrac{1}{3}\left(4K_{04}+K_{22}+L_{13}
\right)x^2y^3,\vspace{0.2cm}\\
g=&-4\Lambda(x^2+y^2)-\dfrac{1}{3}\left(8K_{04}+2K_{22}+2L_{13}+3L_{31}\right)x^3+(4L_{40}+K_{31})x^2y\vspace{0.2cm}\\
&\dfrac{1}{3}\left(8L_{40}+2L_{22}+2K_{31}+3K_{13}
\right)y^3-(4K_{04}+L_{13})xy^2
\end{array}
\]
\end{proof}

\begin{example}
Quartic differential system \begin{equation}\label{uuu}
\dot{x}=-y+y^4,\quad \dot{y}=x+x^4-x^2y^2, \end{equation} admits a
center at the origin and foci at the point $(\approx -1,3247,1).$
\end{example}

\begin{proposition}
Quartic quasi homogenous differential system
\begin{equation}\label{rff2}
\begin{array}{rl}
\dot{x}=&-y+L_{40}x^4+L_{22}x^2y^2-\dfrac{2K_{04}}{\lambda}(2\lambda-1)xy^3\vspace{0.2cm}\\
&-\dfrac{1}{3\lambda^2}\left((10\lambda-4)K_{04}+2\lambda(\lambda-1)K_{22}\right)x^3y\vspace{0.2cm}\\
&-\dfrac{1}{3\lambda^2}\left(\lambda(2-5\lambda)L_{22}+((15\lambda^2-16\lambda+4)L_{40}\right) y^4,\vspace{0.2cm}\\
\dot{y}=&x+K_{04}y^4+K_{22}x^2y^2-\dfrac{2(2\lambda-1)L_{40}}{\lambda}x^3y\vspace{0.2cm}\\
&-\dfrac{2}{3\lambda^2}\left(5\lambda-2)L_{40}+\lambda(\lambda-1)L_{22}\right)xy^3\vspace{0.2cm}\\
&-\dfrac{2}{3\lambda^2}\left((15\lambda^2-16\lambda+4)K_{04}+(2\lambda-5\lambda^2)K_{22}\right)x^4
\end{array}
\end{equation}
is Poincar\'e--Liapunov integrable.
\end{proposition}
\begin{proof}
It is possible to show that the function
$F=(1+(1-\lambda))H^{(\lambda-1)/\lambda}$ for
$\lambda\in\mathbb{R}\setminus\{0\}$ with
\[\begin{array}{rl}
g=&\dfrac{2}{3\lambda^2}\left((3\lambda-2)K_{04}-\lambda\,K_{22}\right)x^3
+\dfrac{2}{3\lambda^2}\left(2(L_{40}(2-3\lambda)+\lambda\,L_{22}\right)y^3\vspace{0.2cm}\\
&+\dfrac{2L_{40}}{\lambda}x^2y+¡-\dfrac{2K_{04}}{\lambda}xy^2,
\end{array}
\]
is a first integral of \eqref{rff2} and if $\lambda=1$ then
$F=H_2e^{-\tilde{g}}$ with $\tilde{g}=g|_{\lambda=1}$ is a first
integral of \eqref{rff2} with $\lambda=1,$ i.e is a first integral
of quartic system
\[\begin{array}{rl}
\dot{x}=&-y+(x^2+y^2)\left(-L_{40}x^2-(L_{22}-L_{40})y^2+2K_{04}xy\right)\vspace{0.2cm}\\
\dot{y}=&x-(x^2+y^2)\left((K_{04}-K_{22})x^2+2L_{40}xy-K_{04}y^2\right).
\end{array}
\]
  Thus the proof follows.\end{proof}
\begin{corollary}
The quartic quasi-homogenous differential system
\[\begin{array}{rl}
\dot{x}=&-y+x\left(L_{40}x^3+L_{22}xy^2+K_{04}y^3+K_{22}x^2y\right),\vspace{0.2cm}\\
\dot{y}=&x+y\left(L_{40}x^3+L_{22}xy^2+K_{04}y^3+K_{22}x^2y\right).
\end{array}
\]
has an uniformly isochronous center at the origin and  admits the
first integral
\[F=\dfrac{\left((K_{22}+2K_{04})x^3-3L_{40}x^2y+3K_{04}xy^2-(L_{22}+2L_{40})y^3-1)\right)}{(x^2+y^2)^3}.\]
\end{corollary}
\begin{proof}
Follwos from Corollary \ref{rrr} with $m=4.$
\end{proof}

\begin{corollary}
Differential system \eqref{ff} is Hamiltonian if and only if
\begin{equation}\label{TRT}
\begin{array}{rl}
L_{22}+3/2K_{13}=0,\quad 4L_{40}+K_{31}=&0,\vspace{0.2cm}\\
4K_{40}+K_{22}++L_{13}+3/2L_{31}=&0.
\end{array}
\end{equation}
\end{corollary}
\begin{proof}
From the previous result follows that differential system \eqref{ff}
can be rewritten as follows
\[
\dot{x}=-y-\dfrac{\partial H_5}{\partial y},\quad \dot{y}=x+\dfrac{\partial H_5}{\partial
x},
\]
if and only if $ g_3\equiv 0 ,$ which is equivalent to \eqref{TRT}.
\end{proof}

\begin{proposition}\label{hhhh}
The  quartic differential system
\[
\begin{array}{rl}
\dot{x}=&-y-\dfrac{(5\lambda-2)\left((3\lambda-2)L_{40}-\lambda
L_{22}\right)}{3\lambda^2}y^4+L_{40}x^4+L_{22}x^2y^2\vspace{0.2cm}\\
&-\dfrac{2\left(
\lambda(\lambda-1)K_{22}+(2-5\lambda)L_{40}\right)}{3\lambda^2} yx^3,\vspace{0.2cm}\\
\dot{y}=&x+\dfrac{(5\lambda-2)\left((3\lambda-2)L_{40}+\lambda
K_{22}\right)}{3\lambda^2}x^4-L_{40}y^4+K_{22}x^2y^2\vspace{0.2cm}\\
&-\dfrac{2(2\lambda-1)L_{40}}{\lambda}yx^3-\dfrac{2\left(\lambda(\lambda-1)
L_{22}+(5\lambda-2)L_{40} \right)}{3\lambda^2}xy^3
\end{array}
\]
with $\lambda\in\mathbb{R}\setminus\{0,1\}$ admits the Darboux first
integral
\[
F=\Big(1+(1-\lambda)g_3\Big)H_2^{(\lambda-1)/\lambda}
\]
where \[ g_3=\dfrac{\left(\lambda K_{22}
+(3\lambda-2)L_{40}\right)}{3\lambda^2}x^3+\dfrac{2L_{40}}{\lambda}xy(x+y)
-\dfrac{\left((3\lambda-2)L_{40}-\lambda\,
L_{22}\right)}{3\lambda^2}y^3.\] and if $\lambda=1$ the quartic
quasi homogenous differential system
\[
\begin{array}{rl}
\dot{x}=&-y+(x^2+y^2)((L_{40}(x^2+2xy-y^2)+L_{22}y^2),\vspace{0.2cm}\\
\dot{y}=&x+(x^2+y^2)((L_{40}(x^2-2xy-y^2)+K_{22}x^2),
\end{array}
\]
admits the first integral $H_2e^{-\tilde{g}_3}=Const.$ where
$\tilde{g}_3=g_3|_{\lambda=1}.$
\end{proposition}

\begin{proof}
{From} the equation
\[H_5+\lambda\,g_3H_2=\sum_{j+k=5}a_{jk}x^jy^k=0
\] and by solving $a_{jk}=0$ for $j+k=5$
we deduce the proof of the proposition.
\end{proof}
In \cite{Salih} the following theorem was proved (see Theorem 2)
\begin{theorem}\label{Salih}
Differential system
\begin{equation}\label{ff11}
\begin{array}{rl}
\dot{x}=&-y,\vspace{0.25cm}\\
\dot{y}=&x+K_{40}x^4+K_{04}y^4+K_{22}x^2y^2+K_{13}xy^3+K_{31}x^3y,
\end{array}
\end{equation}
or equivalently
\[\dot{z}=iz+A_{40}z^4+A_{04}\bar{z}^4+A_{22}z^2\bar{z}^2+A_{13}z\bar{z}^3+A_{31}z^2\bar{z},\]
is called the  quartic lopsided system where $A_{jk}$ are the
following parameters

\[
\begin{array}{rl}
A_{40}=&\dfrac{1}{16}\left( K_{40}+K_{04}-K_{22}+i(K_{31}-K_{13})\right),\vspace{0.2cm}\\
A_{04}=&\dfrac{1}{16}\left( K_{40}+K_{04}-K_{22}-i(K_{31}-K_{13})\right),\vspace{0.2cm}\\
A_{31}=&\dfrac{1}{8}\left( 2K_{40}-2K_{04}+i(K_{31}+K_{13})\right),\vspace{0.2cm}\\
A_{13}=&\dfrac{1}{8}\left( 2K_{40}-2K_{04}-i(K_{31}+K_{13})\right),\vspace{0.2cm}\\
A_{22}=&\dfrac{1}{8}\left( 3K_{40}+3K_{04}+K_{22}) \right),
\end{array}
\]

The origin is a center for the the lopsided differential system
 if and only if one of
the following conditions hold

\begin{itemize}
\item[(1)]
All the coefficients $A_{jk}$ with $j+k=4$ are reals.
\item[(2)] $A_{13}=A_{22}=0$ and $\Re(A_{04})=0.$
\item[(3)] $A_{04}=A_{22}=0$ and $\Re(A_{13})=0.$
\item[(3)] $A_{22}=\Re(A_{13})=\Re(A_{04})=0$ and $\Im(A_{04})+\beta \Im(A_{13})=0,$ where
$\beta\in\mathbb{R}.$
\end{itemize}
\end{theorem}

After some computations from the given  conditions we get

\begin{itemize}
\item[(1)] $K_{13}=K_{31}=0,$  thus differential system
\eqref{ff11} becomes
\[\dot{x}=-y,\quad
\dot{y}=x+K_{04}y^4+K_{40}x^4+K_{22}x^2y^2.
\]
\item[(2)]
\[
K_{04}=K_{40}= K_{22}=0,\quad K_{13}=-K_{31},
\]  thus differential system
\eqref{ff11} becomes
\begin{equation}\label{rrre}
\dot{x}=-y,\quad
\dot{y}=x++K_{13}xy\left(x^2-y^2\right).
\end{equation}
\item[(3)]
\[
K_{04}=K_{40}=K_{22}=0,\quad K_{13}=K_{31},\quad
\]  thus differential system
\eqref{ff11} becomes
\[\dot{x}=-y,\quad
\dot{y}=x+K_{31}xy\left(x^2+y^2)\right).
\]
\item[(4)] a)
\[
K_{04}=K_{40}=0,\quad K_{13}=\dfrac{2b-1}{1+2b}K_{31},\quad
K_{22}=0,\quad K_{40}=0,
\] with $\beta\ne -1/2$  thus differential system
\eqref{ff11} becomes
\[\dot{x}=-y,\quad
\dot{y}=x+K_{31}\left(x^3y-\dfrac{2b-1}{1+2b}xy^3\right)  .
\]
b) If $\beta= -1/2$ then
\[
K_{04}= K_{22}=K_{31}=K_{40}=0,
\]  thus differential system
\eqref{ff11} becomes
\[\dot{x}=-y,\quad
\dot{y}=x++K_{13}xy^3.
\]
\end{itemize}
We observe that the differential systems \eqref{rrre} are particular
case of the system
\[\dot{x}=-y,\quad \dot{y}=x+ax^4+by^4+cx^2y^2,\]
where $a,b$ and $c$ are arbitrary constants, which is invariant
under the change $(x,-y,-t)\longrightarrow (x,y,t)$ and the
differential systems of the other cases are particular cases of the
differential system
 \[\dot{x}=-y,\quad \dot{y}=x+xy(\alpha\, x^2+\beta y^2),\]
where $\alpha,\beta$ and $\gamma$ are arbitrary constants, which is
invariant under the change $(-x,y,-t)\longrightarrow (x,y,t)$ (see
Proposition \ref{RVN}).
\begin{corollary}
Theorem   \ref{Salih}  give only sufficient conditions.
\end{corollary}
\begin{proof}
Indeed the following non-reversible quartic differential system
\begin{equation}\label{ttt}
\dot{x}=-y,\quad \dot{y}=x+ay(y^3-4x^3),
\end{equation}
has a center at the origin .

\end{proof}
\begin{corollary}
Quartic quasihomogenous differential system with holomorphic center
is
\begin{equation}\label{ff1012}
\begin{array}{rl}
\dot{x}=&-y+L_{40}x^4 +L_{22}x^2y^2-4L_{40}xy^3-(5L_{40}+L_{22})y^4\vspace{0.2cm}\\
&+(2K_{22}-8L_{40})yx^3,\vspace{0.2cm}\\
\dot{y}=&x+(5L_{40}-K_{22})x^4+4L_{40}yx^4+K_{22}x^2y^2\vspace{0.2cm}\\
&-L_{40}y^4+(2L_{22}+8L_{40})xy^3.
\end{array}
\end{equation}
is Poincar\'e--Liapunov integrable with the first integral
\[
\dfrac{1+2(5L_{40}-K_{22})x^3+2(5L_{40}+L_{22})y^3+6L_{40}xy(x+y)
}{H_2^3}=Const..
\]
\end{corollary}
\begin{proof}
Follows from Proposition \ref{hhhh} with $\lambda=1/4.$ We observe
that \eqref{ff1012} by introducing the respectively notations,
coincide with the quartic planar differential system deduced in
\cite{Devlin}.
\end{proof}
\begin{corollary}
The quasi--homogenous quartic planar differential system  with
holomorphic center is
\[\begin{array}{rl}
\dot{x}=&-y+L_{40}\left(x^4+y^4-6x^2y^2-4xy^3+4yx^3\right),\\
\dot{y}=&x+L_{40}\left(-y^4-x^4+6x^2y^2-4xy^3+4yx^3\right),
\end{array}
\]
or equivalently
\[\dot{z}=iz+L_{40}(1-i)z^4.\]

\end{corollary}
\begin{proof}
By solving the equations
\[\Delta H_5+2g_3=0,\quad H_2g+4H_5=0,\]
we obtain the proof.

\smallskip

We observe that the first integral in this case is

\[F=\dfrac{1+10L_{40}(x+y)(x^2+y^2)}{(x^2+y^2)^5}.\]
\end{proof}
\section{Quintic quasihomogenous polynomial planar vector field with
non-degenerate center}

We shall study the quintic polynomial vector field
\begin{equation}\label{Q1}
\begin{array}{rl}
\dot{x}=&-y+L_{50}x^5+L_{05}y^5+L_{23}x^2y^3++L_{32}x^3y^2\vspace{0.2cm}\\
&+L_{41}x^4y+L_{14}xy^4:=-y+X,\vspace{0.2cm}\\
\dot{y}=&x+K_{50}x^5+K_{05}y^5+K_{23}x^2y^3++K_{32}x^3y^2\vspace{0.2cm}\\
&+K_{41}x^4y+K_{14}xy^4:=x+Y,
\end{array}
\end{equation}
\begin{proposition}
Differential system \eqref{Q1} can be rewrite as follows
\[
\begin{array}{rl}
\dot{x}=&-\dfrac{\partial H_2+H_6}{\partial y}-yg_4,\vspace{0.2cm}\\
\dot{y}=&\dfrac{\partial H_2+H_6}{\partial x}+xg_4,
\end{array}
\]
where $H_6=H_6(x,y)$ and $g_4=G_4(x,y)$ are convenient homogenous
polynomial of degree six and four respectively, if and only if
\begin{equation}\label{Q3}
\displaystyle\int_0^{2\pi}\left.\left(\dfrac{\partial X}{\partial
x}+\dfrac{\partial Y}{\partial y}
\right)\right.|_{x=\cos{t},\,y=\sin{t}}dt=K_{41}+K_{23}+L_{23}+L_{14}+5(K_{05}+L_{50})=0,
\end{equation}
where $X$ and $Y$ are polynomials given by the formula \eqref{Q1}
\end{proposition}
\begin{proof}
Indeed, under the condition \eqref{Q3} the polynomial $H_6$ and
$g_4$ are
\[
\begin{array}{rl}
H_6=&-\dfrac{\left(K_{32}+L_{23}+2K_{14}+2L_{41}+2K_{50}\right)}{12}y^6-L_{50}x^5y\vspace{0.2cm}\\
&-\dfrac{L_{41}+K_{50}}{2}x^4y^2-\dfrac{L_{23}+2L_{41}+2K_{50}+K_{32}}{4}x^2y^4\vspace{0.2cm}\\
&+\dfrac{\left(K_{23}+L_{14}+5K_{05}\right)}{3}x^3y^3+\Lambda\,(x^2+y^2)^3+K_{05}xy^5,\vspace{0.2cm}\\
g_4=&\dfrac{\left(L_{23}+L_{23}+2L_{41}+2K_{14}+ K_{32}\right)}{2}\,y^4-\left(K_{23}+L_{14}+L_{32}+5K_{05}\right)x^3y\vspace{0.2cm}\\
&-(L_{14}+5K_{05})xy^3+(2L_{41}+K_{32})x^2y^2+2(K_{50}-6\Lambda)(x^2+y^2)^2
\end{array}
\]
\end{proof}
\begin{proposition}\label{chava3}
Quintic differential system
\[
\begin{array}{rl}
\dot{x}=&-y+L_{50}x^5+L_{41}x^4y-\dfrac{1}{\lambda}( 3\lambda K_{05}+L_{50}(2-3\lambda) )y^3x^2+\dfrac{K_{05}(5\lambda-2)}{\lambda}xy^4\vspace{0.2cm}\\
&+L_{23}x^2y^3+\dfrac{(3\lambda-1)}{2\lambda^3}\left(\lambda^2L_{23}-2\Lambda(2\lambda^2-3\lambda+1)+\lambda(1-2\lambda)L_{41}\right)y^5,\vspace{0.2cm}\\
\dot{y}=&x-\dfrac{2-5\lambda}{\lambda}L_{50}x^4y-(\dfrac{(3\lambda-2)K_{05}}{\lambda}-3L_{50})x^2y^3-K_{05}\,y^5\vspace{0.2cm}\\
&+\dfrac{1}{2\lambda^3}\left(\lambda^2(\lambda-1)+\Lambda(6\lambda^2-8\lambda+2)+\lambda(3\lambda-1)L_{41}\right)xy^4.
\end{array}
\]
is Poincar\'e--Liapunov integrable.
\end{proposition}
\begin{proof}
Indeed, the function $F=(1+(1-\lambda)g_4)H_2^{(\lambda-1)/\lambda}$
if $\lambda\in\mathbb{R}\setminus\{0,1\}$ where
\[
\begin{array}{rl}
g_4=&-\dfrac{2\Lambda}{\lambda}x^4+
\dfrac{2L_{50}}{\lambda}x^3y-\dfrac{2K_{05}}{\lambda}xy^3\vspace{0.2cm}\\
&-\dfrac{2(\lambda-1)\Lambda+\lambda\,L_{41}}{\lambda^2}x^2y^2
+\dfrac{1}{2\lambda^3}\left(\lambda^2L_{23}-2(2\lambda^2-3\lambda+1)\Lambda+\lambda(-2\lambda)L_{41}\right)y^4,
\end{array}
\]
and $F=H_2e^{-\tilde{g}}$ where
\[\tilde{g}=g|_{\lambda=1}=-2\Lambda x^4+2L_{50}x^3y+L_{41}x^2y^2-2K_{05}xy^3+\dfrac{L_{23}-L_{41}}{2}y^4,
\] are first integrals.
\end{proof}
\begin{corollary}
The quintic quasi-homogenous differential system
\[
\begin{array}{rl}
\dot{x}=&-y+x\left(L_{50}x^4+L_{41}x^3y+L_{23}xy^3+K_{05}y^4-3(K_{05}+L_{50})x^2y^2\right),\vspace{0.2cm}\\
\dot{y}=&x+y\left(L_{50}x^4+L_{41}x^3y+L_{23}xy^3+K_{05}y^4-3(K_{05}+L_{50})x^2y^2\right),\vspace{0.2cm}\\,,
\end{array}
\]
has an uniformly isochronous center at  the origin and admits the
following rational first integral
\[F=\dfrac{\left(4L_{50}x^3y-2L_{41}x^2y^2-4K_{05}xy^3+(L_{41}+L_{23})y^4-4\Lambda(x^2+y^2)^2\right)^2}{(x^2+y^2)^3}.\]
\end{corollary}
\begin{proof}
Follows from Proposition \ref{chava3} with $\lambda=1/3$  and in
view of the Corollary \ref{rrr}.
\end{proof}
\begin{remark}
In the paper \cite{Salih1} the following result was proposed.
\begin{theorem}
The quintic differential system
\[
\begin{array}{rl}
\dot{x}=&-y:=P,\vspace{0.2cm}\\
\dot{y}=&x+K_{50}x^5+K_{05}y^5+K_{23}x^2y^3++K_{32}x^3y^2+K_{41}x^4y+K_{14}xy^4;=Q,
\end{array}
\]
admits a center at the origin if and only if the coefficients
$A_{kn}$ such that
\[P+iQ=iz+\displaystyle\sum_{k+n=5}A_{kn}z^k\bar{z}^n,\]
are real,i.e,
\begin{equation}\label{rQ1}
\begin{array}{rl}
K_{05}-K_{23}+K_{41}=&0,\vspace{0.2cm}\\
5K_{05}-K_{23}-3K_{41}=&0,\vspace{0.2cm}\\
5K_{05}+K_{23}+K_{41}=&0,
\end{array}
\end{equation}
\end{theorem}
It is easy to show that the unique quintic polynomial vector field
for which \eqref{rQ1} holds is
\[
\dot{x}=-y,\quad \dot{y}=x+K_{50}x^5+K_{32}x^3y^2++K_{14}xy^4,
\]
which is a particular case of the reversible system \eqref{qh} with

$m=5.$

\smallskip

The following counterexample show that this theorem is only
sufficient condition.

\smallskip

The non-reversible quintic system
\begin{equation}\label{ssss}
\dot{x}=-y,\quad \dot{y}=x-5x^4y+y^5
\end{equation}
has a center at the origin .
\end{remark}

\section{Analytic planar vector field with non-degenerate center}
In this section we study the inverse problem of the center for the
case when the vector field is analytic.

\smallskip

We shall study the analytic planar differential system
\begin{equation}\label{inverse1011}
\dot{x}=-\dfrac{\partial H}{\partial y}-yg,\quad
\dot{y}=\dfrac{\partial H}{\partial x}+xg,
\end{equation}
 where $g=g(x,y)$ is an
analytic function and $H==\dfrac{\lambda}{2}(x^2+y^2)+f(x,y),$ and
$f=f(x,y)$ is a real analytic functions in an open neighborhood of
$\textsc{O}$ whose Taylor expansions at $\textsc{O}$ do not contain
constant, linear and quadratic terms.
\begin{proposition}
Differential system \eqref{inverse1011} is Hamiltonian if and only
if $g=g(x^2+y^2).$
\end{proposition}
\begin{proof}
From divergence condition for \eqref{inverse1011} follows that
\[
x\dfrac{\partial g}{\partial y}-y\dfrac{\partial g}{\partial\,x}=0,
\]
consequently $g=g(x^2+y^2).$ The Hamiltonian is
\[H=V+\dfrac{1}{2}\displaystyle\int g(x^2+y^2)d(x^2+y^2).\]
\end{proof}
Now we study the analytic planar differential system
\begin{equation}\label{ttss}
\dot{x}=\sigma x-(\tilde{\nu} +1)y=-y+X,\quad\dot{y}=\sigma
y+(\tilde{\nu}+1) x=x+Y,
\end{equation}
where $\sigma=\sigma(x,y)$ and $\tilde{\nu}=\tilde{\nu}(x,y)$ are
analytic functions. In polar coordinates this differential system
becomes
\begin{equation}\label{center4}
\dot{r}=r\sigma,\quad \dot{\vartheta}=\tilde{\nu}+1:=\nu.
\end{equation}
\begin{corollary}
Differential system \eqref{center4} is Hamiltonian if and only if
\begin{equation}\label{center5}
\dfrac{\partial\varrho\nu}{\partial\vartheta}+\dfrac{\partial\varrho\,r\sigma}{\partial\,r}=0.
\end{equation}
where $\varrho$ is the integrating factor. The Hamiltonian is
\[H=\displaystyle\int\varrho\left(\nu dr-r\sigma
d\vartheta\right).\]
\end{corollary}
\begin{proof}
Indeed, from the relations
\[ \dfrac{\partial H}{\partial \vartheta}=-\varrho r\sigma,\quad \dfrac{\partial H}{\partial
r}=\varrho\nu,\] and in view of the compatibility conditions we
obtain \eqref{center5}.
\end{proof}

We study the particular case when $\sigma=\dfrac{\partial
A}{\partial\,\theta}$ and $\varrho=\varrho (r),$ where
$A=A(r,\theta).$ Under these conditions \eqref{center5} becomes
\[
\dfrac{\partial}{\partial\,\theta}\left(
\varrho(r)\nu+\dfrac{\partial}{\partial\,r}\left(r\varrho (r)
A\right)\right)=0,
\]
 thus
\[\nu=-\dfrac{1}{\varrho(r)}\left(\dfrac{\partial}{\partial\,r}\left(r\varrho (r) A\right)+\tau (r)\right)
\]
where $\tau=\tau (r)$ is an arbitrary function. The first integral
is
\[H=r\varrho(r)A(r,\theta)+q(r),\quad q'(r)=\tau(r).\]
Differential system \eqref{center4} in this case becomes
\begin{equation}\label{center7}
\dot{r}=\dfrac{1}{\varrho}\dfrac{\partial}{\partial\,\theta}\left(r\varrho\,
A+q(r)\right),\quad \dot{\vartheta}
=-\dfrac{1}{\varrho}\dfrac{\partial}{\partial\,r}\left(r\varrho\,A+q(r)\right).
\end{equation}

\smallskip

 Now we study the generalized weak condition of the center.
\begin{corollary}
If differential system \eqref{ttss} satisfy the conditions
\begin{equation}\label{center6}
\begin{array}{rl}
(x^2+y^2)\left(\dfrac{\partial\,X}{\partial\,x}+\dfrac{\partial\,Y}{\partial\,y}\right)=&k\left(xX+yY\right),\quad
k\in\mathbb{R}\setminus\{0\},\vspace{0.2cm}\\
\sigma=&\dfrac{\partial A}{\partial\,\theta}
\end{array}
\end{equation}
then is Poincar\'e-Liapunov integrable with the first integral
\[H=r^{2-k}A(r,\vartheta)+q(r),\]
where $q=q(r)$ is an arbitrary function.
\end{corollary}
\begin{proof}
In view of the relations
\[\begin{array}{rl}
\dfrac{\partial X}{\partial x}+\dfrac{\partial Y}{\partial
y}=&x\dfrac{\partial\,\sigma}{\partial\,x}+y\dfrac{\partial\,\sigma}{\partial\,y}
+x\dfrac{\partial\,\nu}{\partial\,y}-y\dfrac{\partial\,\nu}{\partial\,x}+2\sigma,\vspace{0.2cm}\\
xX+yY=&\kappa(x^2+y^2)\sigma,
\end{array}
\]
from \eqref{center6}we  have
\[x\dfrac{\partial\,\sigma}{\partial\,x}+y\dfrac{\partial\,\sigma}{\partial\,y}
+x\dfrac{\partial\,\nu}{\partial\,y}-y\dfrac{\partial\,\nu}{\partial\,x}=(\kappa-2)\sigma,
\]
which in polar coordinates becomes
\[r\dfrac{\partial\,\sigma}{\partial\,r}+(2-k)\sigma+\dfrac{\partial
\nu}{\partial\,\theta}=0,\]
 consequently
 \[\dfrac{\partial}{\partial\,\theta}\left(
r\dfrac{\partial\,A}{\partial\,r}+(2-k)A+\nu \right)=0,
\]
thus
\[
\nu=-r^{k-1}\left(\dfrac{\partial\,r^{2-k}A}{\partial\,r}+p(r)\right),
\]
where $p=p(r)$ is an arbitrary function. By introducing the function
$q=q(r)$ such that $q'(r)=p(r)$ we finally obtain that the given
differential system in polar coordinates becomes
\[
\dot{r}=r^{k-1}\dfrac{\partial}{\partial\,\theta}\left(r^{2-k}A+q(r)\right),\quad
\dot{\vartheta}
=-r^{k-1}\dfrac{\partial}{\partial\,r}\left(r^{2-k}A+q(r)\right).
\]
We observe that this system can be obtained from  \eqref{center7}
with $\varrho=r^{1-k}.$
\end{proof}
\begin{corollary}
Differential system \eqref{center7} has an isochronous point at the
origin if
\begin{equation}\label{center77}
A(r,\theta)=\dfrac{\Psi(\vartheta)-\displaystyle\int\varrho(r)dr-q(r)}{r\varrho(r)}
\end{equation}
\end{corollary}
\begin{proof}
Indeed, if \eqref{center77} holds then \eqref{center4} becomes
\begin{equation}\label{center88}
\dot{r}=\dfrac{1}{\varrho(r)}\dfrac{\partial \Psi(\vartheta)}{\partial\,\theta},\quad \dot{\vartheta}=1.
\end{equation}
The Hamiltonian function in this case takes the form
\[
H=r\varrho\,A(r,\vartheta)+q(r)=\Psi(\vartheta)-\displaystyle\int\varrho(r)dr.
\]
\end{proof}

\begin{remark}
The following remarks are related with differential equations
\eqref{center4}.

\begin{itemize}

\item[(i)]  In \cite{Vladimir} was studied the existence of  a uniformly
isochronous polynomial system has the form
\[
\dot{x}=-y+xG(x,y)\Omega\,(x^2+y^2),\quad
\dot{y}=x+yG(x,y)\Omega\,(x^2+y^2),
\]
where $G=G(x,y)$ is a homogenous polynomial in $x$ and $y$ of degree
$k$ and $\displaystyle\int_0^{2\pi}G(\cos\,t,\sin\,t)dt=0.$ Clearly
this system is a particular case of system \eqref{center88}.

\smallskip

\item[(ii)]   The center problem for the system \eqref{center4} with
$\sigma=\dfrac{\partial A}{\partial\,\theta}$ was study in
\cite{Brudnyi}. The following corollary was proved (see Corollary
2.21)

\begin{corollary}  Let now $H(x,y)\in{\mathbb{C}[x,y]}$ be a homogeneous
polynomial. For any holomorphic functions $P1,P2$ defined in an open
neighborhood of $0\in\mathbb{C}$ we define $A(x,y) := P_1(H(x,y)),$
and $B(x,y) := P_2(H(x,y))$. Then the vector field
\begin{equation}\label{Ce}
\begin{array}{rl} \dot{x}=&\left(x\dfrac{\partial A}{\partial
y}-y\dfrac{\partial\,A}{\partial x}\right) x-(B +1)y
,\vspace{0.2cm}\\
\dot{y}=&\left(x\dfrac{\partial A}{\partial y}-y\dfrac{\partial\,
A}{\partial x}\right) y+(B+1) x,
\end{array}
\end{equation}
 determines a center.

 \smallskip

 Clearly system \eqref{center0} is a particular case of \eqref{Ce}.
\end{corollary}
\item[(iii)]
Characterization of isochronous point for planar system was study in
particular in \cite{Algaba}.

\end{itemize}

\end{remark}

\section{The center-foci problem}

As we observe in the previous section the  center- foci problem
consists into distinguish when the origin of \eqref{3} is a center
or foci.

\smallskip

We need the following definitions:

By using the polar coordinates $x=r\cos\theta,\,\,y=r\sin\theta$ in
\eqref{3} and denoting by
$r(t,r_0,\theta_0),\,\theta(t,r_0,\theta_0)$ the solution of the
\eqref{3} such that $r|_{t=t_0}=r_0,\,\theta|_{t=t_0}=\theta_0$. We
say that the origin is stable focus if there exists $\delta>0$ such
that for $0<r_0<\delta$ and $\theta_0\in\mathbb{R}$, we have that
\[
\lim_{t\to\infty}r(t,r_0,\theta_0)=0,\quad\mbox{and}\quad
\lim_{t\to\infty}|\theta(t,r_0,\theta_0)|=\infty,
\]
and unstable if
\[
\lim_{t\to-\infty}r(t,r_0,\theta_0)=0,\quad\mbox{and}\quad
\lim_{t\to-\infty}|\theta(t,r_0,\theta_0)|=\infty.
\]

\smallskip

To solve the center- foci problem  Poincar\'e and Liapunov developed
the method which is given in the following theorems (see for
instance  \cite{Malkin,Liapunov}) .

\begin{theorem}
For the system \eqref{3} there exists a formal power series
\begin{equation}\label{50}
V(x,y)=\displaystyle\sum_{j=2}^{\infty}H_j(x,y),\quad
H_2(x,y)=\dfrac{1}{2}\left(x^2+y^2\right),
\end{equation}
where $H_j(x,y)=H_j$ is a homogenous function of degree $j.$ such
that
\begin{equation}\label{r51}
\begin{array}{rl}
 \dfrac{d V}{dt}=&\left(x+\dfrac{\partial H_3}{\partial x}+
 \dfrac{\partial H_4}{\partial x}+\ldots\right)\left(-\lambda\,y
 +\displaystyle\sum_{j=2}^{m}X_j(x,y)\right)\vspace{0.2cm}\\
 &+\left(y+\dfrac{\partial H_3}{\partial y}+
 \dfrac{\partial H_4}{\partial y}+\ldots\right)\left(\lambda\,x
 +\displaystyle\sum_{j=2}^{m}Y_j(x,y)\right)\vspace{0.2cm}\\
=& \displaystyle\sum_{j=0}^\infty\,V_j(x^2+y^2)^{j+1},
\end{array}
\end{equation}
where $V_j$ are constants called  Poincar\'e-Liapunov constants.
\end{theorem}

The main objective of the next section is to study the following
inverse problem in ordinary differential equations (see for instance
\cite{RS2}).

\begin{problem}\label{analytic}

\smallskip

Determine the analytic planar vector fields
\[
\X=(-y+\displaystyle\sum_{j=2}^\infty X_j(x,y))\dfrac{\partial
}{\partial x}+(x+\displaystyle\sum_{j=2}^\infty
Y_j(x,y))\dfrac{\partial }{\partial y},
\] where $X_j$ and $Y_j$ for $j\geq 2$ are unknown
homogenous polynomial of degree $j,$ in such a way that \eqref{r51}
holds.
\end{problem}
\begin{problem}
Study the Problem \ref{analytic} when $\X$ is polynomial of degree
$m,$ i.e.
\[
\X=(-y+\displaystyle\sum_{j=2}^m X_j(x,y))\dfrac{\partial }{\partial
x}+(x+\displaystyle\sum_{j=2}^m Y_j(x,y))\dfrac{\partial }{\partial
y},
\]
\end{problem}

{F}rom \eqref{50} and \eqref{r51} if  $\lambda=1$ we obtain the
matrix equation
\[
A\,\Psi=B+C,
\]
where $A,\,\Psi,\,B$ and $C$ are matrix such that
\[
\left(
    \begin{array}{ccccccccc}
      \partial_x\,H_{2} & \partial_y\,H_{2} & 0 &0 & 0 &0 & 0 & \hdots &0 \\
      \partial_x\,H_{3} & \partial_y\,H_{3}
      & \partial_x\,H_{2} & \partial_y\,H_{2} &0 & 0 & 0 &\hdots &0 \\
     \partial_x\,H_{4} & \partial_y\,H_{4} & \partial_x\,H_{3}
      & \partial_y\,H_{3} &\partial_x\,H_{2} & \partial_y\,H_{2} &0& \hdots &0 \\
      \vdots & \vdots & \vdots & \vdots & \vdots & \vdots &  \vdots &\vdots &0  \\
      \partial_x\,H_{m} & \partial_y\,H_{m}& \partial_x\,H_{m-1}
      & \partial_x\,H_{m-1}  & \vdots &\vdots & \vdots &\partial_y\,H_{2}& \partial_x\,H_{2} \\
       \partial_x\,H_{m+1} & \partial_y\,H_{m+1} &\partial_x\,H_{m}
       & \partial_y\,H_{m} & \vdots &\vdots &\vdots & \partial_y\,H_{3}& \partial_x\,H_{3}
       \\
       \vdots & \vdots & \vdots & \vdots & \vdots & \vdots &  \vdots &\vdots &\vdots  \\
 \vdots & \vdots & \vdots & \vdots & \vdots & \vdots &  \vdots &\vdots &\vdots  \\
    \end{array}
  \right)
\]
\[\begin{array}{rl}
\Psi=&\left(X_2,\,Y_2,\,X_3,\,Y_3,\ldots,X_m,Y_m,0,\ldots,0,\ldots\right)^T,\mbox{for polynomial vectot fields}\vspace{0.2cm}\\
\Psi=&\left(X_2,\,Y_2,\,X_3,\,Y_3,\ldots,X_m,Y_m,\ldots,\ldots\right)^T,\mbox{for
analytic vectot fields}
\end{array}
\]
\[\begin{array}{rl}
B=&\left(\{H_3,\,H_2\},\,\{H_4,H_2\},\,\{H_5,H_2\},\ldots,\{H_m,H_2\},\ldots\right)^T,\vspace{0.2cm}\\
C=&\left(V_1(x^2+y^2)^2,\,0,V_2(x^2+y^2)^3,\,0,V_4(x^2+y^2)^4,0,\ldots,V_m(x^2+y^2)^m,\ldots\right)^T,
\end{array}
\]
where $\{f,g\}:=f_x\,g_y-f_y\,g_x,$  $f_x=\dfrac{\partial
f}{\partial x}$ and  $f_y=\dfrac{\partial f}{\partial y}.$

\smallskip

\section{Quasihomogenous polynomial differential  system of degree even with a foci at the origin }
The problem which we study in this and the following section is the
following.

\begin{problem}
 Determine the quasihomogenous polynomial vector field
\begin{equation}\label{7}
\dot{x}=- y+X_m(x,y),\quad \dot{y}=x+Y_m(x,y),
\end{equation}
where  $X_m=X_m(x,y),$ and $,Y_m=Y_m(x,y)$ are homogenous polynomial
of degree $m$  in such a way that
\begin{equation}\label{r88}
\begin{array}{rl}
 \dfrac{d V}{dt}=&\left(x+\dfrac{\partial H_3}{\partial x}+
 \dfrac{\partial H_4}{\partial x}+\ldots\right)\left(- y+X_m(x,y)\right)\vspace{0.2cm}\\
 &+\left(y+\dfrac{\partial H_3}{\partial y}+ \dfrac{\partial H_4}{\partial y}+\ldots\right)\left( x+Y_m(x,y)\right)\vspace{0.2cm}\\
=& \displaystyle\sum_{j=0}^\infty\,V_j(x^2+y^2)^{j+1},
\end{array}
\end{equation}
In this section we study the case when the degree of the vector
field is even.
\end{problem}

\smallskip

\begin{proposition}
Polynomial vector field \eqref{7} of degree $m=2k-2$ for which
\eqref{r88} holds is
\begin{equation}\label{9}
\begin{array}{rl}
\dot{x}=&-y-\dfrac{\partial H_{2k-1}}{\partial y}-y g_{2k-3},\vspace{0.2cm}\\
\dot{y}=&x+\dfrac{\partial H_{2k-1}}{\partial x}+x g_{{2k-3}},,
\end{array}
\end{equation}
where   $g_{2k-3}=g_{2k-3}(x,y),$ is a homogenous polynomial of
 degree $2k-3,$ and
\begin{equation}\label{10}
\begin{array}{rl}
H_4=\kappa_2H^2_2,\quad H_6=\kappa_3H^3_2,\ldots,
H_{2k-2}=&\kappa_{k-1}H^{k-1}_2,\vspace{0.2cm}\\
H_3=H_5=H_7=\ldots=H_{2k-3}=&0,\vspace{0.2cm}\\
V_1=V_2,\ldots=V_{k-2}=0,\vspace{0.2cm}\\
g_{2k-1}\{H_2,H_{2j+1-2k},\}+\{H_{2j+1-2k},H_{2k+1}\}+\{H_2 ,H_{2j}\}=&V_{j}\,(x^2+y^2)^{j+1},\vspace{0.2cm}\\
g_{2k-1}\{H_{2j+2-2k},H_2\}+\{H_{2j+2-2k},H_{2k+1}\}+\{H_2,H_{2j+1}\}=&0,
\end{array}
\end{equation}
for $j\geq k-1.$
\end{proposition}
\begin{proof}
From \eqref{9} follows that
 if $m=2k-2$ then
\[
\begin{array}{rl}
\{H_3,H_2\}=&0,\vspace{0.2cm}\\
\{H_4,H_2\}=&V_1(x^2+y^2)^2,\vspace{0.2cm}\\
\{H_5,H_2\}=&0,\vspace{0.2cm}\\
\{H_6,H_2\}=&V_2(x^2+y^2)^3,\vspace{0.2cm}\\
\qquad\qquad\vdots\qquad\qquad\qquad\vdots\qquad\qquad\vdots\qquad
&\qquad\vdots\vspace{0.2cm}\\
\{H_{2k-3},H_2\}=&0,\vspace{0.2cm}\\
\{H_{2k-2},H_2\}=&V_{k-2}(x^2+y^2)^{k-1},\vspace{0.2cm}\\
xX_{2k-2}+yY_{2k-2}+\{H_{2k-1},H_2\}=&0,\vspace{0.2cm}\\
\dfrac{\partial H_3}{\partial x}X_{2k-2}+\dfrac{\partial\,H_3}{\partial x}Y_{2k-2}+\{H_{2k},H_2\}=&V_{k-1}(x^2+y^2)^{k},\vspace{0.2cm}\\
\qquad\qquad\vdots\qquad\qquad\qquad\vdots\qquad\qquad\vdots\qquad
&\qquad\vdots.
 \end{array}
 \]
 Hence  by considering that  that
\[\int_{0}^{2\pi}\left.\{f,H_2\}\right|_{x=cos
t,\,y=sin t}dt=0,\] for arbitrary derivable function $f,$ we obtain
that $V_j=0$ for $j=2,\ldots,k-2.$

\smallskip

In view of the relations
\[\{H_2,H_j\}=0\Longleftrightarrow f=\left\{
                                     \begin{array}{ll}
                                       0, & \hbox{if $j=2n+1$,} \\
                                       \kappa_j(x^2+y^2)^j, & \hbox{if $j=2n+1$.}
                                     \end{array}
                                   \right.
\]
we have that $H_{2j}=\kappa_jH^{j}_2,$ and $H_{2j-1}=0$ for
$j=1,2,\ldots,k-2.$

\smallskip

Finally from the equation $xX_{2k-2}+yY_{2k-2}+\{H_{2k-1},H_2\}=0,$
follows that
\[
 X_{2k-2}= -\dfrac{\partial H_{2k-1}}{\partial y}-yg_{2k-3},\quad
Y_{2k-2}=\dfrac{\partial H_{2k-1}}{\partial x}-xg_{2k-3}.
\]
By inserting in the remain equations we obtain the last two systems
of first order partial differential equations, where
$g_{2k-3}=g_{2k-3}(x,y)$ is an arbitrary polynomial of degree
$2k-3.$ Thus we obtain the proof of the proposition.
\end{proof}
\begin{corollary}\label{hhh}
Any quasihomogenous differential equations of degree $m=2k-2$ can be
rewrite as \eqref{7}.
\end{corollary}

\section{Quasihomogenous polynomial differential  system of degree odd with a foci at the origin }
 \begin{proposition}
Polynomial vector field \eqref{7} of degree $m=2k-1$ for which
\eqref{r88} holds is
\[
\begin{array}{rl}
\dot{x}=&-y-\dfrac{\partial H_{2k}}{\partial y}-yg_{2k-2}+V_{k-1}\, (x^2+y^2)^{k-1}x,\vspace{0.2cm}\\
\dot{y}=&x+\dfrac{\partial H_{2k}}{\partial x}-xg_{2k-2}+V_{k-1}\,
(x^2+y^2)^{k-1}y,
\end{array}
\]

and
\begin{equation}\label{12}
\begin{array}{rl}
H_4=&\kappa_2H^2_2,\quad H_6=\kappa_3H^3_2,\ldots,
H_{2k-2}=\kappa_{k-1}H^{k-1}_2,\vspace{0.2cm}\\
H_{2l+1}=&0,\quad\mbox{for}\quad l=1,2\ldots\vspace{0.2cm}\\
0=&V_1=V_2,\ldots=V_{k-1},\vspace{0.2cm}\\
0=&g_{2k-2}\{H_{2m+1-2k},H_2\}+\{H_{2m+1-2k},H_{2k}\}+\{H_{2m-1},H_2\},\vspace{0.2cm}\\
V_{m-1}(x^2+y^2)^{m-1}=&g_{2k-2}\{H_{2m-2k},H_2\}+\{H_{2m-2k},H_{2k}\}+\{H_{2m-2},H_2\}\vspace{0.2cm}\\
&+(2m-2k)V_k(x^2+y^2)^{k-1}H_{2m-2k},
\end{array}
\end{equation}
for $m\geq k,$ where $g_{2k-2}=g_{2k-2}(x,y),$ is a homogenous
polynomial of  degree $2k-2,$
\end{proposition}
\begin{proof}
From \eqref{9}  if $m=2k-1$ follows that
\[
\begin{array}{rl}
\{H_3,H_2\}=&0,\vspace{0.2cm}\\
\qquad\qquad\vdots\qquad\qquad\qquad\vdots\qquad\qquad\vdots\qquad
&\qquad\vdots\vspace{0.2cm}\\
\{H_{2k-3},H_2\}=&0,\vspace{0.2cm}\\
\{H_{2k-2},H_2\}=&V_{k-2}(x^2+y^2)^{k-1},\vspace{0.2cm}\\
xX_{2k-1}+yY_{2k-1}+\{H_{2k},H_2\}=&V_k(x^2+y^2)^{k},\vspace{0.2cm}\\
\dfrac{\partial H_3}{\partial x}X_{2k-1}+\dfrac{\partial H_3}{\partial y}Y_{2k-1}+\{H_{2k+1},H_2\}=&0,\vspace{0.2cm}\\
\dfrac{\partial H_4}{\partial x}X_{2k-1}+\dfrac{\partial\,H_4}{\partial x}Y_{2k-1}+\{H_{2k+2},H_2\}=&V_{k+1}(x^2+y^2)^{k+1},\vspace{0.2cm}\\
\qquad\qquad\vdots\qquad\qquad\qquad\vdots\qquad\qquad\vdots\qquad
&\qquad\vdots.,\vspace{0.2cm}\\
\dfrac{\partial\,H_{2l+1}}{\partial\,x}X_{2k-1}+\dfrac{\partial\,H_{2l+1}}{\partial\,x}Y_{2k-1}+\{H_{2k+2l-1},H_2\}=&0
 \end{array}
 \]
Thus after some computations we prove the proposition.
\end{proof}

\begin{corollary}
Liapunov constants for the vector field \eqref{9} and \eqref{rr11}
can be computing as follows
\begin{equation}\label{13}
V_{j}=\dfrac{1}{2\pi}\displaystyle\int_{0}^{2\pi}\left.\left(g_{2k-1}\{H_2,H_{2j+3-2k}\}+\{H_{2j+4-2k},H_{2k}\}\right)\right|_{x=\cos
t,\,y=\sin t}dt
\end{equation}
where $j\geq k,$ if $m=2k$ and
\[
V_{j}=\dfrac{1}{2\pi}\displaystyle\int_{0}^{2\pi}\left.\left(g_{2k-2}\{H_{2},H_{j-k+1}\}+
\{H_{2(j-k}+1,H_{2k}\}+2kV_{j-1}H_{2k}\right)\right|_{x=cos
t,\,y=sin t}dt
\]
where $j\geq k$ if $m=2k-1.$
\end{corollary}
\begin{proof}
Follows from \eqref{10} and by \eqref{12}  by considering that
\[\int_{0}^{2\pi}\left.\{f,H_2\}\right|_{x=cos
t,\,y=sin t}dt=0,\] for arbitrary derivable function $f.$
\end{proof}

\begin{example}
We shall consider Bautin  quadratic system (see \cite{Bautin})
\begin{equation}\label{11}
\begin{array}{rl}
\dot{x}=&-y-\lambda_3x^2+(2\lambda_2+\lambda_5)xy+\lambda_6
y^2,\vspace{0.2cm}\\
\dot{y}=&x+\lambda_2x^2+(2\lambda_3+\lambda_4)xy-\lambda_2 y^2,
\end{array}
 \end{equation}
By comparing \eqref{9} and \eqref{11} we obtain that
\[
\begin{array}{rl}
\dfrac{\partial H_3}{\partial
x}=&\left(B+2\lambda_3+\lambda_4\right)yx+(\lambda_2+A)\,x^2-\lambda_4\,y^2,\vspace{0.2cm}\\
\dfrac{\partial H_3}{\partial
y}=&\left(A-2\lambda_2-\lambda_5\right)yx+\lambda_3\,x^2+(B-\lambda_6)y^2,\vspace{0.2cm}\\
g_1=&Ax+By.
\end{array}
 \]
 From the compatibility conditions and after integration we deduce that
 \[\begin{array}{rl}
 g_1=&\lambda_5 x-\lambda_4 y,\vspace{0.2cm}\\
H_3=&\dfrac{1}{3}\left(\lambda_2+\lambda_5\right)x^3+\lambda_3x^2y-\dfrac{1}{3}(\lambda_2+\lambda_6)y^3.
 \end{array}
 \]
The  Liapunov constants in this case we calculated by using formula
 \eqref{13}. It is possible to show that
 \[V_1=\dfrac{1}{2\pi}\displaystyle\int_{0}^{2\pi}\left.g_{1}\{H_2,H_{3}\}\right|_{x=\cos
t,\,y=\sin t}dt   =\dfrac{1}{8}\lambda_5(\lambda_3-\lambda_6).\]
\end{example}
\begin{example}
We shall consider quasi-homogenous cubic system \eqref{r666}. After
some computations we can prove that in this case the function
$H_4,\,g_2$ and the Liapunov constant $V_2$ in this case are
\[
\begin{array}{rl}
H_4=&(V_2-A)x^3y+\dfrac{M}{2}x^2y^2-\dfrac{D}{4}y^4+\dfrac{M+B+K}{4}x^4+(\dfrac{N}{4}-V_2)x^3y\vspace{0.2cm}\\
&+\dfrac{\Lambda}{4}(x^2+y^2)^2,\vspace{0.2cm}\\
g_2=&-(M+B)x^2-(4V_2-3A-L)xy+\Lambda (x^2+y^2),\vspace{0.2cm}\\
V_2=&\dfrac{1}{8}\left(3(A+N)+L+C\right).
\end{array}
 \]
\end{example}


\begin{thebibliography}{99}
\bibitem{Algaba}{\sc A.Algaba and M.Reyes} Characterizattion
isochronous point and computing isochronous sections, {\em J. of
Mathematical Analysis and Applications}, \textbf{2}, (2009),
564--576.
\bibitem{Lloyd}{\sc  M.A. Alwash - N.G. Lloyd,} Non-autonomous equations related to
polynomial two dimensional systems,
 {\em Proc. Roy. Soc. Edinburgh,} \textbf{105} A (1987), 129--152.

\bibitem{Bautin} {\sc N.N. Bautin,} On the number of limit cycles appearing with variation
of the coefficients from an equilibrium state of the type of a focus
or a center, {\em Mat. Sb. (N.S.)}, \textbf{30(72)}:1
(1952),\,181–-196

\bibitem{Basset}, {\sc A.B. Basset}(1901), The Lemniscate of Gerono, An
elementary treatise on cubic and quartic curves, {\em Deighton
Bell,} (1901), 171–172.



\bibitem{Bendix} {\sc I. Bendixson,} Sur les courbes d\'efinies par des ´equations differentielles,
{\em Acta Math.,}  \textbf{24} (1901), 1--88.


\bibitem{Brudnyi} {\sc Alexander Brudnyi,} The
Center Problem For Ordinary Di?erential Equations, Arxiv 0301339v3.

\bibitem{chava} {\sc J. Chavarriga and M. Sabatini},  A survey on
isochronous center, {\em Qualitative theory of dynamical systems},
\textbf{1},(1999), 1--70.


\bibitem{Conti} {\sc R. Conti,} Centers of planar polynomial systems. a
review.{\em Le Matematiche}, Vol. LIII, Fasc. II, (1998), 207--240.

\bibitem{Darboux} {\sc G. Darboux} {\em Bull. Sci. Math. 2me s´erie,} \textbf{2} (1878), 60–96;
123–144; 151–200.


 \bibitem{Devlin} {\sc J. Devlin} Coexisting isochronous centers and
non-isochronous centers, University of Wales, preprint (1994).

\bibitem{DLA} {\sc F. Dumortier, J. Llibre and J.C. Art\'{e}s},
{Qualitative theory of planar differential systems}, Universitext,
Springer, 2006.

\bibitem{From} {\sc  M. Frommer,} Die Integralkurven einer gew¨ohnlichen Di?erential-gleichung
erster Ordnung in der Umgebung rationaler
Unbestimmtheitsstellen,{\em Math. Ann.,} \textbf{99} (1928),
222--272.


\bibitem{Gasull} {\sc A.Gasull, A. Guillarmon and V.Ma\~{n}osas}, Centre
and isochronicity conditions for systems with homogenous
nonlinearities, {\em Proc. of the second Catalan days of applied
mathematics}, 106--116.
\bibitem{Poin} {\sc H. Poincar\'e,} Sur les courbes d´efinies par une \'equation
diferentielle, {\em Oeuvres}, \textbf{1}, Paris, 1892.

\bibitem{Vladimir} {\sc V.V. Ivanov, E. Volokitin,} Uniformly
isochronous polynomial centers, {\em Electronic Journal of
Differential Equations } \textbf{133}, (2005), 1--10.


\bibitem{Llibre1} {\sc J. Llibre},
{Integrability of polynomial differential systems}, Handbook of
Differential Equations, Ordinary Differential Equations, {\em Eds.
A. Ca\~nada, P. Drabek and A. Fonda, Elsevier} (2004), pp. 437--533.

\bibitem{LR0} {\sc J. Llibre and R. Ram\'{\i}rez},
{Inverse problems in ordinary differential equations and
applications}, book, to appear.


\bibitem{Loud} {\sc W.S. Loud}, Behavior of the period of solutions of certain plane
autonomous systems near centers,{\em Contr. Diff. Eqs.}, \textbf{3}
(1964),  21--36.


 \bibitem{Liapunov} {\sc A. M. Lyapunov,} A. M. Lyapunov, The general problem of the stability of motion,
Translated by A. T. Fuller from ´Edouard Daraux's French translation
(1907) of the 1892 Russian original, Internat. J. Control 55 (1992),
no. 3., 521--790.


\bibitem{Plesh} {\sc  I.I. Pleshkan and  K.S. Sibirskii,} Isochronism of systems with quadratic nonlinearity,
 (Russian), {\em Mat. Issled. Kishinev YI,} \textbf{4}  (1971), 140--154,
(in Russian).


\bibitem{Plesh1} {\sc  I.I. Pleshkan  }, A new method of investigating the
isochronism of a system of two differential equations,{\em  Diff.
Uravn.}, \textbf{5} (1969), 1083--1090.

\bibitem {Po} {\sc H. Poincar\'e,}
{Sur l'int\'egration des \'equations diff\'erentielles du premier
ordre et du premier degr\'e I and II}, {\em Rendiconti del circolo
matematico di Palermo} {\bf 5} (1891), 161--191; {\bf 11} (1897),
193--239.

\bibitem{Salih1} {\sc  R. Pons and N.Salih}  Center conditions for a
lopsided quintic polynomial vector field, {\em Qualitative theory of
dynamical systems},\textbf{ 3} (2002) 331-343.

\bibitem{Stepanov} {\sc  V.V. Nemytskii - V.V Stepanov,} Qualitative
Theory of Differential Equations, {\em Princeton Univ. Press,} (1960).

\bibitem{Malkin} {\sc I.G. Malkin,} Criteria for center of a differential equation, (Russian),
{\em Volg. Matem. Sbornik,} \textbf{ 2} (1964), pp. 87--91.


\bibitem{Malkin1} {\sc I.G. Malkin,} Stability theory of movements,
{\em  Ed. Nauka,} Moscow, (1966) (in Russian).


\bibitem{Mardesic} {\sc P. Marde\^{s}i\'{ c},  C. Rousseau  and B. Toni,} Linearization of
isochronous center, {\em J. Diff. Eqs.}, \textbf{121} (1995),
67--108.


\bibitem{RS1} {\sc Rafael Ram\'{i}rez and Natalia Sadovskaia,} On
the problem of isochronous center, preprint Universitat Rovira i
Virgili, (1996).

\bibitem{R1} {\sc R. Ram\'{\i}rez  and N. Sadovskaia},Inverse
problem in the theory of planar vector fields, {\em Revista
Iberoamericana de Matem\'{a}ticas}, \textbf{14} (1998), 481--517.

\bibitem{RS22} {\sc R. Ram\'{\i}rez  and N. Sadovskaia},
Polynomial planar vector field with two algebraic invariant curves,
preprint Universitat Polit\'ecnica de Catalu\~{n}a (2004), 1--38.


\bibitem{RS2} {\sc R. Ram\'{\i}rez  and N. Sadovskaia}, On the
construction of vector fields in space, Preprint Universitat
Polit\`{e}cnica de Catalunya, April 1994.

\bibitem{Saharnikov} {\sc N.A. Saharnikov,} Solution of the center focus problem in one case,
{\em Prikl. Mat. Meh.,} \textbf{14} (1950), 651--658.


\bibitem{Salih} {\sc N.Salih and R. Pons}  Center conditions for a
lopsided quartic polynomial vector field, {\em Bull. Sci.math.}
(2002), 369--378. \textbf{126}, 369--378.

\bibitem{Schlomiuk} {\sc  D.Schlomiuk, J.Guckenheimer and R. Rand,} Integrability of plane quadratic
 vector fields,{\em  Expo. Math.,} \textbf{8} (1990),  3--25.

\bibitem{Sibirskii} {\sc K.S. Sibirskii,}  Method of invariants in the qualitative theory of
differential equations, (Russian),{\em Acad. Sci. Moldavian SSR,}
Kishinev, (1968).

\bibitem{soto} {\sc J. Sotomayor}, Curves definidas por
equa\c{c}\'{o}es diferenciais no plano, {\em IMPA}, 1979.

\bibitem{Zoladek} {\sc H. Zoladek,} The solution of the center-focus
problem, {\em Institute of Mathematics, University of Warsaw,}
Poland, (1992),  preprint 80 pp.


\end{thebibliography}
\end{document}